\documentclass[12pt,english]{article}
\usepackage{graphicx}
\usepackage[font=small,labelfont=bf]{caption}
\usepackage[T1]{fontenc} 
\usepackage[utf8]{inputenc}
\usepackage[left=2cm,right=2cm,top=2cm,bottom=2cm]{geometry}
\usepackage{amsthm}
\usepackage{amsmath}
\usepackage{amssymb}
\usepackage{amscd}
\usepackage{amsfonts}
\usepackage{mathtools}
\usepackage{xfrac}
\usepackage{graphics}
\usepackage{enumitem}
\usepackage{tikz-cd}
	\tikzset{
		labl/.style={anchor=south, rotate=90, inner sep=.5mm}
	}
	\tikzset{
		close/.style={outer sep=-2pt}}
\usepackage{mathrsfs}
\usepackage{color}
\usepackage{lipsum}
\usepackage{microtype}
\usepackage{stmaryrd}
\usepackage[all]{xy}
\usepackage[makeroom]{cancel}
\usepackage{tocloft}
\usepackage{bm}
\usepackage{marginnote}
\usepackage{faktor}
\usepackage{nicefrac}
\usepackage{glossaries}
\usepackage{multicol}
\usepackage{stmaryrd}
\usepackage{setspace}
\usepackage{indentfirst}
\usepackage{parskip}
\usepackage{xpatch}

\usepackage{hyperref}

\makeatletter
\def\thm@space@setup{%
	\thm@preskip=2pt plus 1pt minus 2pt
	\thm@postskip=\thm@preskip 
}
\makeatother

\makeatletter
\xpatchcmd{\proof}{\topsep6\p@\@plus6\p@\relax}{}{}{}
\makeatother

\swapnumbers
\theoremstyle{definition}
\newtheorem{theorem}{Theorem}[section]

\theoremstyle{definition}
\newtheorem{caseA}[theorem]{Case(A)}
\newtheorem{caseB}[theorem]{Case(B)}
\newtheorem{warning}[theorem]{Warning}
\newtheorem{alternative}[theorem]{Alternative}

\newtheorem{corollary}[theorem]{Corollary}
\newtheorem{lemma}[theorem]{Lemma}
\newtheorem{proposition}[theorem]{Proposition}
\newtheorem{definition}[theorem]{Definition}
\newtheorem{remark}[theorem]{Remark}
\newtheorem{fact}[theorem]{Fact}
\newtheorem{claim}[theorem]{Claim}
\newtheorem{setup}[theorem]{Set Up}

\newtheorem{step}[theorem]{Step}

\newtheorem{hypothesis}[theorem]{Hypothesis}
\newtheorem{construction}[theorem]{Construction}

\newtheorem{variant}[theorem]{Variant}

\newtheorem{notation/defn}[theorem]{Notation/Definition}
\newtheorem{observation/defn}[theorem]{Observation/Definition}
\newtheorem{setup/defn}[theorem]{Set Up/Definition}
\newtheorem{setup/notation}[theorem]{Set Up/Notation}
\newtheorem{definition/fact}[theorem]{Definition/Fact}

\newtheorem{summary/defn}[theorem]{Summary/Definition}

\newtheorem{fact/prop}[theorem]{Fact/Proposition}
\newtheorem{fact/defn}[theorem]{Fact/Definition}

\newtheorem{s-induction}[theorem]{Sub-Induction}
\newtheorem{sub-claim}[theorem]{Sub-claim}

\newtheorem{start indu}[theorem]{Start of the Induction}
\newtheorem{indu hypo}[theorem]{Inductive Hypothesis}

\newtheorem{indu/def}[theorem]{Induction Defining 
	$F_s^\bullet$ from $F_{s-1}^\bullet$}
\newtheorem{conclusion}[theorem]{Conclusion}
\newtheorem{partial finish}[theorem]{Partial Finish}

\newcommand\restr[2]{{
		\kern-\nulldelimiterspace 
		#1 
		\vphantom{|} 
		\,\big|_{#2} 
	}}

\newcommand{\quoziente}[2]{
	{\raisebox{.1em}{$#1$}\left/\raisebox{-.3em}{$#2$}\right.}
}
\newcommand\quotient[2]{
	\mathchoice
	{ \displaystyle
		{\raise1pt\hbox{$#1$}
			\big/\lower1pt\hbox{$#2$}
		}
	}
	{ \textstyle
		#1/#2}
	{ \scriptstyle
		#1/#2
	}
	{ \scriptscriptstyle  
	\text{{$#1$}$\diagup$\lower3pt\hbox{$#2$}
	}
	}
}

\newcommand{\ie}{\emph{i.e. }}
\newcommand{\eg}{\emph{e.g. }}
\newcommand{\op}{\emph{op.cit. }}
\newcommand{\cfr}{\emph{cf. }}
\newcommand{\wloss}{without loss of generality }
\renewcommand{\th}{\text{th}}
\newcommand{\wti}[1]{\widetilde{#1}}
\newcommand{\what}[1]{\widehat{#1}}
\newcommand{\zar}{\text{Zar}}
\newcommand{\et}{\text{ét}}
\newcommand{\usc}{u.s.c.$\,$}
\newcommand{\lqs}{\leqslant}
\newcommand{\gqs}{\geqslant}

\newcommand{\Z}{\mathbb{Z}}
\newcommand{\Q}{\mathbb{Q}}

\newcommand{\C}{\mathbb{C}}
\newcommand{\Zz}{\Q_{\gqs 0}}

\newcommand{\LbV}{\LL_{-b}(V_{d})}
\newcommand{\LB}{H}

\newcommand{\OO}{\mathscr{O}}
\newcommand{\II}{\mathscr{I}}
\newcommand{\hO}{\what{\OO}}

\newcommand{\JJ}{J}
\newcommand{\DIF}{\mathrm{diff}}

\newcommand{\mm}{\mathfrak{m}}
\newcommand{\pri}{\mathfrak{p}}
\newcommand{\MM}{\quotient{\mm}{\mm^2}}

\newcommand{\G}{\mathbb{G}}
\renewcommand{\v}{\mathrm{V}}
\newcommand{\UU}{\mathfrak{U}}
\newcommand{\VV}{{V}}
\newcommand{\A}{\mathbb{A}}
\newcommand{\AAk}{\A_{k}^{N+1}}
\newcommand{\AL}{A_k}
\newcommand{\minusO}{\setminus \{{\underline{0}\}}}
\newcommand{\PP}{\mathbb{P}}
\renewcommand{\P}{\PP_{k}(\ua)}
\newcommand{\rP}{\mathrm{P}_k}
\newcommand{\sP}{\PP_k(\ua')}
\newcommand{\RP}{\mathrm{P}}
\newcommand{\QQ}{\mathcal{Q}}
\newcommand{\Dl}{\Delta}
\newcommand{\HH}{\mathrm{H}}
\newcommand{\hh}{\mathrm{h}}
\newcommand{\xx}{x}
\newcommand{\yy}{y}
\newcommand{\XX}{\mathcal{X}}
\newcommand{\ZZ}{Z}
\newcommand{\YY}{\mathcal{Y}}
\newcommand{\V}{V}
\newcommand{\CU}{\mathcal{U}}
\newcommand{\cg}{\mathscr{I}}
\newcommand{\CP}{\mathscr{P}}
\newcommand{\GP}{\mathfrak{P}}
\newcommand{\exe}{\mathscr{E}}
\newcommand{\cexe}{\mathcal{E}}
\newcommand{\buA}{\mathcal{A}}
\newcommand{\buW}{\mathcal{W}}
\newcommand{\cC}{\mathcal{C}}
\newcommand{\x}{\alpha}
\newcommand{\y}{\beta}
\newcommand{\ff}{f}
\newcommand{\FF}{\varDelta}
\newcommand{\GG}{G}
\newcommand{\si}{\sigma}
\newcommand{\ux}{\underline{x}}
\newcommand{\uX}{\underline{X}}
\newcommand{\uz}{\underline{z}}
\newcommand{\uxi}{\underline{\xi}}
\newcommand{\ueps}{\underline{\epsilon}}
\newcommand{\ue}{\underline{\varepsilon}}

\newcommand{\uEE}{\underline{G}}

\newcommand{\uF}{\underline{F}}
\newcommand{\uf}{\underline{f}}

\newcommand{\ueitop}{\ue_{i}^{\text{top}}}
\newcommand{\uettop}{\ue_{t}^{\text{top}}}
\newcommand{\ur}{\underline{r}}
\newcommand{\ub}{\underline{\ab}}
\newcommand{\ab}{a}

\newcommand{\invUB}{\INV_{\quotient{\UU}{\BB}}}
\newcommand{\inv}{\mathrm{inv}}
\newcommand{\INV}{\textsc{inv}}
\newcommand{\ui}{\underline{i}}
\newcommand{\uq}{\underline{q}}
\renewcommand{\l}{\lambda}
\newcommand{\BB}{B}
\newcommand{\bb}{b}
\newcommand{\ua}{\underline{a}}
\newcommand{\ug}{\underline{g}}
\newcommand{\wt}{\mathbf{wt}}
\newcommand{\wtx}{\mathbf{wt}_{X}}
\newcommand{\wtxi}{\mathbf{wt}_{\xi}}
\newcommand{\ssym}{\mathrm{\underline{S} y\underline{m}}}
\newcommand{\gr}{\mathrm{gr}}
\newcommand{\defc}{\epsilon}
\newcommand{\T}{\mathrm{T}}
\newcommand{\LL}{\mathrm{L}}

\newcommand{\De}{\mathrm{D}}
\newcommand{\pr}{\mathrm{pr}}
\newcommand{\id}{\mathrm{id}}
\newcommand{\mult}{\mathrm{mult}}
\newcommand{\spec}{\mathrm{Spec}\,}

\newcommand{\abs}[1]{|#1|}

\newcommand{\Gl}{\mathrm{GL}}
\newcommand{\Sym}{\mathrm{Sym}}

		\title{{Very fast, very functorial, and very easy} \\ { resolution of singularities} }
		\author{Michael McQuillan \\ \small{with the collaboration of} \\Gianluca Marzo}

\usepackage{titlesec}
\titlespacing\section{0pt}{8pt plus 1pt minus 6pt}{0pt plus 2pt minus 2pt}

\begin{document}
\setlength{\parskip}{3pt}
\setlength{\parindent}{0pt}
\setlength{\abovedisplayskip}{1pt plus 2pt minus 1pt}
\setlength{\belowdisplayskip}{2pt plus 2pt minus 1pt}
\setlength\cftsecnumwidth{30pt}
	\maketitle
	\begin{abstract}
	\noindent The main theorem, \ref{T1}, is the existence for excellent Deligne-Mumford champ of characteristic zero of a resolution functor independent of the resolution process itself. Perceived wisdom was that this was impossible, but the counterexamples overlooked the possibility of using weighted blow ups. The fundamental local calculations take place in complete local rings, and are elementary in nature, while being self contained and wholly independent of Hironaka's methods and all derivatives thereof, \ie existing technology. Nevertheless Abramovich, Temkin, and Wlodarczyk, \cite{DanTW}, have varied existing technology to obtain an even shorter proof of principalisation, \ref{DanT1}, in the geometric case. Excellent patching is more technical than varieties over a field, and whence easier geometric arguments are pointed out when they exist.
\end{abstract}

\section{Introduction}
It is a known fact that resolution of singularities, already in characteristic zero, cannot be achieved in a way that is both étale local and independent of the resolution process itself while blowing up in smooth centres. More precisely one would like in the category of reduced 
excellent algebraic spaces (all Henselian local rings excellent and some, whence any, scheme like cover Noetherian J2, or equivalently admitting an excellent atlas.
In particular an excellent algebraic space in this sense which is also a scheme is only quasi-excellent in standard parlance.
Nevertheless since the catenary condition is close to meaningless for algebraic spaces - it's only interpretation is that every global irreducible component is everywhere étale locally equidimensional, \cite[7.8.4 (iii)]{egaIV1}, and we won't use this - it will be systematically eschewed globally while étale locally it is tautologically true),
a modification functor 
\begin{tikzcd}[cramped, sep=small]
U \ar[r, mapsto] &M(U)
\end{tikzcd}
and an invariant $\inv(U) \in \Gamma_{\gqs 0}$ in a (preferably discrete) ordered group such that\\
	(M.1 bad) 
	\begin{tikzcd}[cramped, sep=small]
	M(U) \ar[r] &U
	\end{tikzcd}
	is a blow up in a smooth centre. \\
	(M.2) 
	$U = M(U)$ iff $U$ is regular.\\
	(M.3) 
	$M$ commutes with étale base change 
	\begin{tikzcd}[cramped, sep=small]
	U' \ar[r] &U,
	\end{tikzcd}
	\ie	$M(U') = M(U) \times_{U} U'$
	whenever $U, \; U'$ are connected and $\inv(U') = \inv(U)$.\\
	(M.4) 
	For any
	\begin{tikzcd}[cramped, sep=small]
	U' \ar[r] 	&U
	\end{tikzcd}
	étale, $\inv(U') \leq \inv(U)$.\\
	(M.5) 
	$\inv\bigl(M(U) \bigr) < \inv( U )$.

The impossibility of this is shown by the following example, \cfr \cite[pg. 142]{kolres},
\begin{equation}\label{G1}
	\begin{tikzcd}
	U: x^2 + y^2 + (zt)^2 = 0 \ar[r, hook] 
		&\A_K^4
	\end{tikzcd}
\end{equation}
wherein the singular locus is the union of the two lines,
\[
L_1: x=y=z=0 \;\text{ \& }\; L_2: x=y=t=0 .
\]
On the other hand if 
\begin{tikzcd}[cramped, sep=small]
	M(U) 	\ar[r] 	&U
\end{tikzcd} 
were to exist then by (M.1 bad), (M.2) \& (M.3) it must be a blow up in a smooth centre contained in the singular locus, so the only possibilities are $L_1, L_2$ or their intersection, \ie the origin. Now the latter operation leaves \eqref{G1} unchanged where the proper transform of either line meets the exceptional divisor, while a choice amongst $L_1,\,L_2$ is inadmissible because the process must respect, (M.3), the symmetry 
$z \xleftrightarrow{\;\;\;\;} t$, and that's without even addressing the issue that \eqref{G1} might only be valid after Henselisation, so that globally the $L_i$ could be branches of the same curve.

The traditional get out from this difficulty is to change the problem, \eg the argument of the modification functor becomes not just varieties but varieties with marked divisor, so, in particular, blowing up \eqref{G1} in the origin creates a marked divisor and amongst the new singular lines one of them is marked. 
The point of view of this article is, however, to change the paradigm and adapt the modification to the problem, so that (M.1 bad) is replaced by,

	(M.1 new)
	\begin{tikzcd}[cramped, sep=small]
	M(U) \ar[r] &U
	\end{tikzcd}
	is a smoothed weighted blow up in a regular centre.

This operation is defined in \cite[I.iv.3]{mp1}, and will not be repeated here. It should, however, be noted that $M(U)$ is by definition, \op, a (Deligne-Mumford) champ, albeit if we were to work with varieties over $\C$ the $2$-category of orbifolds would be adequate for our current purposes, and in any case the $2$-category of champs/orbifolds is just a categorical subterfuge which allows us to work with quotient singularities while doing linear algebra. With this in mind, the paradigm shift works, \ie

\begin{theorem}\label{T1}[\ref{WP2 prop}] 
	In the $2$-category of reduced excellent Deligne-Mumford champ 
	(defined exactly as above for spaces so \emph{inter alia} with no separation condition) 
	there is a modification functor 
	\begin{tikzcd}[cramped, sep=small]
	U \ar[r] &M(U),
	\end{tikzcd} 
	\ref{ESD1 sum/d}, satisfying (M.1 new), (M.2), (M.3), (M.4), (M.5), albeit $\inv$ takes values in 
	$
	\Q_{\gqs 0}^{\infty} = \varinjlim \Q_{\gqs 0}^{N}.
	$
	Nevertheless, the invariant has self-bounding denominators, \ref{MD1 defn}.
\end{theorem}

Here
self bounding denominators is a certain technical condition, \ref{MD1 defn}, which has all the effects, \ref{MF1 fact}, of defining the invariant in $\Z_{\gqs 0}$ while allowing us to define the invariant and perform various construction, \eg \ref{s-induction}, where they naturally occur, \ie $\Q_{\gqs 0}$.
More substantially \ref{T1} is the global manifestation of some much more basic local algebra. Specifically for $I$ an ideal of a $m$-dimensional regular local ring, $A$, of characteristic zero, with maximal ideal $\mm$ we construct an invariant, \S \ref{sec II INV}, 
$\inv_{A}(I) $ with self bounding denominators in $\Q_{\gqs 0}^{2m}$ ordered lexicographically. 

Better there is a yoga for constructing $\inv$ that makes the resolution process more widely applicable to more difficult problems such as vector field singularities, which, essentially views the resolution process as a diagram chase, and manifest itself as follows,

(Y.1) Generically most thing are regular, a.k.a. $I=\mathcal{O}$, so $\inv=\underline{0}$ and there is nothing to do.

(Y.2) If (Y.1) didn't happen then generically most things have an isolated singularity at the closed point, and after a single blow up in the same the multiplicity will decrease,

(Y.3) If (Y.2) didn't happen then there is proper sub-space of the tangent space where the multiplicity did not decrease and its annihilator in $\MM$ gives us the start of a filtration of $A$ which depends only on $I$.

(Y.4) Construct inductively, \ref{start INDUCT} - \ref{indu hypo}, a sequence of filtrations, $F_s^{\bullet}(I)$, according to the dichotomy,

\begin{caseA}
	Something generic happens, case (A), \ref{case A}, then $ s \xmapsto{\;\;\;} s+1$;
\end{caseA}

\begin{caseB}
	Nothing generic happens, case (B), \ref{case B}, then at worst, $F^{\bullet}_{s}(I)$ converges $\mm$-adically.
\end{caseB}

Proceeding in this way leads to the key,
\begin{fact}\label{FA1}[\cfr V.b]
	There is an invariant, $\inv_{A}(I) \in \Q_{\gqs 0}^{2m}$, of regular $m$- dimensional characteristic zero local rings and their ideals with self bounding denominators such that if $\mathfrak{U}$ is the completion of its spectrum at the closed point, then there is a smoothed weighted blow up $\rho: \wti{\mathfrak{U}}
	\xrightarrow{\;\;\;\;} \mathfrak{U}$ such that at every closed point of $\mathfrak{U}$ the invariant strictly decreases provided $I \neq A$.
\end{fact}
Rather plainly at this point the only remaining issue is whether the weighted centre defining $\rho$ is well defined in $A$, or even just its strict Henselisation $A^{\hh}$, rather than its completion, $\what{A}$.
It is, however, a genuine issue since
both \cite[7.9.3]{egaIV1} and its proof are valid exactly as stated even on allowing resolutions either by algebraic spaces or Deligne-Mumford champs, \ie excellent Henselian local rings (which is the same as quasi-excellent and Henselian) are a necessary condition for resolution of singularities. It is therefore pleasing to observe that (quasi-) excellence is just what's needed to establish 	
\begin{proposition}\label{P1}[\cfr \ref{WA4 alternative}]
	If the centre in \ref{FA1} is of dimension 0 or, $A$ is an excellent regular local ring then, \ref{EE1 fact}, the (canonically defined) smoothed weighted blow up of \ref{WF3 fact} is the completion in the exceptional divisor of a smoothed weighted blow up of $\spec A$.
	Similarly if $A$ is an excellent reduced local ring, $\mathfrak{V}$ its completion in the closed point, and 
	$\rho: \wti{\mathfrak{V}}\rightarrow\mathfrak{V}$
	the proper transform of $\mathfrak{V}$ along \ref{FA1} after a choice of an embedding of $\mathfrak{V}\hookrightarrow \UU$ in a smooth formal scheme, \ref{WFD4}, then there is a smoothed weighted blow up, \ref{EEF3 fact}, of $\spec A$ whose completion in the exceptional divisor is $\rho$.
\end{proposition}
Needless to say convergence is (much) easier when everything is of finite type over a field, and whence alternatives \ref{WA4 alternative}, resp. \ref{EEA2 alternative}, are offered to the more general \ref{EE1 fact}, resp. \ref{EEF3 fact}. Similarly to go from convergence to \ref{T1} one needs the upper semi-continuity of the invariant which is an attractive consequence, \ref{EEA1}, of the properties peculiar to the diagonal in the geometric case. Otherwise, \ref{EEF2 fact} \& \ref{EEA1}, this adopts Dade's proof of the \usc of the multiplicity in his un-published 1960 Princeton thesis (of which Villamayor's summary, \cite[6.1.3]{vill1}, was invaluable) and leads to the wholly natural  intervention of the (global) J-2 condition. It is also important not to  lose sight of the wood for trees, and 
in particular the critical principalisation statement
which in the geometric case has been obtained simultaneously, \cite{DanTW},
by Abramovich, Temkin, and Wlodarczyk, 

\begin{theorem}\label{DanT1}[\ref{WP1 prop}] 
There is a modification functor from the 2-category whose objects,
$(U,\mathcal{I})$, are ideals on regular excellent Deligne-Mumford
champs whose value 
\begin{equation}\label{dan1}
M_{(U,\mathcal{I})}\,=\, (\widetilde{U}, \widetilde{\mathcal{I}})
\end{equation}
is the proper (rather than total) transform 
$\wti{\mathcal{I}}$
on a smoothed weighted blow up 
$\wti{U}\rightarrow U$, 
satisfying (in the obvious change of notation) (M.1 new), (M.2),  (M.4), (M.5), 
while (M.3) becomes,
$M_{(U,\mathcal{I})}=0$ iff $\mathcal{I}=\mathcal{O}_U$,
and, again,
$\inv$ takes values in $\Q_{\gqs 0}^{\infty} = \varinjlim \Q_{\gqs 0}^{N}$
with  self-bounding denominators.
\end{theorem}

which is equivalent to the more pleasing assertion that there is a fully étale
local modification functor, \ref{WFD3 f/d}, by smoothed weighted 
blow ups which resolves any rational map. 
In any case, the paper is organised as follows,

	\textbf{\S II.} This contains some linear algebra about weighted projective spaces (technically champs because we want them to be regular) which describes the manifestation of item (Y.3) above in the generality necessary for the distinctions between generic and non-generic phenomena in item (Y.4).
	
	\textbf{\S III.} 
	This is the inductive definition of the invariant as outlined in (Y.1)-(Y.4). The key step is the sub-induction \ref{s-induction} whose illustration by way of its Newton polyhedron, figure \ref{fig:NP} of page \pageref{fig:NP}, should facilitate its understanding.
	
	\textbf{\S IV.}
	Calculates the invariant for ideals on weighted projective champs. It is the proof that the invariant goes down on blowing up in its weighted centre.
	
	\textbf{\S V.}
	This begins to address  the aforesaid convergence issues, and related questions such as upper semi-continuity of the invariant by calculating it in a suitably general, \ref{WD1 obs/defn}, relative setting. 
	It does convergence and \usc out of the box in the geometric case of $A/K$ essentially of finite type over a field of characteristic zero on completing
	2 copies of $\spec A$ in the diagonal, while more generally, \cfr \ref{WC2 constru}, systematically working with formal champs sidesteps thorny issues like the diagonal is an embedding iff the champ is a (separated) algebraic space.
	
	\textbf{\S VI.}  
	Is the final assembly of the preceeding into a modification functor, \ref{WP1 prop}, for the (weak) principalisation (a.k.a. resolution of rational maps \ref{EER1 rmk}) of ideals on excellent regular champs. Unlike the preceeding sections it assumes
	a working familiarity with the rudiments of algebraic champs and is much less elementary than \S \ref{sec I WPS} - \ref{sec III INV} wherein any intervention of champs does not go much beyond linear algebra of graded rings.
	
	\textbf{\S VII.}
	Pushes things into a resolution functor, \ref{ESD1 sum/d}, for excellent champs. The geometric case is easy \ref{EEA1} \& \ref{EEA2 alternative} for a geometric reason, \cfr the summary of \S V above, and otherwise it's an exercise in appreciating Grothendieck's excellent definition.

Talks about the paper have been given at U.C.S.D., N.Y.U., and Imperial, but, the one that really generated interest was given at Valencia (in homage to the university's founder and his nephew, and in no way related to the Celtic game) in February, 2019, from which news came to Oberwolfach the subsequent week. Amongst the participants there, Dan Abramovich, Michael Temkin and Jaroslaw Wlodarczyk provided demonstrable proof that they were writing, and have now written, \cite{DanTW}, an algorithm satisfying (M.1 new), (M.2), (M.3) \& (M.4).
Similarly, credit must also go to Daniel Panazzolo who although he did not participate in the preparation of this manuscript introduced the majority of the key ideas in \cite{PP}. Indeed, the only one he was missing was the functoriality yoga, \cfr (Y.1)-(Y.4), which first appeared in \cite{mp1}.

\section{Weighted Projective Champs}\label{sec I WPS}	
\begin{setup/defn}\label{setup WPS}	
	Throughout this section, $k$ is a ring of characteristic $0$, and $\AL := \AAk \minusO $. For $n \lqs N$, let 
	$\ua = (\ua_{0},\ua_{1},...,\ua_{n}) 
	\in \mathbb{Z}_{>0}^{N+1}$ with each 
	$\ua_{i} = \bigl(a_{i},..., a_{i} \bigr) \in \mathbb{Z}_{>0}^{c_{i}}$,	$c_i \gqs 1$  and 
	$N + 1 = c_0 +...+ c_n$. We denote the coordinates of $\AAk$ by $\xx_{ij}$  for 
	$ 0 \lqs i \lqs n$ and 
	$1 \lqs j \lqs c_i$, and we will call the set of variables with the same weight $a_i$, \ie 
	$\{\xx_{i\,1},...,\xx_{i\,c_i}\}$, a \emph{block}, or a 
	\emph{block of weight $a_i$}, and often abbreviate it by $X_i$, similarly, consistent with this decomposition, we will abbreviate monomials 
	$\prod \xx_{ij}^{e_{ij}}$ by $X_{i}^{E_i}$, where $\abs{E_i} = \sum_{j} e_{ij}$ (\ie the degree of the monomial in the relevant block);
	while $X_i = 0$ means $x_{ij}=0, \; \forall \, 1 \lqs j \lqs c_i$.
\end{setup/defn}

\begin{definition}\label{def wps}
	The weighted projective champ
	$\P:=\PP(\ua_{0},\ua_{1},...,\ua_{n})$
	 is defined to be the classifying champ $\left[\quotient{A_k}{\G_{m}}\right]$ of the action 
	\begin{equation}\label{eq WPS groupoid}
	\begin{tikzcd}[row sep=3pt, column sep=18pt]
				\mathbb{G}_m \times \AL
				\ar[r, shift left, "\lambda"]
				\ar[r, shift right, "\id"']
					&\AL ,\;
		(\lambda^{a_0}X_0,...,\lambda^{a_n}X_n) 		
			&(X_0,...,X_n)
			\ar[r, mapsto, "\id"]
			\ar[l, mapsto, "\lambda"']
				&(X_0,...,X_n)	
	\end{tikzcd}
\end{equation}

	 on which the tautological bundle $\OO_{\P}(1)$ corresponds to the character:
	\begin{equation}\label{eq WPS groupoid.1}
	\begin{tikzcd}
	\G_{m} \ar[r]
	&\G_{m} : \lambda \ar[r, mapsto]
	& \lambda^{-1}.
	\end{tikzcd}
	\end{equation}
	In particular, functions on $\AAk$ are naturally graded by the action, and we denote the grading of a $\G_{m}$-homogeneous equivariant function by $\wt$, \ie
	\begin{equation}
		\wt(X_i) = a_i, \quad \wt(X_i^{E_i}) = a_i \abs{E_i}
	\end{equation}
\end{definition}
Finally if $\ur = \bigl( \ur_0, ..., \ur_n \bigr) \in \Q_{> 0}^{N+1}$, $\ur_i := (r_i, ..., r_i ) \in \Q_{>0}^{c_i}$, and $\ua \in \Z_{>0}^{N+1}$ is the unique integer tuple parallel to $\ur$ without common factors we define
\begin{equation}\label{M1 eq}
	\rP (\ur) := \PP_k(\ua)
\end{equation}
to which we add the hypothesis specific to our situation \ie 
\begin{hypothesis}\label{HP}  
	Suppose $a_0 < a_1 < ... < a_n$ and let $\V_d$ be a $k$-submodule of $\HH^0\bigl(\P, \OO_{\P}(d)\bigr),$ $ d \gqs 0$, such that if 
	$\sP = \PP_k(\ua_{1},...,\ua_{n})$ 
	is the sub-weighted projective champ defined by the 
	block of variables $X_0 = 0$ of weight $a_0$ and $\V'_d$ is 
	the image of $\V_d$ in 
	$\HH^0\bigl(\sP, \OO_{\sP}(d)\bigr)$ then for 
	all quotients 
	\begin{tikzcd}[cramped, sep=small]
		k \ar[r, two heads]	&k',
	\end{tikzcd}
	 $-b <0$ and $\partial \in 
	\mathrm{H}^0\bigl( 
					\PP_{k'}(\ua'), \T_{\PP_{k'}(\ua')}(-b) 
				\bigr)$
	\begin{equation}\label{HP eq}
	\partial (\ff') = 0, \; \forall \, \ff' \in \V'_d \otimes_{k} k'
			\, \Longleftrightarrow \,
	\partial = 0.
	\end{equation}
\end{hypothesis}
In the presence of such a supposition we have,

\begin{lemma}\label{lem L} 
	Let be everything as in \ref{setup WPS}-\ref{HP}, and for $-b < 0$ a strictly negative integer define
	\begin{equation}\label{CBF1 eq}
		\LbV :=	\bigl\lbrace \, 
					\partial \in \HH^{0} 	\bigl(
											\P, \, \T_{\P}(-b) 
											\bigr) 
					\; \big| \; \partial(\V_d)= 0 \, 
				\bigr\rbrace 
	\end{equation}
	the sub-module of global weighted vector fields of weight $-b$ which vanish on $\V_d$. Then
	If $b \neq a_0$, $\LbV=0$, otherwise there is a natural injective map,
	\begin{equation}\label{eq LL}
		\begin{tikzcd}[column sep=23pt]
			\LL_{-a_0}(V_d) 	\ar[r, hook] 
					&\HH^{0}	\bigl(
								 \P,\OO_{\P}(\ua_{0})
								 \bigr) ^\vee 
					:=\HH^{0}	\bigl( 
								\P, \OO_{\P}(a_0)^{\oplus c_0}
								\bigr)^\vee.
			\end{tikzcd}
		\end{equation}
	Better still if for every quotient 
	\begin{tikzcd}[cramped, sep=small]
	k \ar[r, two heads]	&k',
	\end{tikzcd}
	\begin{equation}\label{CBDF2 eq}
		\LL_{-a_0} \otimes_{k} k' = 
				\left\lbrace \,
					\partial \in \HH^{0} 
									\bigl(
										\PP_{k'}(\ua), \, \T_{\PP_{k'}(\ua) (-b)}
									\bigr) 
					\, \big| \,
					\partial	\bigl(
									V_d \otimes_{k} k' 
								\bigr)	=0
					\, \right\rbrace
	\end{equation}
	then \eqref{eq LL} remains an injection on tensoring with $k'$.
\end{lemma}

\begin{proof}
	Without loss of generality $\dim \P >0$, so from the Euler  Sequence,
	\begin{equation}
		\begin{tikzcd}[column sep=0.5cm]
		0 
		\ar[r] 
				&\OO_{\P} \ar[r] 
						&\coprod_{i=0}^n \OO_{\P}(\ua_i) 
							\bigl(
								:= \OO_{\P}(a_i)^{\oplus c_i} 
							\bigr)
						\ar[r] 
								&\T_{\P}  
								\ar[r] 
										&0,
		\end{tikzcd}
	\end{equation}
	tensored by $\OO_{\P}(-b)$ there is an isomorphism in co-homology
	\begin{equation}\label{eq SS}
	\begin{tikzcd}
	 \coprod_{i=0}^{n} 
		\HH^0	\bigl( 
					\P, \OO_{\P}(\underline{a}_i -b)
				\bigr) 
		\ar[r, "\sim"] 
			&\HH^0	\bigl( 
						\P, \T_{\P}(-b)
					\bigl)
	\end{tikzcd}
	\end{equation}
	unless $\dim \P = 1$ and $b = a_0 + a_1$. Indeed this is trivial for $\dim \P > 1$ by the analogue of Serre's explicit calculation for weighted projective champ, \cite[I.c.3]{mcqSS}, while by, \op, for $\dim \P = 1$, $\OO_{\P}(-b)$ has non-trivial $\hh^1$ if and only if $b \gqs a_0 + a_1$ and 
	\begin{tikzcd}[cramped, sep=small]
		\T_{\P} 	\ar[r, "\sim"] 	&\OO_{\P}(a_0 + a_1),
	\end{tikzcd}
	so again \eqref{eq SS} follows unless $b = a_0 + a_1$. 
	To avoid this fastidious exception observe that it only occurs if $ \dim \P= 1$,  $\P= \PP(a_0,a_1)$ is defined by blocks $X_0, X_1$ of rank $1$ and, by hypothesis, weights $a_0 < a_1$. As such if $\De$ were an element of $\LL_{-a_{1} -a_{0}}(V_d)$ then $X_{0} \De \in \LL_{-a_{1}}(V_d)$, and $a_1 \neq a_0$, so \ref{lem L} for $b=a_{1}$ implies the same for $b= a_1 + a_0$. Thus without loss of generality $b \neq a_1 + a_0$ if $\dim \P = 1$.	
	
	Now suppose $b \neq 0 $ and 
	$\HH^{0}\bigl(
				\P, \OO_{\P}(\underline{a}_0-b)
			\bigr) 
	\neq 0,$ 
	then $a_0 >b>0 $ which is equivalent to $a_0 > a_0 - b > 0$. 
	However, for any $ e > 0$, \cite[I.c.3]{mcqSS},
	\begin{equation}\label{eq SS H0}
		\HH^{0}	\bigl(
					\P, \OO_{\P}(e)
				\bigr)  
		= \coprod_{\abs{E_0} a_0+ ...+ \abs{E_n} a_n = e} 
				k \cdot X_0^{E_0} \cdots X_n^{E_n}
	\end{equation}
	so $\HH^{0}	\bigl(
					\P, \OO_{\P}(e)
				\bigr) 
		\neq 0$ 
	implies 
	$e = \abs{E_0} a_0 + ...+  \abs{E_n} a_n \gqs a_0$. 
	Thus
	$\HH^{0}	\bigl(
				\P, \OO_{\P}(\underline{a}_0 - b)
				\bigr) 
	=0$
	for $a_0 > b > 0$ so by \eqref{eq SS} both items in \ref{lem L} will follow from the more general:
	\begin{claim} \label{claim L}
		Let $b > 0 $ (so $b=a_0$ is allowed) and $\mathrm{pr}$ the projection, 
		\[
		\begin{tikzcd}
		\HH^0	\bigl( 
					\P, \OO_{\P}(\underline{a}_0 -b)
				\bigr)
				& \HH^0 \bigl( 
							\P, \T_{\P}(-b) 
						\bigr),
				\ar[l, "\pr"'] 
		\end{tikzcd} 
		\]
		afforded by \eqref{eq SS},
		then the submodule $\LL_{-b}^0 := \lbrace \partial \in 
		\LL_{-b} \; | \; \mathrm{pr}(\partial)=0 \rbrace \subset \LL_{-b}$ consists only 
		of the null derivation. 
	\end{claim}

	\begin{proof}
		In order to emphasise their role say, by way of notation, that  $\{\yy_1, ..., \yy_{c_0}\}$ is the  (since any other is obtained via the action of $\Gl_k(c_0)$) block $Y:= X_0$ of weight $a_0$, and $\{\xx_{i\bullet}\}$, $i > 0$, are blocks $X_i$ of weight 		$a_i > a_0$.
		Then $\partial \in \HH^{0}(\P, \T_{\P}(-b))$ can be written as
		\begin{equation}\label{eq D1}
		\partial = 
			\sum\nolimits_{I} Y^I \partial_I, 
				\quad \text{ with }	
					\wt(\partial_I)= -b -|I|a_0 <0.
		\end{equation}
		where by hypothesis $\partial \mapsto 0$ in 
		$\HH^{0}	\bigl( 
			\P, \OO_{\P}(\underline{a}_0-b)
					\bigr)$, 
		thus by \eqref{eq SS} we have, 
		\[
		\partial \in \coprod\nolimits_{i\gqs 1} 
			\HH^{0}	\bigl(
						\P,\OO_{\P}(\underline{a_i}-b)
					\bigr)
		\]
		and each $\partial_I$ may be naturally identified to an element of 
		$\HH^{0}	\bigl(
					\sP,\T_{\sP}(-b-\abs{I} a_0)
					\bigr)$ 
		via the Euler sequence and $\G_m$-equivariance.
		Now, suppose  $\partial \in \LL_{-b}^0$ different from $0$ and let 
		\[
		i_0 = \min \lbrace \, |I| \; | \; \partial_I \neq 0 
		\text{ for some } |I| \text{ as in \eqref{eq D1}}\, \rbrace .
		\]
Similarly, we can  (again, wholly canonically because of the $\mathbb{G}_m$-equivariance) write each 
		$\ff \in \V_d $ as
		\begin{equation}\label{eq D2}
		\begin{split}
			\ff  = \ff' + \sum\nolimits_{\abs{J}>0} \ff_J Y^J,
			\quad
			\wt(\ff') = \wt(\ff_J Y^J) = 
			 a_0\abs{J} + \wt(\ff_J) = d, \;
			\end{split}
			\end{equation}
		where $\ff'$ and $\ff_J$ are non-zero $\G_{m}$-homogeneous polynomials in the variables $X_1,..., X_n$ ($\ff'$ may be identified with its image in 
		$V'_d \subseteq \HH^0
		\bigl( \sP, \OO_{\sP}(d) \bigr)$) so,  by hypothesis, $\partial(\ff) = 0$ and on the other hand
		\begin{equation}\label{eq D3}
		\partial (\ff) = 
			\sum\nolimits_{|I|=i_0} 	\Bigl( 
								Y^I\partial_I(\ff')  \quad + 
			\sum\nolimits_{J} Y^{I+J}\partial_I(\ff_J) 
							\Bigr),
		\end{equation}
		where $Y^{I+J}\partial_I(\ff_J)$ consists of monomials where $Y$ is of degree $> i_0$, therefore $\partial(\ff)=0$ only if $\sum_{|I|=i_0}  Y^I\partial_I(\ff') = 0$. 
		However, on identifying (as ever via the $\G_m$- equivariance) $V'_d$ with a subspace of $\V_d$,  $\partial_I(\ff')$ depends only on the blocks $X_{\gqs 1}$, so $\partial_I(\ff') = 0$ for all $|I|=i_0$, which, by \ref{HP}, implies the absurdity $\partial_I = 0$. 
		\end{proof}
	This certainly implies \ref{lem L} when $b\neq a_0$, while for $b=a_0$ we have
	\begin{equation*}
		\begin{tikzcd}[sep=25pt]
		\HH^{0} \bigl( 
					\P, \T_{\P}(-a_0) 
				\bigr) 
		\ar[r, "\pr"]
			&\HH^{0}\bigl( 
						\P, \OO(\ua_0 - a_0) 
					\bigr) 
			\ar[r, "\sim"]
				& \HH^{0}	\bigl( 
								\P , \OO(a_0) 
							\bigr)^{\vee}
		\end{tikzcd}
	\end{equation*}
	so in this case the claim is exactly \eqref{eq LL}.
	Better since by construction the hypothesis \ref{HP} is stable
	under base change to an arbitrary quotient of $k$, our initial conclusions are too, so \eqref{eq LL} is an injection on tensoring as soon as the definition of $\LL_{-a_0}$ enjoys the stability under base change in \eqref{CBDF2 eq}
\qedhere \end{proof}

To profit from the lemma, let us introduce,
\begin{notation/defn}\label{not/def ssym}
	Let $W := W_0 \amalg ... \amalg W_n$ 
	be a $k$-module with a $\G_{m}$-action such that $\G_{m}$ acts on $W_{i}$ by the character $\lambda^{b_i}$, $b_i \in \Z,$ for $ 0 \lqs i \lqs n$, then for $q \in \Z$, $\ssym^q(W)$ is the subspace of the symmetric algebra $\Sym(W)$ where $\G_{m}$ acts by the character $\lambda^q$. Similarly, given blocks $X_i$, $n \gqs i \gqs 0$, as in \ref{setup WPS}, with a slight abuse of notation, we define 
	\[
	\ssym^q (X_0 \amalg ... \amalg X_n) := 
		\coprod_{|E_0|a_0+ ...+ |E_n|a_n = e} 
		k \cdot X_0^{E_0} \cdots X_n^{E_n} 
	\] 
	so, in this notation \eqref{eq SS H0}
	is 
	$\HH^{0}\bigl(
		\P, \OO_{\P}(e)
		\bigr) = 
		\ssym^e(X_{0} \amalg ... \amalg X_{n})$. 
	Finally, as in \eqref{M1 eq}, if the weights $r_{0}, ..., r_{n} \in \Q_{>0}$ were any rationals and $(a_{0}, ..., a_{n} ) = D(r_{0}, ..., r_{n})$ the unique parallel tuple of positive integers without common factors, we define for $q \in \Q_{\gqs 0}$
	\begin{equation}\label{M1bis eq}
		\ssym^{q} \big( X_{0} \amalg ... \amalg X_{n} \big) := 
				\coprod_{a_{0}\abs{E_{0}} + ... a_{n}\abs{E_{n}} = Dq} 
				k \cdot X_{0}^{E_{0}} \cdots X_{n}^{E_{n}} .
	\end{equation} 
	\end{notation/defn}

In any case to apply the lemma, observe that,
\begin{equation}
	\LL := \coprod\nolimits_{b > 0} \LL_{-b} = \LL_{-a_0}
\end{equation}
is plainly a Lie algebra wherein by \eqref{eq LL} the bracket is even trivial; thus

\begin{corollary}\label{cor L}
	Again let everything be as in \ref{setup WPS}-\ref{HP} and suppose further that \eqref{eq LL} is an isomorphism onto a trivial (\ie admitting a basis) free $k$-module. As such there is a block $Z$ associated to 
	the annihilator of $\LL$, \ie 
	\begin{equation}\label{eq ZL}
	\bigcap\nolimits_{\partial \in \LL}
		\ker(\partial) 		\subset 
		\HH^{0} 			\bigl( 
								\P,\OO_{\P} (\ua_0) 
							\bigr), 
	\text{ and,}
	\end{equation}
		(i) there are \emph{ blocks}, \ie weighted projective coordinates $X_{1},...,X_{n}$, of weight $a_1,...,a_n$, generating a space of functions, $X$, such that 
		\[
		\V_d \subset \ssym^d \bigl(X \amalg \ZZ \bigr) := \ssym \bigl( X_0 \amalg ... \amalg X_n \amalg Z \bigr).
		\]
		
		(ii) If $\widetilde{X}_i, \; 1 \lqs i \lqs n$ is a 
		system of coordinates with $\wt(\wti{X}_i)=a_i$, 
		which generates a space of functions $\widetilde{X}$, and
		$\widetilde{\ZZ} \subseteq 
		\HH^{0}	\bigl(
				\P,\OO_{\P}(a_0)	
				\bigr)$ 
		such that (i) holds i.e. 
		$\V_d \subset \ssym^d \big(\wti{X} \amalg \wti{\ZZ} \big)$, then the $k$-module generated by $\wti{\ZZ}$  $\supset \ZZ$.
		
		(iii) 
		If $\widetilde{X}_i, \; 1 \lqs i \lqs n,\,Z$ is any other system of coordinate such that 
		\ref{cor L}.(i) holds, then $\widetilde{X}_{i}=\widetilde{X}_{i}(X,Z),\, 1 \lqs i \lqs n$, \ie 
		 unused coordinates are not involved.
	\end{corollary}

\begin{proof} 
	Item \ref{cor L}.(i) is trivial if $\LL_{-a_0}=0$, so 
	suppose the image of \eqref{eq LL} is non-zero, and profit from the fact that the the image is a trivial $k$-module to choose $0 \neq \partial \in \LL_{-a_{0}}$ along with coordinates $Z,\,y_{1}$ where the former is a basis of
	\[
	\ker \partial 
	\subset 
	\HH^{0}		\bigl(
				\P, \OO_{\P}(\underline{a}_0)
				\bigr)
	\]
	(thus empty if $c_{0}$ and the dimension of $\LL_{-a_0}$ are $1$), and $\partial \yy_1 =1$. 
	Again, let $X_i$, for $1 \lqs i \lqs n$, be the blocks of weight strictly greater than $a_0$; so 
	$ Z,\, \{ y_1 \},\, X_i= \{ \xx_{i \bullet} \},$ 
	$ \; 1 \lqs i \lqs n$ is a basis for everything and in these of coordinates $\partial$ takes the form 
	\begin{equation}\label{eq Par1}
		\partial = 	\frac{\partial}{\partial \yy_1 } + 
				\sum_{i=1}^n	\biggl( \, 
								\sum_{j=1}^{c_i} 
									\l_{ij} \, 
									\frac{\partial}{\partial \xx_{ij}}
								\, \biggr), 
		\qquad \wt(\xx_{ij}) = a_{i} > a_0,
	\end{equation}
	where $\wt(\l_{ij}) - \wt(\xx_{ij}) = - a_0$, so 
	$\wt(\l_{ij}) = \wt(\xx_{ij}) - a_0 < \wt(\xx_{ij})$, thus 
	\begin{equation}\label{eq lam ij}
	\l_{ij} = \l_{ij}(Z, \yy_1, X_{<i}), \quad \text{where} \; \wt(X_{<i}) < a_i, 
	\end{equation}
	\ie $\l_{ij}$ only depends on variables of weight strictly less than $a_i$. 	
	To simplify the notation we'll write $\partial_{\yy_1}$, resp. $\partial_{\xx_{ij}}$, for $\frac{\partial}{\partial \yy_1}$, resp. $\frac{\partial}{\partial \xx_{ij}}$, and employ the summation convention so that \eqref{eq Par1} becomes:
	\begin{equation}\label{eq P1}
	\partial = \partial_{\yy_1} + \l_{ij}\,\partial_{\xx_{ij}} \;.
	\end{equation}
	
	By increasing induction on $\wt(X_i)$ we will eliminate everything from \eqref{eq P1}, except ${\partial_{\yy_1}}$, by way of a global change of weighted projective coordinates. The starting point is $a_{i-1} = a_0$ which is a minor abuse of notation, but it is certainly true, so by induction we have
	\begin{equation}\label{eq P2}	
	\partial = 
	\partial_{\yy_1} + 
	\l_{hj}\,\partial_{\xx_{hj}}, 
			\quad \wt(\xx_{hj}) \gqs a_{i}.
	\end{equation}
	Thus in weight $a_{i}$ we aim for a global change of coordinates of the form
	\begin{equation}\label{eq P3}	
	\xx_{ij} \mapsto \xx_{ij} + \GG_{ij}(Z, \yy_1, X_{<i}), 
	\quad
	\wt(X_{< i}) < a_i =  \wt(\GG_{ij}) 
	\end{equation}
	and otherwise do nothing for weights strictly greater than $a_{i}$. 
	Consequently we need to solve 
	$ \partial \big(\xx_{ij} + \GG_{ij} \big) = 0$, \ie
	\begin{equation}\label{eq P4}	
	  \partial \big(\xx_{ij} + \GG_{ij} \big) =
	\l_{ij} + \partial_{\yy_1} \, \GG_{ij}	= 0,
	\end{equation}
	which is trivially solvable on any ring of characteristic 0 by \eqref{eq lam ij} with
	$ \wt( \GG_{ij}) = \wt(\l_{ij}) - \wt(\partial_{\yy_1})  = a_i.$
	As such for our given $\partial$ we have a system of coordinates
	$ \{\,Z, \yy_1, X_i \,\}$ such that $\partial = \partial_{\yy_1}$ and, of course, any other $D \in \LL_{-a_0}$ can be expressed in this basis as
	\begin{equation} \label{eq P5}	
	D = \nu \, \partial_{Z} + \mu \, \partial_{\yy_1}  + \l_i\, \partial_{X_i},
	\end{equation}
	with $ \mu, \nu, \in k $ but not $\l_i$ if $\l_i \neq 0,$ where  $\l_i \partial_{X_i} := \sum_{j=1}^{c_i} \l_{ij} \partial_{\xx_{ij}}$. 
	By \eqref{eq LL} if $D$ is linearly independent of $\partial_{\yy_1}$, replacing $D$ by $D - \mu\, \partial_{\yy_1}$, $\mu=0$ and some $\nu\, \partial_{Z} \neq 0$. Further, from $[\partial, D]=0$, $D$ is canonically a derivation of the algebra $k[Z , X_1, ..., X_n]$, which in turn inherits a $\mathbb{G}_m-$action.
	Consequently we may repeat the first step for $D$ and $\ker D$ to get coordinates $\yy_1, \yy_2, X_i$, $ n \gqs i > 0$, and $Z$, which, now, is a block of coordinates of $\ker\partial \cap \ker D$, in which
	\begin{equation}\label{eq P6}
	\partial = \partial_{\yy_1}, \quad D = \partial_{\yy_2}.
	\end{equation}
	and whence, by induction we arrive at a $\mathbb{G}_m$-equivariant system of coordinates $Z$,  $Y= \{\yy_1,...,\yy_{\ell}\}$, $X_i, \, n \gqs i > 0$
	  with the properties
	\begin{equation}\label{Cor 1234}
	\begin{split}
		(1) \quad &\partial_Y = \{\partial_{\yy_1},...,\partial_{\yy_\ell}\} 
		\text{ is a basis of } 
		\LL_{-a_0},\,  
		Y=\{ \yy_1, ..., \yy_\ell \} \text{ its dual basis;} \\
		(2) \quad &Z \text{ is a basis of ann}(\LL) \text{ in } \HH^{0} (P, \OO_{\P}(a_0)), \cfr \text{ \eqref{eq ZL}};\\
		(3) \quad &X_i, \; 1 \lqs i \lqs n, \text{ are the other coordinates;}\\
		(4) \quad & \text{ $\V_d$ is a submodule of weight $d a_0$ of the } \mathbb{G}_m \text{-algebra } k[Z,X_{\gqs 1}];
		\end{split}
		\end{equation}
	which complete the proof of part \ref{cor L}.(i).
	
	In regard to part \ref{cor L}.(ii), under the hypothesis of \op, $\LL$ contains a subspace of fields, $M$, whose annihilator under the natural map of \eqref{eq LL} is generated by $\wti{Z}$, while the annihilator of $\LL$ is generated by $Z$, so from $M \subseteq \LL$ we get $Z$ is contained in the $k$-module generated by $\wti{Z}$. 

	Finally as to part (iii), by definition the $\wti{X}_i$'s and the $X_i$'s, $ 1 \lqs i \lqs n $, modulo 
	$\HH^{0}	\bigl(
				\P,\OO_{\P}(a_0)
				\bigr),$ 
	are systems of weighted projective coordinates of the sub-weighted projective champ $\sP = \PP(a_1,...,a_n)$, \cfr \ref{HP}, so without loss of generality (\ie after replacing say the $X_i$'s by a weighted automorphism of themselves) with $Y, \, Z$ as in \eqref{Cor 1234} 
	\begin{equation}
		\widetilde{X}_i = X_i \mod (Y,Z).
	\end{equation} 
	and we have:
	\begin{equation}\label{eq WA1}	
	\widetilde{X}_i = 
		X_i + 
		\widetilde{X}_i(Z, X_{\lqs i}) + 
			\sum\nolimits_{\abs{E}=\x_i}
		Y^{E}
		\lambda_{E}(Z, X_{\lqs i}) + 
	\bigl(
		\text{higher order in $Y'$\text{s}}
	\bigr),
	\end{equation}
	where $\wt(Y^{E} \lambda_{E})= \abs{E} a_0 + \wt(\lambda_{E}) = a_i$, and by definition $\x_i$ is of minimal weight amongst the monomials in $Y$. As such, we may, in light of our goal, 
	\ref{cor L}.(iii), \wloss replace  $X_{i} + \widetilde{X}_{i}(Z, X_{\lqs i})$ by $X_{i}$  (which is an automorphism because it is so modulo $Z$)  so that for $\y = \min_{i} \{ \x_i \}$ \eqref{eq WA1} is 
	\begin{equation}\label{eq WA2}
	\begin{split}
	\widetilde{X}_i 
		= X_{i} + 
			\sum\nolimits_{\abs{E_i}=\y } 
				Y^{E_i} \lambda_{E_i}(Z, \uX)  
	+
	\bigl( 
		\text{order $\gqs \y+1$ in  $Y'$\text{s}}
	\bigr) 
		=: X_i + \eta_{i},
	\end{split}
	\end{equation}
	and by hypothesis every $\ff = \ff(Z, \uX) \in \V_d$ can be written as $\varphi_\ff(Z,\widetilde{\uX})$, where 
	$\uX := (X_1,...,X_n)$
		and 
		$\wti{\uX} = \uX + \underline{\eta}$, \eqref{eq WA2}. However from
	\begin{equation}\label{eq WA3}
	\varphi_\ff(Z, \uX + \underline{\eta}) = \ff(Z, \uX)
	\end{equation}
	we must have $\varphi_\ff= \ff$ and whence
	\begin{equation}\label{eq WA4}
		\begin{split}
		\ff(Z, \uX) =&\, \ff(Z, \uX + \underline{\eta}) \\
					=&\, \ff(Z, \uX) + \eta_{i} (\partial_{X_i} \, \ff )(Z, \uX) 
		+
		\bigl( 
			\text{order $\gqs \y+1 $ in  $Y'$ }
		\bigr) 
		\\
					=&\, \ff(Z, \uX) + 
					\sum\nolimits_{\abs{E_i}=\y} 
						Y^{E_i} \lambda_{E_i} 
						(\partial_{X_i}\,\ff) 	+
		\bigl(		
			\text{order $\gqs \y+1 $ in  $Y'$ }
		\bigr)
	\end{split}
	\end{equation}
	Thus for every $E_i$ with $\abs{E_i}=\y$ the term $\sum_i \lambda_{E_i} (\partial_{X_i}\ff)$ must be  equal to $0$ and, as we have said, $\abs{E_i} a_0 + \wt(\lambda_{E_i}) = \wt(X_i)$, so the operator 
	$\lambda_{E_i} \partial_{X_i}$ has weight
	\begin{equation}\label{eq WA5}
	\wt(\lambda_{E_i} \partial_{X_i})= 
	\wt(\lambda_{E_i}) + \wt(\partial_{X_i}) = 
	a_i - \abs{E_i} a_0 - a_i = -\abs{E_i}a_0 < 0
	\end{equation}
	and it vanishes on all of $\V_d$, so it belongs to $\LL_{-a_0}$. Therefore by lemma \ref{lem L} its image under \eqref{eq LL} in $\HH(\P,\OO_{\P}(a_0))^\vee$ is non-zero, which is nonsense since
	$\lambda_{E_i} \partial_{X_i}$ has value $0$ on both the $Y$'s and the $Z$'s.
	\qedhere \end{proof}
\section{The Invariant}\label{sec II INV}
\setcounter{equation}{0}
We are going to define an invariant of rings and their ideals which is most naturally expressed in an appropriate number of copies of $\Q_{\gqs 0}$ with the lexicographic ordering. On the other hand this is not a discrete group, so to avoid fastidious statements about denominators we introduce,

\begin{definition}\label{MD1 defn}
	Let $N \in \Z_{\gqs 0}$; $\Q^{N+1}$ ordered lexicographically; and
	$\pr_i$, resp. $\pr_{\lqs i}$, the projection onto the $i^{\th}$ factor, resp. first $i$ factors, $1 \lqs i \lqs N$, then a function
	\begin{tikzcd}[cramped, sep=small]
		f: E 	\ar[r]		&\Q_{\gqs 0}^{N+1}
	\end{tikzcd} 
	is said to have self bounding denominators 
	if,
	
		(i)		
		\begin{tikzcd}[cramped, sep=small]
		f ^{*}\pr_1 : E	\ar[r]	&\Q_{\gqs 0}
		\end{tikzcd}
		takes values in $\Z_{\gqs 0}$.
		
		(ii)		
		If $N \gqs 1$, then for all $1 \lqs i \lqs N$ there are increasing (in the lexicographic order) functions
		\begin{tikzcd}[cramped, sep=small]
			D_{i}: \Q_{\gqs 0}^{i}		\ar[r] 		&\Z_{\gqs 0}
		\end{tikzcd} 
		such that,
		\begin{equation}\label{M2 eq}
			\big( 
				f^{*} \pr_{\lqs i}^{*} D_i
			\big)
			f^{*}\pr_{i+1} \in \Z_{\gqs 0} .
		\end{equation}
\end{definition}

The utility of the definition results from,

\begin{fact}\label{MF1 fact}
	Let everything be as in \ref{MD1 defn} with 
	\begin{tikzcd}[cramped, sep=small]
		f:E		\ar[r]		&\Q_{\gqs 0}^{N+1}
	\end{tikzcd}
	a function enjoying self bounding denominators, and define a function
	\begin{tikzcd}[cramped, sep=small]
		F:E		\ar[r]		&Z_{\gqs 0}^{N+1}
	\end{tikzcd}
	whose first projection is that of $f$ while its $(i+1)^{\th}$ projection is \eqref{M2 eq} for $1 \lqs i \lqs N$, then in the lexicographic order,
	\[
			f(x) \lqs f(y) 
				\quad \xLeftrightarrow[\null]{\qquad} \quad
			F(x) \lqs F(y).
	\]
\end{fact}

\begin{proof}
	Manifestly 
	\ref{MF1 fact}
	is true if $N=0$, so suppose $N \gqs 1$ and $f(x) < f(y)$, then without loss of generality, 
	$	\pr_{\lqs N} f(x) =  	\pr_{\lqs N} f(y) $
	but 
	$	\pr_{ N + 1 } f(x) <  	\pr_{ N + 1} f(y) .$	
	Consequently, $\big(	f^{*} \pr_{\lqs N}^{*} D_{N}	\big)$
	is the same at $x$ and $y$, so:
	$
		\pr_{ N + 1 }f(x) \lqs \pr_{ N + 1 }f(y) 
	$ 
	iff
	$
		\pr_{ N + 1 }F(x) \lqs \pr_{ N + 1 }F(y).
	$
\qedhere \end{proof}

\begin{setup/notation}\label{setup INDU}
	$A$ is a regular local ring of dimension $m$, with residue field $k$ of characteristic $0$, and $\mm$ its maximal ideal. We will employ,
\end{setup/notation}

\begin{definition}\label{def WF}
	A regular weighted filtration (or simply a weighted filtration or even just filtration if there is no danger of confusion) on a ring $A$, is the filtration, $F^{\bullet}$, associated to a system of coordinates (\ie modulo $\mm^2$ affords a basis of $\MM$) 
	$\{\, x_1 \, ,..., \, x_m \,\}$ and non-negative numbers, 
	$r_1, ... , r_m$, by the ideals, 
	\begin{equation}\label{eq def WF}
	F^p A = 
	\lbrace \, 
		x_{1}^{e_1} \cdot ... \cdot 
		x_{m}^{e_m}
		\; \big| \; 
		r_1		\,e_1 +...+ 
		r_m		\,e_m \gqs p
	\, \rbrace,
		\quad p \in \Q_{>0}.
	\end{equation}
\end{definition}
 In addition, since in the string of rationals $(r_1, ... , r_m) \in \Q^{m}_{\gqs 0}$, repetitions are allowed, we define
\begin{definition}\label{def blocks}
	A block of coordinates, $X$, is a set which may be extended to a system of coordinates and, which is maximal amongst such sets with the same weight. 
	In particular any weighted filtration can always be expressed in terms of a system of blocks $X_0,...,X_s$, $s <m$, where each $X_i$ has the same weight and $X_0 \amalg ... \amalg X_s$ is a system of coordinates of $A$.
\end{definition}

For $I$ an ideal of $A$ we will define inductively a weighted filtration $F^\bullet(I)$ which only depends on the pairs $(A,I)$ together with 
\begin{equation}\label{eq invIA}
\inv(I) = \inv_A(I) \in \Q_{\gqs 0}^{2m}
\end{equation}
where $\Zz^{2m}$ is endowed with the lexicographic ordering.
At each step $s \gqs 0$ of the induction we will, actually, define two successive entries of $\inv(I)$, $(g_s,\ell_s)$, beginning with

\begin{start indu}\label{start INDUCT}
Let $A$ be as in \ref{setup INDU}, and $I \lhd A$ an ideal, 
then the multiplicity of $I$, is  
\[ \mult(I) := 
				\begin{cases}
					\max \{ \alpha \in \Zz \, \big| \, I \subseteq \mm^{\alpha} \}, 
						&I \neq 0 
				\\
					\infty \;,
						&I =0
				\end{cases}
\]
As such if $\mult(I) = d \in \Z_{> 0}$,
\[
\begin{tikzcd}
	\V_d := I \mod \bigl(\mm^{d+1} \bigr) \ar[r, hook] 			 
			&\Sym^d\bigl(\MM\bigr),
\end{tikzcd}
\]
and we apply lemma \ref{lem L} to 
\begin{equation}
\begin{tikzcd}
\V_d \ar[r, hook]
	&\HH^{0} \bigl( \mathbb{P}(\MM), \,\OO_{\PP(\MM)}(d)\bigr)
\end{tikzcd}
\end{equation}
with $\ell_0(I) := \dim \LL_{-1}(V_d)$, in notation of \eqref{CBF1 eq}. Then by corollary \ref{cor L}.(i) there is a unique minimal subspace $Z = Z(I) \subseteq \MM$ of dimension $c_0:= m-\ell_0$ such that $\V_d \subseteq \Sym^d (\,Z\,)$. We therefore start the induction by way of:

		(S.0)
	The first two entries of $\inv(I)$ are equal to $\bigl(\mult(I), \ell_0(I) \bigr)$.
	
	(S.1)
	If either of these entries of the invariant are zero, then so are all the subsequent ones, and the process terminates.
	
	(S.2)
	The weighted filtration $F_0^\bullet(I)$ is the weighted filtration in which each $x_i$ has weight $1$, \ie the powers of the maximal ideal $\mm^{\bullet}$.
	
	(S.3)
	Under the hypothesis of (S.1), the definition of $F^{\bullet}(I)$ also terminates, $F^{\bullet}(I) = F_{0}^{\bullet}(I)$.
	
	(S.4)
	The first block, $X_0$, of cardinality $c_0$ is a choice of basis of $Z$.
\end{start indu}

\begin{indu hypo}\label{indu hypo}
For $s \gqs 1$, there is a (weighted) filtration $F_{s-1}^{\bullet}(I)$ depending only on $I$ (and for this reason we will write just $F_{s-1}^{\bullet}$ if there is no danger of confusion) defined by blocks of coordinates $X_{s-1}^{0},...,X_{s-1}^{s-1},$ respectively $Y$ of cardinality $c_0, c_1, ..., c_{s-1},$ respectively $\ell_{s-1}$, where, for  $0 \lqs i < s-1$,
\begin{equation}\label{eq m-c=ell}
\ell_i := m - (c_0 + ... + c_i) \; 	\text{ or equivalently }
\ell_{i+1} := \ell_i - c_{i+1}, \;
\end{equation}
and rationals weights $g_{s-1}^{0} > g_{s-1}^{1}> ... > g_{s-1}^{s-1} \in \Q_{>0}^{s},\; g_{s-1}^{i} \gqs 1$ such that:
	
	(F.0)
	If $Y$ is any block completing $X_{s-1}^{0},...,X_{s-1}^{s-1}$ to a system of coordinates then $1=\wt(Y) \lqs g_{s-1}^{s-1} $.
	
	(F.1)
	$ I \subseteq F^{d g_{s-1}^{0}}_{s-1}$.
	
	(F.2)
	For 
	$V_{s-1}^{d a_{s-1}^0}:=I \mod F^{>d a_{s-1}^0}_{s-1},$ 
	$V_{s-1}^{d a_{s-1}^0} \subseteq \ssym^{d a_{s-1}^0} 
	\big( X_{s-1}^{0}\amalg ... \amalg X_{s-1}^{s-1}\big)$, \cfr \ref{not/def ssym}. 
	
	(F.3)
	There are no vector fields of negative weight on the 
	 $\rP(\ug)$, \cfr \eqref{M1 eq}, associated to the graded algebra
	\begin{equation}\label{eq grA}
		\gr_{s-1}A = 
		\coprod\nolimits_{q \gqs 0} 
		\quoziente{F_{s-1}^{q}}{F_{s-1}^{q+1}}
	\end{equation}
	 leaving $V_{s-1}^{d a_{s-1}^0}$ invariant.
	
	(F.4)
	There are strictly positive integers 
	$d^{t}_i$, $0 \lqs i \lqs t \lqs s-1$, $d^{0}_{0}=d$ as in \ref{start INDUCT}, such that the weights $g_t^i$ are derived from $g_t \in \Q_{>0}$ according to the following rules:
	if given $g_{t}$, we define
	$g_t^i = g_{i+1}...g_{t}, \; g_{t}^{t} = 1$, then
	\begin{equation}\label{eq gg1}
	\begin{split}
					g_{0}^{0} = 
					g_0 						
						&= 1 			\\
					g_1^0  d^0_0 - 
					(g_1^0 d^1_0)			
						&= d^1_1 		\\
			 		g_2^0 d^0_0 - 
			 		(g_2^0 d^2_0 + g_2^1 d^2_1)	
			 			&=d^2_2			\\
					\vdots \qquad \qquad \qquad 
						&\quad \vdots  	\\
					g_{s-1}^{0} d^{0}_{0} 	
						- \big( 
							g_{s-1}^{0} d^{s-1}_{0}		+ 
							g_{s-1}^{1} d^{s-1}_{1}		+ ... + 
							g_{s-1}^{s-2} d^{s-1}_{s-2}	\big)	
								&= d^{s-1}_{s-1}, \;
	\end{split}
	\end{equation}

	\begin{equation}\label{eq gg2}
		\begin{split}
		\text{and, }
			g_{t}^{0} d^{t+1}_{0}		+ 
			g_{t}^{1} d^{t+1}_{1}		+ ... + 
			g_{t}^{t-1} d^{t+1}_{t-1}	+	
			d^{t+1}_{t} 				+ 
			d^{t+1}_{t+1}				\;
				> g_{t}^{0} d^{0}_{0}&		
				, \\
			\text{for every }			\, 
				0 \lqs t \lqs s-2& 		\, .
		\end{split}
	\end{equation}
	Notice that by \eqref{eq gg1} \& \eqref{eq gg2}, $g_t > 1$ for every $1 \lqs t \lqs s-1$.
	
	(F.5)
	The function $\ug = (d, g_1, ..., g_{s-1})$
	of rings and their ideals has self bounding denominators, \ref{MD1 defn}. 
	
\end{indu hypo}

\begin{indu/def}\label{indu/def INV Fs}
The induction is divided as follows:
\end{indu/def}

\begin{step}\label{step1} 
	If $c_0 + ... + c_{s-1}= m$, or equivalently, by \ref{eq m-c=ell}, if $\ell_{s-1}=0$, then stop and define 
	$F^p(I):= F^p_{s-1}(I)$,  
	together with the invariant:
	\begin{equation}\label{EQS1}
		\inv{(I)} :=
			\begin{cases}
				(\,d, \,\ell_0, \underline{0} \,) 			\;, 	
				& s=1; 
			\\
				(\,d, \, 	\ell_0,  
					\, 	g_1, 
					\, 	\ell_1			,..., 
					\, 	g_{s-1}, 
						\ell_{s-1}=0,
					\,	\underline{0}		\,)\;,	
				& s\gqs 2,
			\end{cases}
	\end{equation}
	wherein $\ell_{s-1}$ and the last $2(m-s)$ entries are equal to $0$.
\end{step}

\begin{step}\label{step2} 
	Otherwise $m - (c_0 + ... + c_{s-1}) = \ell_{s-1} > 0$, and define for $H \in \Q_{>1}$ a set 
	$ \Lambda_H := \{(\x_0,...,\x_{s-1},\y)\} \subseteq \Z_{\gqs 0}^{s+1}$ by the rules: 

	(R.1)
	$ H \cdot 
	\big( 
	g_{s-1}^0		\,\x_0 + ... +  
	g_{s-1}^{s-1} 	\,\x_{s-1}
	\big) + 
	\y \gqs 
	H \cdot 
	\bigl( 
	g_{s-1}^0		\,d 
	\bigr); $
	
	(R.2) 
	$g_{s-1}^0		\,\x_0 + ... +  
	g_{s-1}^{s-1} 	\,\x_{s-1} 
	< 
	g_{s-1}^0 		\, d$. 

Now observe that by (R.2) the possibilities for $(\x_i)$ are finite, so if (R.1) is an actual equality for some $H$ then the denominator of $H$ is bounded. It therefore makes sense to introduce
\begin{fact/defn}\label{f/d theta}
	The discrete set of sub-inductive parameters $\Theta_{s-1}(I,A)$, contained in $\Q_{>1}$, is the subset of $H \in \Q_{>1}$ where equality occurs in (R.1) for some tuple of integers satisfying (R.2), and its predecessor $h = h(H)$ is the minimum of $\Theta_{s-1} \cap \Q_{<H}$ or $1$ if $H$ is already the minimum of $\Theta_{s-1}$.
	\end{fact/defn}
Better still, observe,
\begin{fact}\label{MF2}
	Let $\ug = \ug(I,A)$ be as in \ref{indu hypo}.(F.5), and 
	$h_s = h_s(I,A)$ any function taking values in the set $\Theta_{s-1}(I,A)$ of sub-inductive parameters in \ref{f/d theta}, then $\ug \times h_s$ is a function of rings and their ideals with self bounding denominators.
\end{fact}
\begin{proof}
	By the definition of $h_s$ there are non-negative integers $\x_i$ and a positive integer $\y$ such that \ref{step2}.(R.1) is an equality. In addition there are 
	\begin{tikzcd}[cramped, sep=small]
		D_i : \Q_{\gqs 0}^{i}		\ar[r]		&Z_{\gqs 0},
	\end{tikzcd}
	$0 \lqs i \lqs s-1$ self bounding the denominators of $\ug$ in the sense of \ref{MD1 defn}. Consequently we must have,
	\[
		\bigl(	D_0 \cdots D_{s-1}	\bigr)(\ug) \y = h_s N
	\] 
	where $N \in \Z_{> 0}$ is an integer no greater than
	\begin{equation}\label{EQQM3}
		d g_{s-1}^{0} \bigl(	D_0 \cdots D_{s-1}	\bigr)(\ug)
	\end{equation}
	so $D_s$ the factorial of \eqref{EQQM3} will do.
\qedhere \end{proof}

Having cleared any scruples about denominators, consider the following,

\begin{s-induction}[$H \in \Theta_{s-1}$]\label{s-induction}
	For $h = h(H)$ the predecessor of $H$, and $h_{s-1}^{i} = h \cdot g_{s-1}^i$, $0 \lqs i \lqs s-1$, there is a weighted filtration $F_{s-1}^\bullet(h)$	depending only on $I$, in which all of \ref{indu hypo}.(F.0)-(F.3) hold but with $h_{s-1}^i$ instead of $g_{s-1}^i$.
	\end{s-induction}
Plainly the sub-induction \ref{s-induction} begins with $F_{s-1}^\bullet(1) = F_{s-1}^\bullet$, while
by corollary \ref{cor L}.(iii) each block 
$X_{s-1}^{i}$, $0 \lqs {i} \lqs s-1$, is (up to a weighted projective transformation in the $X_{s-1}^{t},\;  0 \lqs {t} < {i} \lqs s-1$) well defined modulo $F_{s-1}^{h_{s-1}^{i}}(h)$. 	
As such if $\widetilde{X}_{s-1}^{i}$ and $\widehat{X}_{s-1}^{i}$ are any two liftings of the $i$-th block to $A$, then 
\begin{equation}\label{eq W-D}
	\widetilde{X}_{s-1}^{i} = \widehat{X}_{s-1}^{i} \quad \mod F_{s-1}^{>h_{s-1}^{i}}(h)
	\end{equation}
and we assert that for $H$ as in \ref{s-induction},
\begin{lemma}\label{lem W-D}
	If $\widetilde{X}_{s-1}^{i}, \; 0 \lqs {i} \lqs s-1$, is a lifting of the blocks from $\gr_{s-1}^{(h)}A$ (\cfr \ref{eq grA}), and $\widetilde{X}_{s-1}^s$ some choice of completing this to a system of coordinates, then the new filtration, ${F}_{s-1}^\bullet(H)$ say, defined by the weights
	\begin{equation}\label{eq W-D pesi}
	\begin{split}
		&\wt_{H}(\widetilde{X}_{s-1}^{i}) 
			= H\cdot g_{s-1}^{i}, \text{  for } \; 
				0 \lqs {i} \lqs s-1, \\
		&\wt_{H}(\widetilde{X}_{s-1}^s) = 1\, ,
	\end{split}
	\end{equation}
	does not depend on the aforesaid choices. 
	\end{lemma}
\begin{proof}
	To this end, by \eqref{eq W-D}, it is sufficient to prove 
	\begin{claim}\label{claim W-D}
		$\ff 	\in 	F_{s-1}^{>\,h g_{s-1}^{i}}(h)
			\Longrightarrow 
		\wt_{H}(\ff) \gqs H \cdot g_{s-1}^{i}$, 
		\ie $\ff \in {F}_{s-1}^{H g_{s-1}^{i}}(H)$.
		\end{claim}
	\begin{proof}
		By hypothesis $\ff$ is contained in the ideal, $F_{s-1}^{>\,h g_{s-1}^{i}}(h)$, generated by monomials with total degrees $\x_{i}$, resp. $\y$, for the blocks $X_{s-1}^{i}$, $ 0 \lqs {i} \lqs s-1$, resp. $X_{s-1}^{s}$, such that:
		\begin{equation}\label{eq W1}
			h \cdot \bigl(	g_{s-1}^0		\,	
							\x_0 			+ ... +  
			 				g_{s-1}^{s-1}	\,	
			 				\x_{s-1} 
			 		\bigr) 
			 		+ \y					\,	>\,  
			  		h \cdot g_{s-1}^{i} ;
		\end{equation}
		while from the definition of the integers $d^{i}_i$, 
		\ref{indu hypo}-(F.4),
		\begin{equation}\label{eq W2}
			g_{i}^0		\,d^{i}_0 + 
			g_{i}^1		\,d^{i}_1 + ... +
			g_{i}^{{i}-1}	\,d^{i}_{{i}-1} + 
			(d^{i}_{i} -1) = 
			g_{i}^0 		\,d^0_0 -1
		\end{equation}
		so multiplying this by $g_{s-1}^{i}$
		we get
		\begin{equation}\label{eq W3}
			g_{s-1}^0 		\,d^{i}_{0}		+ 
			g_{s-1}^1		\,d^{i}_{1} 		+ ... +
			g_{s-1}^{{i}-1}	\,d^{i}_{{i}-1} 	+ 
			g_{s-1}^{i}		\,(d^{i}_{i} -1) 	=
			g_{s-1}^0 		\,d^0_0 		- 
			g_{s-1}^{i}
		\end{equation}
		then multiplying \eqref{eq W3} by $h$ and adding it to \eqref{eq W1} gives:
		\begin{equation}\label{eq W4}
		\begin{split}
		h \cdot \Big(	
			g_{s-1}^0 		\,(\x_0 		+ d^{i}_0		) 	+ ... +
			g_{s-1}^{{i}-1} 	\,(\x_{{i}-1} 	+ d^{i}_{{i}-1}	)	+
			g_{s-1}^{i} 		\,(\x_{i} 		+ d^{i}_{i} -1 	) 	\; +
			\\
			g_{s-1}^{{i}+1}	\,\x_{{i}+1} + ... +  
			g_{s-1}^{s-1} 	\,\x_{s-1} 
			\Big) +  
			\y > 
		h \cdot 
			g_{s-1}^0 \, d^{0}_0
			\end{split}
		\end{equation}  
		so from the definition of $h = h(H)$, \ref{f/d theta},
		\begin{equation}\label{eq W5}
		\begin{split}
		H \cdot 
		\Big( 
		g_{s-1}^0		\,(\x_0 		+ d^{i}_0			) + ... +
		g_{s-1}^{{i}-1}	\,(\x_{{i}-1} 	+ d^{i}_{{i}-1}		) + 
		g_{s-1}^{i}		\,(\x_{i} 		+ d^{i}_{i} 	-1	) +
		\\  
		g_{s-1}^{{i}+1}	\,\x_{{i}+1} 	+ ... +  
		g_{s-1}^{s-1} 	\,\x_{s-1}
		\Big) 							+
		\y \gqs  
		H \cdot 
		g_{s-1}^0 		\,d_{0}^0 	.
		\end{split}
		\end{equation}
		Now multiply \eqref{eq W3} by $H$ and subtract from \ref{eq W5}
		to get
		\begin{equation}\label{eq W6}
		H \cdot 
		g_{s-1}^0		\,\x_0  		+ ... + 
		H \cdot 
		g_{s-1}^{s-1} 	\,\x_{s-1} 	+ 
		\y \; \gqs \; 
		H \cdot 
		g_{s-1}^{i} \,	,
		\end{equation}
		wherein the left hand side is the monomial's weight in the new $H$-filtration.
		\qedhere \end{proof}
	Which in turn complete the poof of \ref{lem W-D} .
	\qedhere \end{proof}

Now in the new filtration ${F}_{(s-1)}^\bullet(H)$, \ie the filtration obtained from 
$F_{(s-1)}^\bullet(h)$ of \eqref{eq W-D pesi} (and unambiguously by \ref{lem W-D}), define
\begin{equation}\label{eq VF}
{V}_{s-1}^{d}(H) := I \mod {F}_{s-1}^{> H g_{s-1}^0\,d}(H),
\end{equation}
then one of the following must occur,
\begin{caseA}\label{case A}
	$\LL	\bigl(
	V_{s-1}^{d}(H)
	\bigr)$ 
	(\cfr \ref{lem L}) does not have maximal dimension, \ie 
	\[
	\dim \LL	\bigl(
	V_{s-1}^{d}(H)
	\bigr) = 
	\ell_s 		< 
	m_s := m - (c_0 + ... + c_{s-1}).
	\] 
	Then by corollary \ref{cor L} applied to 
	$\rP	\bigl(	
	H a_{s-1}^0,..., H a_{s-1}^{s-1}, \underline{1}
	\bigr)$, 
	\eqref{M1 eq}, there is a filtration satisfying 
	(F.1)-(F.4) of \ref{indu hypo} but with blocks $X_s^i,\; 0 \lqs i < s$,  respectively $X_{s}^s$, liftings of the 
	blocks $X_i$, respectively $Z$,  
	\ie the annihilator of 
	$\LL	\bigl(
	V_{s-1}^{d}(H)
	\bigr)$ 
	in corollary \ref{cor L}, and 
	$c_s = m_s - \ell_s$ 
	while $g_{s+1}=H$ in \ref{indu hypo}.(F.4), \ie
	$g_{s}^i= H \cdot g_{s-1}^i$ with $g_{s}^{s} = 1$.
\end{caseA}

\begin{figure}[!ht]
	\begin{flushleft}
		\begin{tikzpicture}[scale=0.8, domain=-1:3.5]
		\draw [lightgray, fill=lightgray] 
				(0,5) -- (4,2.8) -- (4,5);
		\draw [lightgray,fill=lightgray] 
				(4,2.8) -- (10,2) -- (10,3);
		\draw [lightgray, fill=lightgray] 
				(0,5) rectangle (10,6);
		\draw [lightgray, fill=lightgray] 
				(4,2.8) rectangle (10,5);
		\node at (6.5, 4) 
		{$\Lambda_H$};
		\draw[ultra thick, ->]
				(-0.2,0) -- (10.5,0) ;	
		\node at (10.3, -0.55) 
				{$Y$}; 
		\draw[ultra thick, ->] 
				(0,0) -- (0,6.5)
				node[above left] {$\underline{X}$};
		\node at (0,6.9) {=};
		\node at (1.8,6.95) {$\bigl(X_0,...,X_{s-1}\bigr)$};
		\draw[black, dashed]		
				(0,5) -- (5,0)		
				node[below] 	
				{$1$}; 
		\draw[black, dashed] 	
				(0,5) -- (6.5,0) 		
				node[below right] 	
				{$h(H)$};
		\draw[black] 
				(4,2.8) -- (8.5,0)			
				node[below right] 	
				{$H$};
		\tikzstyle{EdgeStyle}=[bend left]
		\draw[->, thick, black]
				(0.5,-1.6)	
				node[right] 
				{$
					\begin{array}{l}
					\text{Start of the induction:}\\
					g_{s-1}^0 \x_{0} 			+ ... + 
					g_{s-1}^{s-1} \x_{s-1}	 	+ \y
					\gqs d^{0}_{0} g_{s-1}^0
					\end{array}
				$}  	
				to [in=165, out=100] (2,2.8);
		\draw[->, thick, black]
				(4, 7)	
				node [right] 
				{ 
					$\begin{array}{l}	
				\text{The previous stage in}
				\\ 
				\text{the sub-induction
				\ref{s-induction}}
				\end{array}	$
				} 
				to (3,2.9);
		\draw[->, thick, black]
		(12,1) 
		node[right]
		{$
			\begin{array}{l}
			\text{If no weighted change of} 
		\\
			\text{coordinates removes the }
		\\
			\text{edge } e(H), \text{ \ref{case A}, stop,} 
		\\
			\text{and otherwise, \ref{case B}, }
		\\
			\text{continue.}
			\end{array}
		$}
		to (5.5, 2);
		\draw[->, thick, black]
		(12,6) 
		node[right]
		{$
			\begin{array}{l}
			\text{Upper limit,} 
			\\
				g_{s-1}^0 \x_{0} 			\,+ \,.\,.\,.\, + 
			g_{s-1}^{s-1} \x_{s-1}
				= d^{0}_{0} g_{s-1}^0
			\end{array}
		$}
		to (9,5);
		\draw[very thick] 
				(0,5) 
				edge node[above]
				{$e(H)$} (4,2.8) ;
		\draw[very thick, dashed] 
				(4,2.8) -- (10,2);
		\draw
				(0,5) -- (10.2,5);
		\draw [fill] 
				(0, 5) 	circle [radius=.1] 	
		node[above left] 
				{\null};
		\draw [fill] 
				(4, 2.8) 	
				circle [radius=.1];
		\end{tikzpicture}
		\caption{Newton Polyhedron for Sub-Induction \ref{s-induction}.
		\label{fig:NP}
}
	\end{flushleft}
\end{figure}
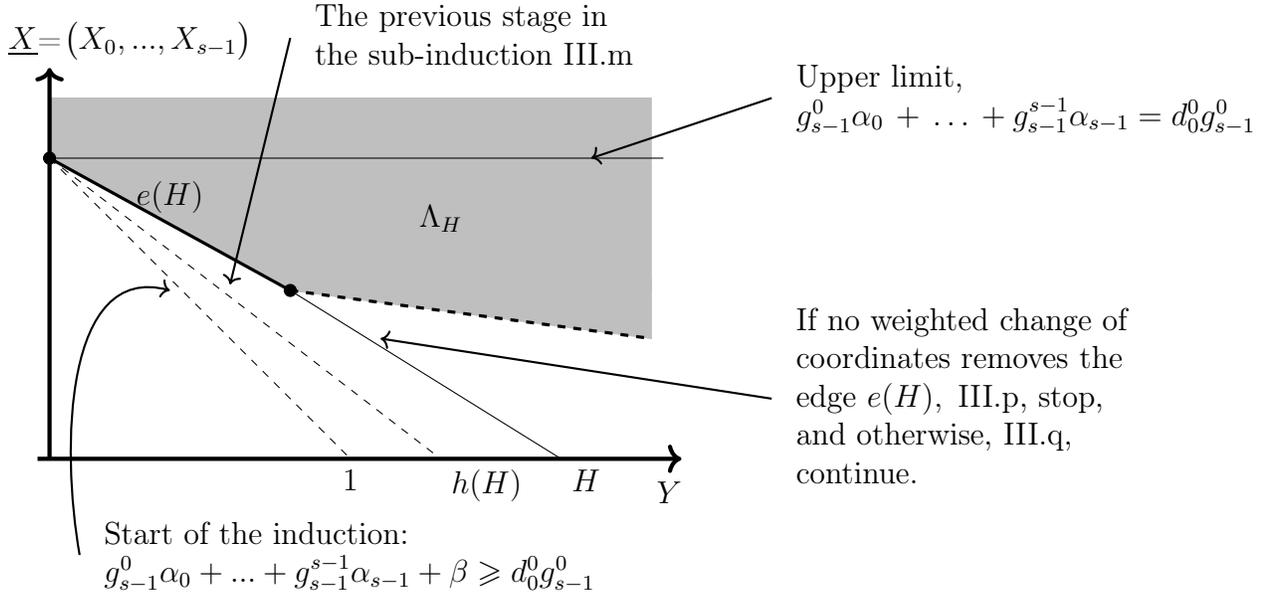

\begin{caseB}\label{case B}
	$\LL	\bigl(
			V_{s-1}^{d}(H)
			\bigr)$ 
	has maximal dimension, so its annihilator in corollary 
	\ref{cor L}, $\ZZ$, is the empty set. Nevertheless, \op 
	still applies to give new liftings, $X_s^i,\; 0 \lqs i \lqs s-1$, of the blocks $X_i$ (of \op 
	applied to $\rP(H a_{s-1}^0,..., H a_{s-1}^{s-1}, \underline{1})$), such that the sub-inductive hypothesis \ref{s-induction} is now valid for the successor of $H$ in $\Theta_{s-1}$.
	\end{caseB}
\end{step}

\begin{partial finish}\label{step3}
	In case (A) ,\ref{case A}, the sub-induction \ref{s-induction} has terminated, 
	and we have found our new filtration $F_{s}^\bullet$, to wit 
	$F_{s-1}^\bullet(H)$, so that the induction now continues in $s$.
	
	Otherwise in case (B), \ref{case B}, we either eventually fall into case (A), \ref{case A}, and, again, terminate the sub-induction, \ref{s-induction}, or we repeat case (B), 
	\ref{case B}, \emph{ad infinitum}. 
	Suppose, therefore, 	
	\begin{hypothesis}\label{ad infinitum}
	 Case \ref{case B} occurs \emph{ad infinitum}.
	\end{hypothesis}
	Such repetition is indexed by the possible $h$ in $\Theta_{s-1}$ of \ref{f/d theta} and we continue to denote by $H$ its successor.
	Our aim is to calculate the coordinates $X_{s-1}^{i}(H)$ of \ref{cor L}.(i) (whose liftings will be, again, the blocks $X_{s-1}^{i}(H)$) and, because we are in case (B), \ref{case B}, the relationship with the old coordinates $X_{s-1}^{i}(h)$ is given by:
\begin{equation} \label{eq mad1}
	X_{s-1}^{i}(H) - X_{s-1}^{i}(h) \; \in 
		\ssym^{hg_{s-1}^i} 
				\bigl( 
					X(h) \amalg 
					\HH^{0} 
					\big( 
						\P		,\,
						\OO_{\P}(\underline{1})
					\big)
				\bigr),
\end{equation}
where $X(h)$ is the space of function generated by $X_{s-1}^{i}(h),$ for every 
$ 0 \lqs i \lqs s-1,$  
and
$\P = \rP \big(
h\underline{g}_{s-1}^{0}		, ...,\, 
h\underline{g}_{s-1}^{s-1} 	,\, 
\underline{1}
\big)$, \cfr \eqref{M1 eq} \& \eqref{M1bis eq}. Now without loss of generality we have equality modulo  
$\HH^{0}
\bigl( 
	\P	,\,	\OO_{\P}(\underline{1})	
\bigr)$, 
\ie the projection of $X_{s-1}^{i}(H) - X_{s-1}^{i}(h)$ onto $\Sym^\bullet{\bigl(X(h)\bigr)}$ is always zero,
thus, $X_{s-1}^{i}(H) - X_{s-1}^i(h)$ is a combination of monomials 
\begin{equation} \label{eq mad2}
	X(h)^{E} \cdot Y^{Q}		\,,
\end{equation}

where 	$X(h)^{E}= \prod_{i} X_i(h)^{E_i}$, 
respectively $Y^{Q} = Y_{1}^{q_{1}} ... Y_{c_{s}}^{q_{c_{s}}}$, 
coming from 
$X(h)$ alone, respectively $\HH^{0} \big(\P,\, \OO_{\P}(\underline{1})\big)$ alone, and by construction
\begin{equation}\label{eq mad3}
	{h}\,g_{s-1}^i = {h}\,\wt(E) + Q,
\end{equation}
where 
$
	\wt(E) = 
	g_{s-1}^0 		\abs{E_0} + ... + 
	g_{s-1}^{s-1} 	\abs{E_{s-1}}	,
$  
$
	Q= q_1 + ... + q_{\ell_{s-1}}
$. 
Therefore, 
$Q= {h} \, ( g_{s-1}^i - \wt(E) )$ while 
$e :=  \min_{E} \{g_{s-1}^i - \wt(E) > 0 \}$ is attained since the weights of the 
${F}_{s-1}^\bullet$ filtration are a discrete set. 
Thus $Q \gqs {h}\,e$ and the right hand side of \eqref{eq mad3} tends to infinity.
Consequently the $X_{s-1}^{i}(h)$ are a Cauchy sequence in the 
$\mm$-adic topology, so if \ref{ad infinitum} were to occur,
\begin{fact/prop}\label{F/P madic}
	The filtrations $F^\bullet_{s-1}(h)$, $h \in \Theta_{s-1}$ converges $\mm$-adically as $h \rightarrow \infty$ to a filtration $F^\bullet(I)$ determined uniquely by $I$ consisting of blocks $X_{s-1}^i$ of weights $\wt(X_{s-1}^i)=g_{s-1}^i$ and cardinality $c_i$, where $m_s = m - (c_0 + ... + c_{s-1}) > 0$. 
	\end{fact/prop}
\end{partial finish}
\begin{conclusion}\label{step4}  
	Should the sub-induction, \ref{s-induction}, eventually not terminate. \ie, \ref{ad infinitum}, then we arrive to a filtration $F^\bullet(I)$ of the completion $\what{A}$ of $A$ in $\mm$ (depending only on $I$) with blocks $X_{s-1}^i$ of cardinality $c_0,...,c_{s-1}$ together with weights $g^0 > ... > g^{s-1}$, satisfying (F.1)-(F.4) of \ref{indu hypo} and we define:
	\begin{equation}
	\inv(I) = \big(d, \ell_0, g_1, \ell_1, ..., g_{s-1}, \ell_{s-1}, \underline{0} \big) \in \Q_{\gqs 0}^{2m}
	\end{equation}
	wherein the last block $\underline{0}$ has length $2(m-s)$. Otherwise, case (A), \ref{case A}, applies for all $s$ and the invariant is eventually defined by \eqref{EQS1}.
\end{conclusion}

Finally it is appropriate to explicitly observe the behaviour under regular maps beginning with,

\begin{fact}\label{WF0 fact}
	The formation of the invariant is étale local, in fact better for
	$\what{A}$ the completion of our regular local ring $A$ of 
	\ref{setup INDU}, and $\what{I}:= I \otimes_{A} \what{A}$ we have,
	
		
		(i)
		$\inv_{A}(I\,) = \inv_{\what{A}}(\what{I}\,)$;
		
		(ii)
		If $F^{\bullet}(I)$, resp. $F^{\bullet}(\what{I})$, is the filtration whether of $A$ or $\what{A}$ resulting whether from the termination of the induction, 
		\ref{indu/def INV Fs}, 
		or the sub-induction,
		\ref{s-induction},
		running
		\emph{ad infinitum}, 
		\ref{ad infinitum}, then
		\begin{equation}\label{eq WF0 F}
			F^{\bullet}(\what{I}) = 
			\begin{cases}
			\quad F^{\bullet}(I) \,, 
				& \text{ should \ref{ad infinitum} occur,}\\
			F^{\bullet}(I) \otimes_{A} \what{A}\, ,
				& \text{ otherwise.}
			\end{cases}
		\end{equation}

\end{fact}

\begin{proof}
	In the situation of the inductive hypothesis 
	\ref{indu hypo},
	\[
	\mm^N \subset F_{s-1}^{N} \quad \text{and} \quad  F_{s-1}^{p} \subset \mm^{\quotient{p}{g_{s-1}^{0}}}\,,
	\]
	so if 
	\ref{ad infinitum}
	 never occurs everything is determined modulo a sufficiently large power of the maximal ideal, and both items (i) \& (ii) are trivial. Otherwise if \ref{ad infinitum} occurs then the conclusion 
	 \ref{step4}
	 and the reasons for it 
	 \eqref{eq mad2}-\eqref{eq mad3} 
	 are $\mm$-adic by definition, so this is trivial too.
\qedhere \end{proof}

In the same vein we may prepare for replacing étale by regular via,

\begin{lemma}\label{EF1 lemma}
	Suppose $B = A	\llbracket 
						z_1, ..., z_{\defc}
					\rrbracket$  is a formal power series ring over $A$; $J$ the pull-back of $I$ to $A$ with
					$\what{A}, \what{B}, \what{I}, \what{J}$ their completions in the maximal ideal of $A$, then:
					
		(i)
		The odd entries of $\inv_B(J)$ and $\inv_A(I)$ agree.
		
		(ii)
		Even entries where the invariant is zero agree, and otherwise the difference $\inv_B(J) - \inv_A(I)$ at an even entry is $\defc$.
		
		(iii)
		The filtrations \eqref{eq WF0 F} are related by, $F^{\bullet}(\what{J})=F^{\bullet}(\what{A}) \otimes_{\what{A}} \what{B}$.
\end{lemma}

\begin{proof}
	By induction in the parameter $s$, we assert that  the relation between the graded rings $\gr_{s-1}A, \; \gr_{s-1}B$ of \eqref{eq grA} is,
	\begin{equation}\label{e1 eq}
		\gr_{s-1}B = \gr_{s-1}A \otimes_{k} k [z_1,...,z_\defc]
	\end{equation}
	while in the sub-induction \ref{s-induction}, the maximal contact spaces 
	$\LL_{B}\bigl(V_{s-1}^d(H)\bigr)$, resp. $\LL_{A}\bigl(V_{s-1}^d(H)\bigr)$ are related by,
	\begin{equation}\label{e2 eq}
		\begin{split}
			\LL_{B} \bigl(V_{s-1}^d(H)\bigr) 
				&=	
			\LL_{A} \bigl(V_{s-1}^d(H)\bigr) 
				\amalg
			k \otimes_{A} \mathrm{Der}_A (B) \\
				&=
			\LL_{A} \bigl(V_{s-1}^d(H)\bigr)
				\coprod\limits_{1 \lqs j \lqs \defc} k \, \frac{\partial}{\partial z_j}
		\end{split}
	\end{equation}
	Indeed for $s=1$, \eqref{e1 eq} is obvious, while for any $s \gqs 1$, \eqref{e1 eq} $\xRightarrow{\;\;\;}$ \eqref{e2 eq} since the $\tfrac{\partial}{\partial z_j}$ always vanish on generators of $I$ so the right hand side of \eqref{e2 eq} is always contained in the left, while modulo the $\tfrac{\partial}{\partial z_j}$ they are plainly equal. Consequently in case A of the sub-induction, \ref{case A}, \eqref{e2 eq} implies \eqref{e1 eq} for $s$, while in case B, \ref{case B}, the convergence is actually modulo the pull-back of the maximal ideal of $A$, equivalently the filtration is pulled back from $\what{A}$.
\qedhere \end{proof}

\section{The Invariant on Weighted Projective Champ}\label{sec III INV}
\setcounter{equation}{0}
\begin{setup}\label{setup INV}
	Let $\P=\PP(\ub^0,...,\ub^s)$ be a $(m-1)$-dimensional weighted projective champ, with blocks of coordinates $X_0,...,X_s$ of weights 
	$\ab^0 > ... > \ab^s$ and cardinality $c_0,...,c_s$ over a field $k$ of characteristic zero.
	Suppose further that $d \in \Z_{> 0}$ and 
	$V \subset \HH^{0}\bigl(\P,\OO_{\P}(d\ab^0)\bigr)$ is a space of sections such that:
	\begin{hypothesis}\label{Hypo H2}
		If for every $s \gqs i > 0$, $\RP_i \hookrightarrow \P$ is the weighted projective sub-champ defined by $X_i=...=X_s=0$, with for 
		convenience of notation $\RP_{s+1}=\P$, then
		\begin{equation}
			\LL_{i}(V):= \coprod\nolimits_{-b < 0} 
			\Big\lbrace 
				\partial \in \HH^{0} (\RP_i, \T_{\RP_i}(-b)) \, \Big| \, \partial(V_{i})=0
			\Big\rbrace = 0
		\end{equation}
		where $V_{i}$ is the image of $V$ in 
		$\HH^{0}\big( \RP_i, \OO_{\RP_i}(d\ab^0)\big)$.
	\end{hypothesis}
\end{setup}
	
Now for consistency with \ref{f/d theta} and 
\ref{indu hypo}.(F.4), define $g_i := \quotient{\ab^{i-1}}{\ab^{i}}$, $ 1 \lqs i \lqs s$, and $\ell_{i}= m - (c_0 +... + c_i)$ then we assert,

\begin{proposition}\label{prop INV gdown}
	If $I$ is the sheaf of ideals generated by $V$, under the non-degeneracy condition \ref{Hypo H2}, then for every geometric point $p$ of $\P$ the value of the invariant $\inv_{\P}(I)(p)$ at the stalk $I_p$ is \emph{strictly} less than
	\begin{equation}\label{eq INV gdown}
		\big( \,d, \ell_0, g_1, \ell_1, ..., g_{s},  \ell_{s}, \underline{0} \,\big) .
	\end{equation}
	More precisely, if 
	$\inv_{\P}(I)(p) = 
		\big(	
			\mult_I(p)	, 
			\ell_0(p)	, 
			g_0(p)		,...,
			\ell_{s}	,
			\underline{0} \,
		\big)$ 
	with $\ell_i{(p)} = m - (c_{0}(p) + ... + c_{i}(p))$ and $ 0 \lqs \si \lqs s$ is maximal such that $X_{\si}(p) \neq 0$, (\ie there is some $1 \lqs i \lqs c_{\si}$, for which $x_{\si i}(p) \neq 0$) then:
	
		(i)
		If $\si = 0$ the multiplicity of $I$ at $p$ is strictly less than $d$, unless $d=0$.
		
		(ii)
		If $ \si > 0$ with, for immediate notational convenience, $g_0=d$ and all of $g_{i}(p) \gqs g_{i}$, $c_i(p) \lqs c_i$, for any $ 0 \lqs i \lqs \si -2$ then $g_i(p)=g_i$ and $c_i(p)=c_i$ for all $0 \lqs i \lqs \si -2$.
	
		(iii)
		If (ii) holds and $g_{\si-1}(p) \gqs g_{\si-1}$, $c_{\si -1 }(p) \lqs c_{\si -1}$, then $g_{\si-1}(p) = g_{\si-1}$, $c_{\si -1 }(p)=c_{\si -1}$, $c_{\si} \gqs 2$, and $g_{\si}(p) < g_{\si}$; so in particular if $c_{\si} = 1$ then $g_{\si-1}(p) < g_{\si-1}$, \ie $g_{\si-1}(p)$ goes down.
	\end{proposition}
Observe that we can immediately reduce to $\si = s$ since,
\begin{lemma}\label{lem sub-champ}
	Let $\QQ$ be a sub-champ of $\P$ containing the geometric point $p$ and such that \ref{prop INV gdown}.(i) holds, for $\restr{I}{\QQ}$, while denoting by a superscript $\QQ$ the values of the blocks associated to the invariant of $\restr{I}{\QQ}$ calculated at $p$, items (ii) \& (iii) of \op hold, albeit, in the modified form:
		
		(ii-bis)
		If $\si > 0$, 
			$g_i(p) \gqs g_i$, 
			$c_i^{\QQ}(p) \lqs c_i$,
		for any $ 0 \lqs i \lqs \si -2$, then
			$g_i(p) = g_i$, 
			$c_i^{\QQ}(p) = c_i$,
			for any $ 0 \lqs i \lqs \si -2$. 
		
		(iii-bis)
		If (ii-bis) holds and 
		$g_{\si-1}(p) \gqs g_{\si-1}$, 
		$c_{\si -1 }^{\QQ}(p) \lqs c_{\si -1}$, then 
		$g_{\si-1}(p) = g_{\si-1}$, 
		$c_{\si -1 }^{\QQ}(p)=c_{\si -1}$, 
		$c_{\si} \gqs 2$, and 
		$g_{\si}(p) < g_{\si}$; 
		so in particular if $c_{\si} = 1$ then 
		$g_{\si-1}(p) < g_{\si-1}$.
\end{lemma}
\begin{proof}
	For the multiplicity $d=g_0$ this is clear, while $c_0$ is the minimum number of coordinates required to describe the ideal modulo $\mm^{d+1}(p)$, so its ambient value $c_0(p)$ is always at least that, $c_0^\QQ(p)$, of a subspace whenever the multiplicity of the intersection coincides. Consequently if
	\begin{equation}\label{G1 eq}
	c_0 \gqs c_0(p) 
	\text{ and } 
	\bigl( \, 
		c_0 \gqs c_0^\QQ(p) \Longrightarrow c_0^\QQ(p) = c_0 
			\, \bigr) 
	\text{ then } 
	c_0(p) = c_0. 
	\end{equation}
	Similarly the presence of a non-zero gradient $g_r$, $ 1 \lqs r \lqs \si$ reflects the necessity, or otherwise, \ref{cor L}, of a new block of coordinates to describe the leading monomials in generators of the ideal, so if one needs a block after intersecting with a sub-widget one certainly needed it before hand, and should this occur $c_i^\QQ(p) = c_i$ will imply $c_i(p)=c_i$ exactly as in \eqref{G1 eq}.
\qedhere \end{proof}

In particular, therefore, after \ref{lem sub-champ}, and the definition of $\si$ it is sufficient to prove \ref{prop INV gdown} on the subspace $X_{\si+1}=...=X_{s}=0$, so \wloss $\si =s$.

\begin{proof}[Proof of Proposition \ref{prop INV gdown}.]
	We proceed by induction on the number of blocks, $s$, starting from $s=\si=0$. In this case by the action of $\mathrm{PGL}_{c_0}$
	we may, without loss of generality suppose $p$ is the point $[1:0:...:0] \in \PP_k^{m-1}$, in some basis $ \{x_1,...,x_m\}$.
	Consequently if the multiplicity does not go down $Z$ of \ref{cor L} is contained in the subspace generated by $x_2,...,x_m$ which contradicts the definition of $\ell_{0}$ (\ie 0 under the present hypothesis) in \ref{start INDUCT} unless $d$ were already 0.
	
	Supposing, therefore, that $\si = s >0$ let us adjust the notation accordingly by denoting the final block $X_s$ as $Y$
	which in turn is a basis of $\HH^{0}(\P,\OO_{\P}(\ab^s))$, which we write as $Y = \{y \} \cup Z$ where
	\begin{equation}
		y(p)=1, \qquad z(p)=0, \; \forall \, z \in Z .
	\end{equation}
%
	In particular, therefore, we have an étale neighbourhood $U$ of $p$ obtaine    
by slicing the groupoid
	$
	R:=\mathbb{G}_m\times\A^m\minusO 
	\rightrightarrows\A^m\minusO
	$
	along the transversal $y=1$, and we write the coordinate functions on $U$ afforded by the elements of the blocks $X_i$ as $x_{ij}+p_{ij}$, $0\lqs i\lqs s-1$, $1\lqs j\lqs c_i$, {\it i.e.}
	\begin{equation}
		U \ni p = \prod_{i= 0}^{s-1} \underline{p_t} \times 1 \times \underline 0, 
			\quad \text{where } \underline{p_t} = p_{t1} \times ... \times p_{tc_t}.
	\end{equation}
	In this notation the correspondence between a global section, 
	$f(X_0, ..., X_{s-1}, Y) $ in 
	$\ssym^{d\ab^0} \bigl(
						X_0 \amalg ... X_{s-1} \amalg Y
					\bigr) = 
	\HH^{0} \bigl( 
				\P, \OO_{\P}(d\ab^0) 
			\bigr)$ and the associated function is simply
	\begin{equation}
	\begin{split}
		f \xmapsto{\quad}
			 f(	x_{ij}+p_{ij},\,	1,\,	\underline{z} ) 
			\in \Gamma(U, \OO_{\P}), \quad
		\text{for $ 0 \lqs i \lqs s-1$ and $ 1 \lqs j \lqs c_{i}$.}
	\end{split}
	\end{equation}
	Furthermore, and needless to say, $U$ is an affine space with origin $p$ via,
	\begin{equation}
	\bigg( 
		\prod_{i=0}^{s-1} \,\prod_{j=1}^{c_{s-1}} x_{ij} \bigg)	
	\times \underline{z} =  
	\begin{tikzcd}
		U				\ar[r]
			&\A^{m-1} .
	\end{tikzcd}
	\end{equation}
	so it makes perfect sense to talk about the maximal degree in the blocks of functions $\ux_{t} := \{ x_{t i} \, \big| \, 1 \lqs i \lqs c_{t} \}$, 
	$ 0 \lqs t \lqs s-1$. With this in mind we assert,
	\begin{claim}\label{claim 1inv}
		The initial $2s$-part of the invariant 
		$(g_0, \ell_0, g_1, \ell_1, ..., g_{s-1},\ell_{s-1})$ cannot increase.
		\end{claim}
	\begin{proof}
		By induction in $s$. The starting point of the multiplicity $d = g_0$ is particular. 
		Modulo the local functions $x_{ij},\, i \gqs 1,\, \underline{z}$, at $p$ we have an 
		affine space $\A^{c_0}$ on which the multiplicity is at most the degree in the block of 
		functions $\ux_0$ which is at most the degree in global block $X_0$, \ie $d$. 
		Furthermore were this bound to be achieved on $U$ then the restriction $I$ to 
		$\A^{c_0}$ at $p$ is, under the isomorphism afforded by:
		\begin{tikzcd}[cramped, sep=small]
			X_{\bullet j} \ar[r, mapsto]
				& \ux_{\bullet j},
		\end{tikzcd}
		exactly the ideal generated under, 
		\begin{equation}
			\begin{tikzcd}[sep=2cm]
			\Gamma(\A^m \minusO) = \Gamma(\A^{m}) 
			\ar[r, "\mod X_{i}"]
					&\Gamma(\A^{c_0}), \quad i \gqs 1
			\end{tikzcd}
		\end{equation}
		at the origin, so $c_{0}(p) \gqs c_0$.
		
		Now we put ourselves in the scenario of the inductive hypothesis 
		\ref{start INDUCT}.(F.0)-(F.4), albeit with an inductive parameter 
		$ 0 \lqs t \lqs s-1$, rather than $s-1$ of \op, and we add to the 
		hypothesis:
		
				(F.4 bis) The $i^{\th}$-block, $ 0 \lqs i \lqs t$, is defined by the block of functions 
				$\ux_i$ and has weight $\quotient{\ab^i}{\ab^t} = g_t^i$ (in notation of 
				\ref{indu hypo}.(F.4) .
				
	Quite possibly we arrive in case (A), \ref{case A}, for a value of 
	$H < \quotient{\ab^t}{\ab^{t+1}}$, 
	but, plainly should this occur then the invariant strictly decreases. If, however, we 
	were to continue in case (B), \ref{case B}, for every $H < \quotient{\ab^t}{\ab^{t+1}}$ by way of 
	changes of coordinates in the blocks $\ux_i$, $ 0 \lqs i \lqs t$, then this in no way changes monomials of the form
	\begin{equation}
		\ux_0^{D_{0}} \cdots \ux_{t+1}^{D_{t+1}}, 
		\quad \ab^0 \abs{D_{0}} + ... + \ab^{t+1} \abs{D_{t+1}} = \ab^0 d
	\end{equation}
	since the weight of the perturbation in $\ux_i$ will be
	\begin{equation}
		H 	\cdot 	\bigl(	\quoziente{\ab^i}{\ab^t}	\bigr) < 
			\quoziente{\ab^i}{\ab^{t+1}}. 
	\end{equation}
	Consequently were we to eliminate all 
	$H < \quotient{\ab^t}{\ab^{t+1}}$, modulo $\ux_i, \; i > t+1$ we would find that,$\mod \ux_i,\; i > t+1,$ the ideal at $p$ is exactly that generated at the origin by the image of $V$ in the origin obtained via the isomorphism
	\begin{equation}
		\begin{tikzcd}[sep=2cm]
		\Gamma(\A^m \minusO) = \Gamma(\A^{m}) 
		\ar[r, "\mod X_{i}"', "\sim"]
		&\Gamma(\A^{c_0 + ... + c_{t+1}}), \quad i > t+1 ;
		\end{tikzcd}
	\end{equation}
	so the claim follows from \ref{cor L}, as employed in the definition of the invariant in case (A), \ref{case A} .
	\qedhere \end{proof}
Suppose therefore that the extremal situation of \ref{claim 1inv} is attained (\ie the invariant did not decrease), then from our original blocks of coordinates, $\ux_i, \; 0 \lqs i \lqs s-1, \; \uz$ we will have performed a change of coordinates to blocks of the form
\begin{equation}\label{eq 2}
\begin{split}
		\uxi_0 &=\; \ux_0 + \ueps_0 (\ux_1, ..., \ux_{s-1}, \uz), \\
		\uxi_1 &=\; \ux_1 + \ueps_1 (\ux_2, ..., \ux_{s-1}, \uz), 
				\qquad \wt_{\ux}(\ueps_{i}) < \ab^i,  \\
		\quad  & \vdots \qquad \qquad \vdots 
				 \qquad \qquad \qquad \quad \text{ for } \wt(x_i) = \ab^i;	\\
		\uxi_{s-1} &= \; \ux_{s-1} + \ueps_{s-1} (\uz);
	\end{split}
\end{equation}
resulting in a filtration $F^\bullet_{\uxi}$ around $p$ in which the blocks $\uxi_{i}$, $0 \lqs i \lqs s-1$ have weights $\quotient{\ab^i}{\ab^s}$, $\uz$ has weight $1$, and around $p$ the ideal generated by $V$ belongs to $F_{\uxi}^{\quotient{\ab^0d}{\ab^s}}$. In particular
\begin{warning}\label{Warn 2bis}
We allow the possibility that the sub-induction \ref{indu/def INV Fs} may still not have terminated in case \ref{case A} and whence the invariant might even go up.
\end{warning}
To analyse this situation we replace the coordinates $x_{ij}$ around $p$ by the 
restriction to $U$ of the $\G_{m}$-equivariant global coordinate functions $X_{ij}$, $0 \lqs i \lqs s-1,$ $1 \lqs j \lqs c_i$ 
in the various block, so that \eqref{eq 2} becomes,
\begin{equation}\label{eq 3}
	\begin{split}
	\uxi_0 &= \restr{
				\bigl(
					\uX_0 - \ue_0 (X_1, ..., X_{s-1}, Z)
				\bigr)
					}{U} \;, 
	\\
	\uxi_1 &= \restr{
				\bigl(
					\uX_1 - \ue_1 (X_2, ..., X_{s-1}, Z) 
				\bigr)
					}{U} \;,
	 \\
	\quad \vdots& \qquad \qquad \vdots \qquad 
	\\
	\uxi_{s-1} &= \restr{\bigl(\uX_{s-1} - \ue_{s-1} (Z)
		\bigr)}{U} \; ;
	\end{split}
	\qquad \wt_X (\ue_{i}) < \ab^i,
\end{equation}
and we assert
\begin{claim}\label{claim 2inv}
	In the above notation and under the hypothesis (\cfr claim \ref{claim 1inv})
	that the first $2s$ terms in the invariant at $p$ are at least 
	$\bigl( d, \ell_0, g_1, \ell_1, ..., g_{s-1}, \ell_{s-1} \bigr)$
	 the coordinate change \eqref{eq 3} is global, \ie there are homogeneous functions
	  $\uEE_i$ on $\A_{k}^{m-1}$ of weight $\ab^i$ such that,
	\begin{equation}\label{eq claim2}
	\restr{\ue_{i}(X_{i+1}, ..., X_{s+1}, Z)}{U} = \uEE_i(X_{i+1}, ..., X_{s+1}, Z).
	\end{equation}
\end{claim}
\begin{proof}
	We have filtrations in which the blocks $X_i$, $0 \lqs i \lqs s-1,$ $ X_s = \{Z,\, Y\}$, respectively $\uxi_i,\; \uz$, with weights 
	$a^i$, $0 \lqs i \lqs s-1, \; \ab^s$, may \emph{a priori }be different and so we will 
	employ the notation $\wtx$, resp. $\wtxi$, to avoid ambiguity. 
	In any case for $\ff \in \V_d$, we have from \eqref{eq 3}:
	\begin{equation}\label{eq eitop1}
		\begin{split}
		\restr{\ff}{U} = 
		\restr{\ff(X_0,..., X_{s-1},1,Z)}{U} = 
		\ff		\big( 
					\uxi_{0} 	+ 	\ue_{0} , ..., 
					\uxi_{s-1} 	+ 	\ue_{s-1}, 
					\underline{1}, 	\uz	
				\big) = \\
		\ff		\big(
					\uxi_{0} 		, ..., 
					\uxi_{s-1}		, 
					\underline{1}	, 
					\uz	
				\big) + 
		\sum_{i=0}^{s-1}	\Big( 
								\frac{\partial\ff}{\partial{X_i}}\,
								\ue_i
							\Big)
				\big(
					\uxi_{0} 		, ..., 
					\uxi_{s-1}		, 
					\underline{1}	, 
					\uz	
					\big) + 
		\text{ stuff,}
		\end{split}
	\end{equation}
	wherein
	$\frac{\partial \ff}{\partial_{X_i} }\,\ue_i 
		= \sum_{j=1}^{c_i} 
		\frac{\partial \ff}{\partial_{x_{ij}} }\,\varepsilon_{ij} $  and  stuff has smaller weight in the $\uxi$-filtration than the expected top weight in
\begin{equation}\label{eq 4}
	\left( \sum_{i=0}^{s-1} \frac{\partial \ff}{\partial{X_i}} \ueitop \right)
	\bigl(\uxi_{0} , ..., \uxi_{s-1}, \underline{1}, \uz	\bigr) 
\end{equation}
to wit: $(d \ab^0) - \min_{0 \lqs i \lqs s-1} 
		\left\lbrace
		\ab^i - \wtxi \bigl( \ueitop \bigr)
		\right\rbrace,$ 
where $\ueitop$ are the monomials in $\uxi, \uz$ in $\ue_i$ which have maximal $\uxi$-weight, 
\begin{equation}\label{eq eitop2}
\ueitop := \sum_{ D  }
				\lambda_{D} 	\; 
				\uxi_0^{D_0} 	\cdots 
				\uxi_{s-1}^{D_{s-1}} 
				\uz^{D_s} 	+ 
				\text{ stuff,}
\end{equation}
where, again, stuff is monomials with lower $\xi$-weight. Let us therefore define homogenous functions on the ambient space, $\A_k^{m-1}$ by way of the formula:
\begin{equation}\label{eq 5}
	\FF_i := 
	\sum\nolimits_{ D_i } 
			\lambda_{D_i} 		\; 
			X_0^{D_0} 			\cdots 
			X_{s-1}^{D_{s-1}} 
			Z^{D_s} ,
\end{equation}
and a homogeneous vector field,
\begin{equation}\label{eq 6}
	\De = \sum_{i=0}^{s-1} 
			\FF_i \, \frac{\partial}{\partial{X_i}} \;
				\quad \text{ of } \; 
				\wtx(\De) = -\min_{0 \lqs i \lqs s-1} 
				\{ \, \ab^i - \wtxi(\ueitop) \,\}.
\end{equation}
So that by construction and \eqref{eq 3}, \eqref{eq 4} vanishes if and only if the top weight term in the grading of $\Gamma(\OO_{U})$ which assigns to $\restr{X_i}{U}$ weight $\ab^i$, $0 \lqs i \lqs s-1$, and to $\restr{Z}{U}$ weight $\ab^s$ of every $\restr{\De(\ff)}{U}$ vanishes for every $\ff \in \V_d$. 
Thus, \emph{ a fortiori,} on the weighted projective hypersurface $\QQ$, defined by the function $Y=0$,
\begin{equation}
	\De(\ff) = 0 \mod Y,  \qquad \forall \, \ff \in \V_d.
\end{equation}
As such there are two cases: either $Z \neq \emptyset$, then since $\De$ acts trivially on $\HH^{0} \bigl( \QQ, \OO_{\QQ}(\ab^s) \bigr)$ by \eqref{eq 6},  $\De = 0 \mod Y$ by \ref{Hypo H2} and 
\ref{cor L}.(ii); or $Z = \emptyset$ and $\De = 0 \mod Y$ by the non-degeneracy hypothesis \ref{Hypo H2} and \ref{cor L}.(ii). In either case $\De = 0 \mod Y$, and whence all the $\FF_i \equiv 0$ by virtue of their definition \eqref{eq 5}, which in turn is nonsense (unless claim \ref{claim 2inv} is true with $\ue_i = \uEE_i = 0$). Thus the top weight term in \eqref{eq 4} is not zero for some $\ff \in \V_d$. However for such a $\ff$, according to our hypothesis that the invariant does not decrease, the top $\uxi-$weight term in \eqref{eq 4} must cancel with the top $\uxi-$weight of
\begin{equation}
	\ff \bigl( \uxi_{0}, ..., \uxi_{s-1}, \underline{1}, \uz \bigr) \mod F_{\uxi}^{\ab^0\,d},
\end{equation}
which in turn has weight, $\ab^0d - \ab^s n$, for some integer $n$. We therefore conclude, 
\begin{equation}\label{eq 7}
	\ab^0 \, d - \ab^s\,n = \ab^0\,d - \min_{0 \lqs i \lqs s-1} \{ \, \ab^i - \wtx(\FF_i) \,\} ,
\end{equation}
\ie for $ 0 \lqs i \lqs s-1$ where the minimum in \eqref{eq 7} is attained,
\begin{equation}
\ab^i = \wtx(\FF_i) + \ab^s\,n \; .
\end{equation}
	
Now consider the change of variables on $\PP(\underline{\ab}^0, ..., \underline{\ab}^s)$ defined by,
\begin{equation}
	X_{i, \text{new}} := \underline{X}_i + Y^n\, \FF_i(\underline{X}_{\gqs i+1}, \underline{Z}), \quad 
	0 \lqs i \lqs s-1,
\end{equation}
then in the \emph{new} coordinates the invariant, $\min_{0 \lqs i \lqs s-1} \{\, \ab^i - \wtxi(\uettop)\,\}$, 
of the coordinate change \eqref{eq 3} has increased and since it is an integer which is at most $\ab^0$ (\cfr \ref{f/d theta}), this process eventually terminates establishing the claim.
\qedhere \end{proof}
The practical upshot of \ref{claim 2inv} is when we come 
to compute the invariant at $p$ we can suppose 
not only that all the $p_{ij}$ are zero for 
$0 \lqs i \lqs s-1$, but that the filtration defined by 
$\wt \bigl( \restr{X_{ij}}{U}\bigr) = \quotient{\ab^i}{\ab^s}$, $\wt(\restr{Z}{U})=1$ is exactly that defined by
the inductive procedure \ref{indu/def INV Fs}, 
albeit for the moment we remain in the situation 
\ref{Warn 2bis}. However by claim \ref{claim 2inv} 
we can now just read the invariant at $p$ from the 
newton polyhedron, \cfr figure \ref{fig:NP} pg. \pageref{fig:NP}, calculated in the 
coordinates $\restr{X_{ij}}{U},\,\restr{Z}{U}$. 
As such if $Z = \emptyset$ then at worst $g_{s-1}$ 
goes down, whereas if $Z \neq \emptyset$ at worst 
$g_s$ must go down.
\qedhere \end{proof}

\section{The Relative Invariant}
\setcounter{equation}{0}
We proceed to construct the invariant relatively in a generality which is adequate for applications but only coincides with \S \ref{sec II INV} for complete local rings, to wit:

\begin{setup/notation}\label{S1 not/defn}
	Let 
	\begin{tikzcd}[cramped, sep=small]
		\pi: \UU = \mathrm{Spf}\,A \ar[r]
			&\BB = \spec k
	\end{tikzcd}
	be a map from an affine formal scheme to a Noetherian affine scheme, and suppose that the trace of $\UU$ is a regularly embedded section $\sigma$ of $\pi$ of co-dimension $m$. Furthermore if $M$ is the ideal of $\sigma$, suppose $\quotient{M}{M^2}$ is trivial, \ie $M = (x_1,...,x_m)$ is 
	the ideal of $\sigma$ (so 
	\begin{tikzcd}[cramped, sep=small]
		A
			&k \llbracket x_1, ..., x_m \rrbracket
			\ar[l, "\sim"'] )
	\end{tikzcd}
	and let $I$ be an other ideal of $\UU$ (so $M$-adically separated by definition), while for objects, over $\BB$, denote by a subscript in $\bb$ the fibre (as a formal scheme, \ie $M$-adically complete tensor product) over $\bb \in \BB$.
\end{setup/notation}

Plainly we begin with the multiplicity, \ie

\begin{fact}\label{F1 fact}
	For $\bb \in \BB$, define $d_{\bb}(I) \in \Z_{\gqs0} \cup \{\infty \}$ by,
	\[
	d_{\bb}(I) := 
	\sup \left\lbrace 
		\x \in \Z_{\gqs 0} \, \big| \, M_{\bb}^{\x} \supset I_{\bb}
	\right\rbrace;
	\]
	then $\bb \xmapsto{\; \; \; } d_{\bb}(I)$ is upper semi-continuous (often abbreviated to \usc ).
\end{fact}

\begin{proof}
	Since $I$ is $M$-adically separated, it is either zero and $d_{\bb}(I)$ is identically $\infty$, or there is a smallest $e \in \Z_{\gqs 0}$ such that $I \subset M^{e}$. The former case is trivial, while in the latter case we have a non-trivial quotient of a free module, \ie
	\begin{equation}\label{W0 eq}
		\begin{tikzcd}
		I \ar[r] 
			& \quoziente{M^{e}}{M^{e+1}}	\ar[r]
				&Q \ar[r]
					&0
		\end{tikzcd}
	\end{equation}
	and the condition $d_{\bb}(I) \gqs e+1$ is equally the non-trivial
	closed condition,
	\begin{equation}
		\dim_{k(\bb)} Q_{\bb} \gqs \mathrm{rank} \big(\quoziente{M^{e}}{M^{e+1}}\big)
	\end{equation}
	so we conclude by Noetherian induction.
\qedhere \end{proof}

Next we proceed to the maximal contact space by way of

\begin{fact}\label{F2 fact}
	Suppose the multiplicity $d_{\bb}$ is identically $d 	\in \Z_{\gqs 0}$ and define the sub-module $V$ in $\quotient{M^{d}}{M^{d+1}}$ 
	to be $I$ modulo $M^{d+1}$, then the following is \usc,
	\begin{equation}\label{1 eq}
		\begin{tikzcd}[sep=0.6cm]
		b \ar[r, mapsto]
			&\lambda_{0}(b) := 
		\end{tikzcd}
		\begin{cases}
			\dim_{k(\bb)} 
			\big\lbrace
				\partial \in 	
					\big(
						\quoziente{M}{M^2} 
						\otimes k(\bb) 
					\big)^{\vee} 
				\, \big| \,
				\partial(V_{\bb}) = 0
			\big\rbrace,			& d > 0,  \\
			0\, ,					& d = 0,
		\end{cases}
	\end{equation}
	\end{fact}
\begin{proof}
	Plainly, without loss of generality $d > 0$, while the action of 
	$\left(\quotient{M}{M^2} \right)^{\vee}$ by derivations affords a pairing,
	\begin{equation}
		\begin{tikzcd}[column sep=22pt,]
		V \otimes_{k} 
		\left( 
			\quoziente{M^{d-1}}{M^{d}} 
		\right)^{\vee}
		\ar[r]
			& \quoziente{M}{M^2} 
			\ar[r, phantom, ":"]
					&F \otimes \varphi \ar[r, mapsto]
							&\lbrace
								\partial \mapsto \varphi (\partial F) 
					 		\rbrace
		\end{tikzcd}
	\end{equation}
	whose image is a $k$-submodule,
	\begin{equation}\label{2 eq}
	\begin{tikzcd}
		\Lambda' \ar[r, hook]
			&\quoziente{M}{M^2}
	\end{tikzcd}
	\end{equation}
	such that the $k(\bb)$-vector spaces \eqref{1 eq} are the annihilators of the image of $\Lambda'_{\bb}$, so equivalently,
	\begin{equation}\label{W00 eq}
	\lambda_{0}(\bb)= \dim_{k(\bb)} \Lambda'' 
	\end{equation}
	where $\Lambda''$ is the quotient of \eqref{2 eq}.
\qedhere \end{proof}

Prior to the inductive definition of the relative invariant let us make a,
\begin{warning}\label{R1 remark}
	In practice one wishes to take $\UU$ to be the completion in 
	the diagonal of the product of $\BB$ with itself whenever the 
	latter is smooth over a field. 
	In such a scenario if $\bb \in \BB$, then $m$ in the sense of 
	\S\ref{sec II INV} for the local ring $\BB_{\bb}$ will be its 
	dimension, $m(\bb)$, which will only coincide with the ambient 
	dimension $m$ in the sense of \ref{S1 not/defn} if $\bb$ is closed.
\end{warning}

In any case if in addition
\begin{tikzcd}[cramped, sep=small]
	b \ar[r, hook]	&\lambda_0(\bb)
\end{tikzcd}
is constant on $\BB$ then generalising \ref{def blocks},

\begin{fact/defn}\label{D1 f/d}
	In the situation of the setup \ref{S1 not/defn}, a block of (relative, should there be danger of confusion) coordinates is a subset $X \subset M$ 
	of regular parameters whose image modulo $M^2$ is
	a subset of a $k$-basis. In particular whenever 
	\begin{tikzcd}[cramped, sep=small]
	b \ar[r, hook]	&\lambda_0(\bb)
	\end{tikzcd}
	is constant we have, possibly at the price of shrinking $\BB$ to ensure that the implied free $k$-module is trivial, \cfr hypothesis in \ref{cor L}, a block $X_0$ consisting of the lifting of \eqref{2 eq}, and of course, modulo the warning \ref{R1 remark},
	\begin{equation}\label{3 eq}
		\lambda_0(\bb) := m - c_0 .
	\end{equation}
\end{fact/defn}

\begin{indu hypo}\label{I1 ind hypo}
	Exactly as in \ref{indu hypo}, with exactly the same notation up to the following minor observations consistent with \ref{R1 remark}, 
	
	(MO.1) $I$ is to be understood in the sense of \ref{S1 not/defn}.

	(MO.2) The definition, \cfr \eqref{3 eq}, of $\lambda_i$, $0 \lqs i \lqs s-1$ is exactly as for the $\ell_i$ in \eqref{eq m-c=ell} but in light of the warning \ref{R1 remark} we will change the notation.
		
	(MO.3) By the definition of a relative block the graded algebra of the filtration has graded pieces free $k$-modules, and after clearing denominators to integers $a^0 > ... > a^{s-1} > a^{s}$, without common factors, defines a family, in the notation of \eqref{M1 eq}, $\rP(\ug, 1):= \P$ of relative weighted projective champs.

	(MO.4) The starting point/initial block is $X_0$ of \ref{D1 f/d} under the hypothesis that the functions $d(\bb)$ and $\lambda_0(\bb)$ of \ref{indu hypo}.(F.1)-(F.2) are identically constant and $\BB$ is sufficiently small to guarantee the triviality of $\Lambda''$ in \ref{F1 fact}
\end{indu hypo}

To which we must again adjoin,

\begin{s-induction}\label{SI1 s-induction}
	Define the set of sub-inductive parameters $\Theta_{s-1}$ exactly as in 
	\ref{f/d theta}, and for $H \in \Theta_{s-1}$ we suppose the sub-inductive hypothesis \ref{s-induction} under which we will say that $g_s(\bb) \gqs H$, $\forall\, \bb \in \BB$.
\end{s-induction}

With this in mind, we have

\begin{observation/defn}\label{O1 obs/d}
	We have a filtration $F_{s-1}^{\bullet}(H)$ 
	defined as in 
	\eqref{eq W-D} which for exactly the same reason, 
	\ref{lem W-D}, is independent of any choices and $V_{s-1}^{d}(H)$ is defined exactly as in 
	\eqref{eq VF}. Finally by way of notation let $\Dl$ be the global vector fields on the associated weighted projective champ, $\P$, of \ref{I1 ind hypo}.(MO.3) \ie 
	\begin{equation}\label{eq DELTA}
		\Dl := \coprod_{-n < 0} 
			\HH^{0}
			\bigl(
				\P, \T_{\P}(-n)
			\bigr) ,
	\end{equation}
	which in turn is a free $k$-module by the generalisation, \cite[I.c.3]{mcqSS}, of Serre's explicit calculation.
\end{observation/defn}

At this juncture \ref{F2 fact} easily generalises to,

\begin{fact}\label{F3 fact}
	Let everything be as in the sub-induction \ref{SI1 s-induction} so in particular $m_s := m - (c_0 + ... + c_{s-1}) > 0$, then the following function is \usc,
	\begin{equation}\label{4 eq}
		\begin{tikzcd}
		b \ar[r, mapsto]	
			&\lambda_{s}^{H}(\bb) 	:=
			\dim_{k(\bb)}
			\Big\lbrace
				\partial \in \Dl_{\bb} 		\, \big| \,
				\partial 	\bigl(
							V_{s-1}^{d}(H) 	\otimes k(\bb)
							\bigr) = 0
			\Big\rbrace
		\end{tikzcd}
	\end{equation}
\end{fact}

\begin{proof}
	As in the proof of \ref{F2 fact}, derivation gives a pairing,
	\begin{equation}\label{5 eq}
		V \otimes_{k} 
			\coprod\limits_{-n < 0} 
			\HH^{0}
			\bigl(
				\P, \OO_{\P}(da^{0}-n)
			\bigr)^{\vee}
			\rightarrow
				\Dl^{\vee} : 
		F \otimes
			\varphi	
			\mapsto
				\lbrace \partial \mapsto \varphi(\partial F) \rbrace,
	\end{equation}
	whose image $\Lambda'$ affords a short exact sequence of $k$-modules,
	\begin{equation}
	\begin{tikzcd}
				&\Lambda' 
				\ar[r]
					&\Dl^{\vee} 
					\ar[r]
							&\Lambda'' \ar[r]
							&0
	\end{tikzcd}
	\end{equation}
	such that the $k(\bb)$-vector spaces in \eqref{4 eq} are the annihilators of the image of $\Lambda'$, while the fibre dimensions, 
	\begin{equation}\label{W000 eq}
		\lambda_{s}^{H} (\bb) = \dim_{k(\bb)} \Lambda'' \otimes k(\bb).
	\end{equation}
	are plainly \usc .
\qedhere \end{proof}

From which we have the corollary,

\begin{corollary}\label{C1 corollary}
	Under the sub-inductive hypothesis \ref{SI1 s-induction},
	let $H' \in \Theta_{s-1}$ be the successor of $H$ and define, $g_{s}(\bb) > H$ to mean $g_{s}(\bb) \gqs H'$ and $g_{s}(\bb)=H$ its complement then,
		
		(i)
		the conditions $g_{s}(\bb)=H$, resp. $g_{s}(\bb) > H$, are open, resp. closed.

		(ii)
		On the open set of $\bb \in \BB$ such that $g_{s}(\bb)=H$ the function $\lambda_{s}^{H}$ is \usc 
\end{corollary}

Equally we have the relative version of the termination of the sub-induction, \ie

\begin{caseA}[Relative, \cfr \ref{case A}] \label{CA relative}
	At $\bb \in \BB$, $g_{s}(\bb)=H$ 
	(say $\BB'$, by way of notation, for the open in 
	\ref{C1 corollary}.(ii)) then we define a function $g_{s}$ to take the value $H$ at $\bb$, and define, $\lambda_{s}(\bb)$ to be $\lambda_{s}^{H}(\bb)$ of \eqref{4 eq}. 
	Now replace $\BB'$ by the constructible subset of $\bb \in \BB'$ on which $g_{s}(\bb) = H,$ and $\lambda_{s}(\bb)$ takes  the constant value  $m_s - c_s < m_s;$ form the fibre of $\pi$,
	\ref{S1 not/defn}, over (the new) $\BB'$; apply \ref{cor L}
	to get blocks $X_0,...,X_s$ of cardinality $c_0, ..., c_s$
	(thus around every $\bb \in \BB$ we replace $\BB'$ by a sufficiently small Zariski neighbourhood);
	and continue the induction \ref{I1 ind hypo} in $s$. 
\end{caseA}

\begin{caseB}[Relative, \cfr \ref{case B}] \label{CB relative}
	The complimentary closed set $\BB''$, \ie $g_s > H$, is non-empty, then at $\bb \in \BB''$ apply \ref{cor L} to get a Zariski neighbourhood of $\bb$, in $\BB''$, on which there are blocks $X_0,..., X_{s-1}$ such that after taking the fibre of $\pi$ over this open the sub-inductive hypothesis \ref{SI1 s-induction} holds at the successor of $H$.
\end{caseB}

In so much as this procedure now involves multiple base changes to the initial set up \ref{S1 not/defn}, we can usefully observe that
if case (B), \ref{CB relative}, does not occur at $\bb \in \BB$ \emph{ ad infinitum} then a posteriori we can simply replace $\BB$ in \ref{S1 not/defn} by a Zariski open neighbourhood of $\bb$ and drop the precision of restricting to an open neighbourhood of $\bb$ in case (A), \ref{CA relative}. 
Necessarily we also want to be able to do this should case (B), \ref{CB relative}, occur \emph{ad infinitum}, and this requires a little more care, to wit:

\begin{fact}\label{F4 fact}
	Suppose the hypothesis of the sub-induction \ref{SI1 s-induction} and let
	\begin{tikzcd}[cramped, sep=small]
		\BB^{\bullet} \ar[r, hook]
			&\BB
	\end{tikzcd}
	be the set of parameters where $g_{s} \gqs H$ for all $H \in \Theta_{s-1}$ then
		
		(i)
		$\BB^{\bullet}$ is closed.
		
		(ii)
		Every $\bb \in \BB^{\bullet}$ admits a Zariski open neighbourhood $\BB \supset \VV_{\bb} \ni \bb$ such that on replacing $\BB$ by $\VV_{\bb}$ in \ref{S1 not/defn} the precision of shrinking to an open neighbourhood of $\bb$ at every instance of case (B), \ref{CB relative}, as $H$ varies in $\Theta_{s-1}$, may be omitted.
		
		(iii)
		After base change of $\pi$ to the constructible set $\BB \cap \VV_b \ni \bb$ the blocks $X_0,..., X_{s-1}$ converge in the $M$-adic topology.
\end{fact}

\begin{proof}
	We have already proved in \ref{C1 corollary} that for any given $H$, $g_{s} \gqs H$ is a closed condition so not only is $\BB^{\bullet}$ closed, it is equal to $g_{s} \gqs h$ for $h$ sufficiently large. As such by base change we may suppose, without loss of generality, that $\BB^{\bullet} = \BB$ and case (A), \ref{CA relative}, never occurs. Now the reason why we may have to restrict to an open neighbourhood of $\bb$ is, in the notation of \ref{F3 fact} that the rank $m_s$ $k$-modules,
	\begin{equation}\label{W0000 eq}
		\De(H) := 
		\left\lbrace
			\partial \in \Dl 	\, \big| \,
			\partial	\bigl(
							V_{s-1}^{d}(H) 
						\bigr) = 0
		\right\rbrace \subset \Dl 
	\end{equation}
	may not be trivial. On the other hand for any $H$ we have a surjection,
	\begin{equation}\label{6 eq}
		\begin{tikzcd}
		\quoziente{M}{M^2} \ar[r, two heads]
			&\quoziente{F_{s-1}^{1}(H)}{F_{s-1}^{>1}(H)}
		\end{tikzcd}
	\end{equation}
	whose kernel (generated by the blocks $X_i$, $0 \lqs i \lqs s-1$) is by construction, \eqref{eq W-D}, independent of $H$. Consequently the quotient \eqref{6 eq} is a vector bundle independent of $H$, but by the better still in \ref{lem L}, $\De(H)$ is naturally isomorphic to its dual should 
	 case (A), \ref{CA relative}, never occur, so we get 
	 \ref{F4 fact}.(ii) by \ref{cor L}.
	Once this is established, (iii) is exactly as in the absolute case \eqref{eq mad2} - \eqref{eq mad3}.
\qedhere \end{proof}

\begin{definition/fact}\label{D2 defn}
	In the set up of  \ref{S1 not/defn} define the relative invariant,
	\begin{equation}
	\begin{tikzcd}
		\invUB(I): \BB \ar[r]
			&\Q_{\gqs 0}^{2m}	
	\end{tikzcd}	
	\end{equation}
	starting from the rules (S.0) \& (S.1) of \ref{start INDUCT} albeit with $d_{\bb}, \; \lambda_0(\bb)$ as defined in \ref{F1 fact} \& \ref{F2 fact}. 
	Subsequently if at $\bb \in \BB$ in the inductive procedure in $s$, every sub-induction terminates at a finite $H$ (\ie case (A), \ref{CA relative}), then define
	\begin{equation}\label{EQ1}
		\invUB(I)(\bb) := 
		\bigl(
			d(\bb),		\; 
			\lambda_0(\bb), \; ...,\;
			\lambda_{s-1}(\bb),		\; 
			g_{s}(\bb),		\; 
			\underline{0} 	\,
		\bigr)	\in \Q_{\gqs 0}^{2m};
	\end{equation}
	where $s$ is minimal for the property $\lambda_{s}(\bb)=0$. Finally if  case (B), \ref{CB relative}, occurs \emph{ad infinitum} at some $s \gqs 1$ put,
	\begin{equation}\label{EQ2}
		\invUB(I)(\bb) := 
		\bigl(
		d(\bb),		\; 
		\lambda_0(\bb), \; ...,\; 
		g_{s-1}(\bb),		\; 
		\lambda_{s-1}(\bb),		\;	
		\underline{0} 	\,
		\bigr)	\in \Q_{\gqs 0}^{2m}.
	\end{equation}
	Consequently for $m(\bb)$ as in \ref{R1 remark}, 
	$\defc = m - m_{\bb}$,
	and 
	$ 
		\bigl(					\,
			g_0= d ,			\; 
			\ell_0 ,		 	\; ...,	\;
			\ell_{t} ,			\; 
			g_{t} ,				\; 
			\underline{0}	 	\,
		\bigr)
	$
	the value of the invariant, $\inv_{\BB_{\bb}}(I_b)$ of \S\ref{sec II INV}, with $t$ minimal amongst even entries $\ell_{2i}$ such that $\ell_{2i}=0$, is
	\begin{equation}\label{CP422 eq}
		\invUB (I)(\bb) := 
		\begin{cases}
		\bigl(
		g_0, \ell_0 + \defc, ..., g_t, \ell_t + \defc, \underline{0}
		\bigr),
		& \text{if } g_t \neq 0\; , \\
		\bigl(
		g_0, \ell_0 + \defc, ..., \ell_{t-1} + \defc, \underline{0}
		\bigr),
		& \text{if } g_{t} = 0, \, t \gqs 1\; , \\
		\qquad \qquad \underline{0}
		& \text{if } t=0, \text{ and } g_0=0 .
		\end{cases}
	\end{equation} 
\end{definition/fact}
We have already encountered a similar difference in 
\ref{EF1 lemma}.(ii) and whence the difference merits a specific notation, to wit:
%
\begin{equation}\label{518bis}
\begin{split}
\DIF(\defc):=
&\begin{cases}
\bigl(	\, 0,
\, \defc,
\, ..., 0,
\,\underbrace{\defc}_{t^{\th}\text{-place}},
\, \underline{0}
\,	
\bigr), 
&\text{if } g_{t} \neq 0, 	t \gqs 1,\\
\bigl(	\, 0,
\, \defc,
\, ..., 0,
\,\underbrace{\defc}_{(t-1)^{\th}\text{-place}},
\, \underline{0}
\,	
\bigr), 
&\text{if } g_{t} = 0, 		t \gqs 1,\\
\bigl(	\,0,
\,...,
\,0,
\,\underline{0} 
\, 
\bigr),
&\text{ if } t=0, g_t=0,
\end{cases}
\end{split}
\end{equation}
Plainly the difference, \eqref{518bis}, between the invariants is minimal, but it is the relative invariant that has the good properties one would expect, for example:

\begin{fact}\label{WF1 fact}
	Let 
	\begin{tikzcd}[cramped, sep=small]
		\invUB: \BB 		\ar[r]
			&\Q_{\gqs 0}^{2m}
	\end{tikzcd}
	be as per \ref{D2 defn}, then
		
		(i) 
		As a function of formal neighbourhoods $\UU$, ideals on the same, and points on the base, $\invUB$ has self bounding denominators in the sense of \ref{MD1 defn}.
		
		(ii)
		The function $\invUB$ is upper semi-continuous in the Zariski topology.
\end{fact}

The proof will require some topological trivialities, to wit:

\begin{lemma}\label{lem Elem}
	Let $X$ be a topological space,
	\[
	\begin{tikzcd}
	\uF := F_1 \times F_2 \,: X \ar[r]
			&\Z_{\gqs 0}^{n_1} \times \Z_{\gqs 0}^{n_2}
		\end{tikzcd}
	\]
	a function and equip each $\Z_{\gqs 0}^{n_i}$, respectively the aforesaid product, with the
	the lexicographic order then for $ \uf := f_1 \times f_2 \in \Z_{\gqs 0}^{n_1} \times \Z_{\gqs 0}^{n_2}$, 
	the set $X_{\gqs \uf}$, of those $x \in X$ such that $\uF(x) \gqs \uf$,
	is closed if the followings hold:
		
		(i)	
		$F_1$ is upper semi-continuous on $Y_0:=X$;
			
		(ii)	
			${Y}_{1}' := 
			\big\lbrace
				x \in Y_1 \, \big| \; \uF_2(x) \geq \uf_2 \,
			\big\rbrace
			$ 
			is closed in the constructible set 
		 $Y_1 := \lbrace x \in Y_0 \, \big| \, F_1(x) = f_1 \rbrace $.
	\end{lemma}

\begin{proof}
	By item (i) $Y_1$ is an open subset of 
	${Y}:= 
	\lbrace 
		x \in X \, \big | \, F_1(x) \geq f_1 
	\rbrace$,
	so $Y_1$ is constructible.
	Now, by construction 
	\begin{equation}\label{eq L1}
	X_{\gqs \uf} 	= {Y}_{1}' 	\cup 
	\left\lbrace 
		x \in X \, \big | \, F_1(x) > f_1 
	\right\rbrace 	= {Y}_{1}' 	\cup 
	\bigl( 
		{Y} \setminus Y_1 
	\bigr) \subseteq {Y}, 
	\end{equation}
	where the latter is closed in $X$, so it is sufficient to prove that 
	${Y}_{1}' \cup 
	\bigl( 
		{Y} \setminus Y_1 
	\bigr)$ is closed in ${Y}$. However its closure in $Y$ is
	\begin{equation}\label{eq L2}
		\overline{{Y}_{1}'} \cup \bigl( {Y} \setminus Y_1 \bigr)  =
		\big( \overline{{Y}_{1}'} \cap Y_1 \big)\cup \bigl( {Y} \setminus Y_1 \bigr) = 	{Y}_{1}' \cup \bigl( {Y} \setminus Y_1 \bigr),
	\end{equation}
	where $\big( \overline{{Y}_{1}'} \cap Y_1 \big) = {Y}_{1}'$ by item (ii), and we conclude.	
	\qedhere \end{proof}
We will apply this in the form:
\begin{corollary}\label{cor Elem}
	Let $X$ be a topological space, $F_i: X \longrightarrow \Z_{\gqs 0}^{n_i}$ functions, respectively $f_i \in \Z_{\gqs 0}^{n_i}$, for $n_i \in \Z_{> 0}$, $ 1 \lqs i \lqs N$, 
	such that if $N > r \gqs 0$, with $Y_{r} := \lbrace x \in X \; \big| \; F_i(x) = f_i ,\, 1 \lqs i \lqs r \rbrace $, $ Y_{0}:=X$, and for all $0 \lqs t \lqs r$ the function ${F}_{t+1}$ is \usc on the set $Y_t$, then $Y_{r}$ is constructible while
	\[
	\begin{tikzcd}
	\uF_{r+1}:= (F_1,...,F_{r+1}): X \ar[r]
		&Z_{\gqs 0}^{n_1+ ... + n_{r+1}} \quad \text{ is u.s.c. }
	\end{tikzcd}
	\]
	\end{corollary}

\begin{proof}
	By induction on $r \in \Z_{\gqs 0}$, with the case $r=0$ being trivial. As such let $ r \gqs 1$, and suppose the proposition for $r-1$, then we may apply \ref{lem Elem} to
	\begin{equation}\label{eq L3}
		\begin{tikzcd}
		\uF_r \times F_{r+1}: X \ar[r]
			&\Z_{\gqs 0}^{n_1+...+n_r} \times \Z_{\gqs 0}^{n_{r+1}}.
		\end{tikzcd}
	\end{equation}
	to conclude by induction.
\qedhere \end{proof}

\begin{proof}[Proof of \ref{WF1 fact}]
The difference between $\INV$ and $\inv$ is given by \eqref{CP422 eq}, so in particular their difference is integer valued, thus self bounding denominators for $\inv$, \ref{MF2}, implies self bounding denominators for $\INV$ while the pre-requisites for deducing the \usc by way of \ref{cor Elem}
	have already been done in \ref{F1 fact}, \ref{F2 fact}, \ref{F3 fact} and \ref{C1 corollary}.
\qedhere \end{proof}

The particular case of \ref{WF0 fact}.(ii) where $\UU$ is a formal neighbourhood of the diagonal in an algebraic variety, \cfr \ref{WC1 constru}, suggests that the upper semi-continuity will demand a modified invariant, to wit:

\begin{definition}\label{EED1 def}
	Let $I$ be an ideal of a regular ring $A$ of characteristic zero with dimension $m$; $x \in \spec A$, while $A_{x}, \; I_x$ denote localisation of $A,\;I$ at $x$; $\defc_x := \dim A - \dim A_x $; and, \cfr \eqref{EQ2} \emph{et seq.}, $(g_0=d, \ell_0, ..., \ell_{t-1}, g_{t}, \underline{0})$ the value of the invariant $\inv_{A_x}(I_x)$ of \S \ref{sec II INV} wherein $t$ is minimal among entries $\ell_{2i}$ with $\ell_{2i}=0$, then the difference $\inv_A^{!}(I)(x) - \inv_{A_x}(I_x)$ is defined to be $\DIF(\defc_x)$ of \eqref{518bis}.
\end{definition}

While it is premature to assert the upper semi-continuity of $\inv_A^{!}$ we do have,
\begin{fact}\label{EEF1 fact}
	Let everything be as in \ref{EED1 def} and $y \in \spec A$, then the set,
	\[
		\bigl\lbrace
			b \in \bar{y} 
		\, \big| \, 
			\inv_{A}^{!}(I)(b) 	= 
			\inv_{A}^{!}(I)(y) 
		\bigr\rbrace
	\]
	contains a non-empty Zariski open subset of $\bar{y}$.
\end{fact}

Unsurprisingly the key point in which we will abuse notation slightly in order to emphasise its relation to the preceeding definitions is:

\begin{claim}\label{EEC1 claim}
	Let everything be as in \ref{EEF1 fact}, then there is an affine neighbourhood $V := \spec A' \ni y$ such that if $\what{A}'$ is the completion of $A'$ in $y$, and $A_{\{y\}}$ the completion of $A_y$ in the maximal ideal then there is a regular weighted filtration $F^p(I)$ of $\what{A}$ such that if $F_{\{y\}}^p$ is the filtration of $A_{\{y\}}$ of \eqref{eq WF0 F} associated to the pair $I_y,\,A_y$ then,
	\[
		F_{\{y\}}^p := F^p(I) \what{\otimes}_{\what{A}'} A_{\{y\}} 
	\] 
	Better still not only is the multiplicity $d$ of $I$ constant along $\bar{y} \cap V$, but if $a^{0}$ is the highest weight amongst the blocks of the filtration and,
	\[
		V_d:= \biggl( \frac{I+ F^{>da^{0}}(I)}{F^{>da^{0}}(I)} \biggl)
		\otimes_{\what{A}'} \, 
		k,
	\]
	where $k:= \quotient{A'}{\bar{y}}$, then $V_d$ enjoys the non-degeneracy condition \eqref{HP eq} of \ref{HP} for the associated $\PP_k(\ua)$. 
\end{claim}

\begin{proof}
	For obvious reason we don't worry about the difference between $A'$ and $A$ and simply understand $\spec A \ni y$ to be a sufficiently small Zariski neighbourhood of the same. Similarly we put $k = \quotient{A}{\bar{y}}$, and, of course suppose that for $M$ the maximal ideal of $\bar{y}$ in $A$, $\quotient{M}{M^2}$ is the trivial $k$-module. Now while we may not be in the hypothesis \ref{S1 not/defn}, \ie $k$ may not embed in $A$, \eqref{W00 eq} \& \eqref{W000 eq} apply as stated to deduce that $\lambda_i(b)$, and whence implicitly the $g_i(b)$, are constant for $b$ in a Zariski open subset of $\bar{y}$. A priori there remains the substantive difference between the Zariski localisation $\bar{A}_y$ and the formal localisation $A_{\{y\}}$ but this has been addressed in \eqref{W0000 eq} \emph{et seq.} in the proof of 
	\ref{F4 fact}.(iii)
\qedhere \end{proof}

We are now in position to give,

\begin{proof}[Proof of fact \ref{EEF1 fact}]
	We may without loss of generality suppose that the conclusion of \ref{EEC1 claim} holds. Consequently, on replacing $\spec A$ by $\spec A'$, it will sufficient to show the stranger statement that $\inv^{!}$ is constant.
	As such let $b \in \bar{y}$, and $A_{b}$, resp. $A_{\{b\}}$ the localisation, resp. formal localisation at $b$, then by \ref{WF0 fact} we have the identities,
	\begin{equation}\label{1eq}
		\inv_{A_{y}} (I_y) = \inv_{A_{\{y\}}} (I_{\{y\}}), \quad 
		\inv_{A_{b}} (I_b) = \inv_{A_{\{b\}}} (I_{\{b\}}).
	\end{equation}
	Furthermore since $A_{\{b\}}$ and $k_{\{b\}}$ are complete local rings we have (non canonically) a splitting,
	\begin{equation}\label{2eq}
		\begin{tikzcd}
		\UU = \mathrm{Spf} A_{\{b\}}
		\ar[d, "\pi"']
		\\
		B = \spec k_{\{b\}},
		\ar[u, bend right, "\sigma"']
		\end{tikzcd}
	\end{equation}
	in which the trace of $\sigma$ is the pull-back of $\bar{y}$.
	As such if $K$ is the quotient field of $k_{\{b\}}$, then the fibre (\emph{qua} formal scheme) $\UU_K$ is the (formal) base change
	\begin{equation}\label{bc1eq}
	\begin{tikzcd}
		A_{\{b\}} \ar[r, mapsto] 
			&A_{\{b\}} \what{\otimes}_{k(b)} K = A_{\{y\}} .		
	\end{tikzcd}
	\end{equation}
	Now (unsurprisingly) \ref{WF2 fact} the relative invariant is stable under base change, so from \eqref{1eq}, \eqref{bc1eq}, and \ref{D2 defn} we have,
	\begin{equation}\label{3eq}
		\inv_{A_y} (I_y) = \invUB (I)(K)
	\end{equation}
	while by the better still in \ref{EEC1 claim}, we have,
	\begin{equation}\label{3.biseq}
		\invUB (I)(K) = \invUB(I)(b).
	\end{equation}
	On the other hand by \ref{D2 defn} the latter is the invariant of $I$ restricted to the special fibre $\UU_{b}$ in \eqref{2eq} which itself is a product $\UU_{b} \times_{k(b)} B$ so we have a second projection $\pr$ to $\UU_{k(b)}$ and a second ideal $\pr^{*} \restr{I}{\UU_{b}}$, with
		\[
		\invUB 	\bigl( 
					\pr^{*} \restr{I}{\UU_{b}}
				\bigr)
		=
		\inv_{A_y}\bigl( I_y \bigr)
		\]
	by \eqref{3eq}, \eqref{3.biseq}, and the base change formula \ref{WF2 fact}, while \ref{EF1 lemma}.(ii), the difference 
	between \linebreak[4] $\inv_{A_{\{b\}}}(\pr^{*} \restr{I}{\UU_{b}})$ and $\invUB(\pr^{*} \restr{I}{\UU_{b}})$ is $\DIF(\defc)$ of \eqref{518bis}, with $\defc = \dim A_{b} - \dim A_{y}$. As such it will suffice to prove,
	\[
		\inv_{A_{\{b\}}}	\bigl(	I
							\bigr) 
		= 
		\inv_{A_{\{b\}}} 	\bigl( \pr^{*} \restr{I}{\UU_{b}}
							\bigr) 
	\]
	which in the presence of the better still in \ref{EEC1 claim}, follows by induction in $s$ in the definition of the invariant \ref{start INDUCT} - \ref{indu hypo}.
\qedhere \end{proof}
An equally useful property is stability under base change, \ie

\begin{fact}\label{WF2 fact}
	Let 
	\begin{tikzcd}[cramped, sep=small]
		\beta : \BB' \ar[r]	
				&\BB
	\end{tikzcd}
	be a map of schemes, and 
	\begin{tikzcd}[cramped, sep=small]
	\pi : \UU' \ar[r]	
	&\BB'
	\end{tikzcd}
	the base change of $\pi$, \ref{S1 not/defn}, \emph{qua} formal scheme with $I'$ the pull-back of $I$ then,
	\[
	\INV_{\quotient{\UU'}{\BB'}} = \beta^{*} \invUB
	\]
\end{fact}

\begin{proof}
	By way of notation let $M'$ be the pull-back of $M$, then the condition $I \subset M^e$ plainly implies $I' \subset (M')^{e}$. At which point we just need to check that the conditions that the dimension of the modules (since the odd entries of $\INV$ are determined by whether this is maximal or not) \eqref{W00 eq} \& \eqref{W000 eq} are stable under base change which is indeed the case since tensor products are right exact.
\qedhere \end{proof}

Of which a particularly pertinent corollary is,
\begin{corollary}\label{EC1 cor}
	Let 
	\begin{tikzcd}[cramped, sep=small]
	\pi : U \ar[r] &B
	\end{tikzcd}
	be a regular map of characteristic zero affine Noetherian schemes in which both $\Gamma(U)$ and $\Gamma(B)$ are regular Noetherian local rings, then if $I$ is an ideal on $B$, $J = \pi^{*}I$ and $\defc = \dim \Gamma(U) - \dim\Gamma(B)$ the invariants $\inv_{\Gamma(U)}$, $\inv_{\Gamma(B)}$ enjoy exactly the same relation as enunciated in items (i)-(ii) 
	of the particular case \ref{EF1 lemma} ($U = \spec B$, $B = \spec A$ in the notation of \op) \ie their difference is $\DIF(\defc)$ of \eqref{518bis}. 
	Better still if there exists a filtration $F^{\bullet}(I)$ on $B$ whose completion is \eqref{eq WF0 F}, then the same is true on $U$ and,
	\begin{equation} \label{e3}
	F^{\bullet}(\pi^{*}I) = \pi^{*}F^{\bullet}(I) .
	\end{equation}
\end{corollary}

\begin{proof}
	By \ref{WF0 fact} we may suppose that $\Gamma(B)$ is a complete local ring, so, \emph{inter alia} it has a coefficient field isomorphic to its residue field $k(b)$, and by \ref{lem W-D}, the invariant $\inv_{\Gamma(B)}(I)$ is equally the relative invariant for,
	\begin{equation}\label{e3 bis eq}
	\begin{tikzcd}
	\mathrm{Spf}\, \Gamma(B) \ar[d] \\
	\spec  k(b).
	\end{tikzcd}
	\end{equation}
	On the other hand if $u$ is the closed point of $U$, then we can, \ref{WF2 fact}, base extend \eqref{e3} to 
	\begin{tikzcd}[cramped, sep=small]
	k(u) \ar[r] &\Gamma(B) \otimes_{k(b)} k(u)
	\end{tikzcd}
	without changing the invariant. 
	At the same time we can complete $\Gamma(U)$ in either $u$ or $\pi^{*}b$, and since completion in $u$ is equally completion in $\pi^{*}b$ subsequently completed in $u$, neither operation changes the invariant. We may thus suppose that $\Gamma(B)$ and $\Gamma(U)$ are complete local rings with the same residue field, and since we are in characteristic zero $\pi$ remains regular (otherwise we'd need to suppose geometrically regular). In particular therefore, $\Gamma(U)$ is a power series ring over $\Gamma(B)$ and \ref{EF1 lemma} applies to give the relation between the invariants while \eqref{e3} follows from \ref{EF1 lemma} and the fact that completion is faithfully flat.
	\qedhere \end{proof}

Notice that \emph{en passant }we have proved
\begin{fact}\label{dim1 fact}
	Let
	\begin{tikzcd}[cramped, sep=small]
		\pi: U 	\ar[r]	&B
	\end{tikzcd}		
	be a regular map of regular affine schemes of dimension $N$, $n$ respectively and $I$ a sheaf of ideals on $B$ then,
	\begin{equation}\label{dim1eq}
		\inv_{\Gamma(U)}^{!} (\pi^{*}I) 	- 
		\inv_{\Gamma(B)}^{!}(I) 			= 
			\DIF(N-n).
	\end{equation}
\end{fact}

\begin{proof}
	From the proof of \ref{EC1 cor} for any $u \in U$ over $b$,
	\[
		\inv_{\OO_{U,u}}(\pi^{*}I_u)	-
		\inv_{\OO_{B,b}}(I_b)	=
			\DIF 
			\bigl(
				\dim\OO_{U,u} - \dim\OO_{B,b}
			\bigr)
	\]
	while by definition \ref{EED1 def},
	\[
	\inv_{\Gamma(B)}^{!}	(I) 	-
	\inv_{\OO_{B,b}}		(I_b)	=
		\DIF 
		\bigl(
			n - \dim\OO_{B,b}
		\bigr)
	\]
	and similar for $U$, whence \eqref{dim1eq} by the additivity of $\DIF$.
\end{proof}

\section{Principalisation}\label{sec V APP}
\setcounter{equation}{0}
To begin with let us make

\begin{observation/defn}\label{WD1 obs/defn}
	Let $\CU$ be a regular Noetherian Deligne-Mumford, or indeed formal champ, of characteristic $0$, and $\cg$ a sheaf of ideals on $\CU$ then for
	\begin{tikzcd}[cramped, sep=small]
		x: \spec K \ar[r]
			&\CU
	\end{tikzcd}
	a geometric point (\ie $K$ is algebraically closed) the invariant
	$\inv_{\CU} (\cg)(x)$ is defined to be $\inv_{\OO_{\CU,x}}( \cg_{x})$ where $\cg_{x}$ is the stalk of $\cg$ in the strictly Henselian ring $\OO_{\CU,x}$. In particular therefore by \ref{WF0 fact} if,
	\begin{center}
		\begin{tikzcd}
		\spec K				\ar[r] \ar[rd, "x"']
		& \v		\ar[d]
		\\
		&\CU
		\end{tikzcd}
	\end{center}
	is any factorisation through an affine étale neighbourhood, with $y$ the image on $\v$ then,
	\begin{equation}\label{WD11 eq}
		\inv_{\CU}(\cg)(x) = \inv_{\OO_{\v,y}}(\cg_{y}),
	\end{equation}
\end{observation/defn}
and we will vary this construction in the obvious way for the variants
$\inv^{!}$, resp. $\inv^{\sharp}$. With this in mind we have the key,
\begin{fact}\label{WF3 fact}
	Let $\UU =\mathrm{Spf}A$ be the formal spectrum of a complete regular ring of characteristic zero, $I$ an ideal of $A$, $F^{\bullet}(I)$ as in \ref{WF0 fact} and 
	\begin{tikzcd}[cramped, sep=small]
		\rho: \wti{\UU} \ar[r]
			&\UU
	\end{tikzcd}
	the smoothed weighted blow up \cite[I.iv.3]{mp1} associated to the aforesaid weighted filtration, then for $\wti{I}$ the proper transform of $I$,
	at every closed geometric point $x$ of $\wti{\UU}$,
		\[
			\inv_{\wti{\UU}}(\wti{I})(x) < \inv_{\UU}(I).
		\]
\end{fact}

\begin{proof}
	Upon clearing denominators the blocks of the filtration have weights $a^{i}$, and we have a
	$\big(\quoziente{A}{F^{>0}}\big)$-module,
	\[
		\bar{I} := I \mod F^{>da^0}
	\]
	such that if $\II$ is the resulting sheaf of ideals on the
	associated weighted projective champ, equivalently the exceptional divisor 
	\begin{tikzcd}[cramped, sep=small]
		\exe \ar[r, hook] 
			&\wti{\UU},
	\end{tikzcd}
	then,
	\[
		\restr{\wti{I}}{\exe} = \II.
	\]
	Consequently we can conclude by \ref{prop INV gdown}
	provided that
	\[
		\inv_{\exe}			\bigl( \,
								\wti{I} 	\, 
								\big|_{\exe}
							\, \bigr)	
							(x) \geq 
		\inv_{\wti{\UU}}	\bigl( \,
								\wti{I} 
							\, \bigr)
							(x)
	\]
	at closed geometric points $x$. As far as the odd entries of the invariant are concerned, \cfr the proof of \ref{prop INV gdown}, this is clear. There is, however, need for caution at the even entries which is provided by items (ii-bis) \& (iii-bis) of \ref{lem sub-champ}, which are satisfied for the inclusion
	\begin{tikzcd}[cramped, sep=small]
		\exe \ar[r, hook] &\UU
	\end{tikzcd}, \ie replace $\QQ$ by $\exe$ in \op and the values of $c_i$ on the ambient space by their value on $\UU$.
\qedhere \end{proof}

Plainly, therefore, the algebraicity or otherwise of the filtration
$F^{\bullet}$ of $\what{A}$ of \ref{WF0 fact} should \ref{ad infinitum}
occurs is the only obstruction to constructing a resolution of singularities for local rings from the invariant, and to address this problem we will proceed from varieties over a field to spectra of complete local rings by way of a particular instance of the relative invariant, to wit:
\begin{construction}\label{WC1 constru}
	Let $\quotient{\v}{K}$ be a smooth affine scheme of dimension $m$ over a field $K$ of characteristic $0$ and let $\CP_{V/{K}}^{n}$	be the sheaf of $n$-jets of \cite[16.7]{egaIV3}  then, for any map
	\begin{tikzcd}[cramped, sep=small]
		\tau : T \ar[r] &\v
	\end{tikzcd}
	from a scheme $T$, we have a formal scheme equipped with a projection,
	\begin{equation}\label{W2 eq}
		\begin{tikzcd}
		\v 
			&&\GP_{T} := 	\mathrm{Spf} 
							\left(
								\lim\limits_{\xleftarrow[\; n \;]{\null}} \tau^{*}\CP^n_{V/K}
							\right)
			\ar[ll, "\pr"] \ar[dd, "\pi"]
			\\
			\\
			&&T \ar[uu, bend right, "\sigma"']
		\end{tikzcd}
	\end{equation}
	whose trace is a regularly embedded section $\sigma$ - in fact $\GP_T$ is the completion of the graph of $\tau$. In particular if $T=\v$ and $I$ is an ideal on $\v$ then we have 
	\begin{equation}\label{new1}
		\INV_{\GP_{\v}/\v}(\pr^{*}I)(x) = \inv_{\Gamma(\v)}^{!}(I)(x)
	\end{equation}
	where the latter is the invariant of \ref{EED1 def}, so, to reiterate, their difference with the invariant $\inv_{\v}(I)(x)$ of \eqref{eq invIA} is $\DIF(\defc)$ of \eqref{518bis}, where
	\begin{equation}
		\defc = \mathrm{Tr}\deg_K K(x) = \dim\v - \dim \OO_{\v,x}.
	\end{equation}
\end{construction}

In light of \ref{WF1 fact} \& \ref{WF2 fact}, we therefore make,

\begin{fact/defn}\label{WFD1 f/d}
	Let everything be as in \ref{WC1 constru}, then for 
	\begin{tikzcd}[cramped, sep=small]
		\tau : T \ar[r] &\v
	\end{tikzcd}
	a map from a Noetherian scheme $T$, we define,
	\begin{equation}\label{W2bis eq}
		\begin{tikzcd}[row sep=4pt]
			\inv_{T}^{!}(I):T	 	
			\ar[r]
				&\Q_{\gqs 0}^{2m} 
				\; : \; 	
					t
					\ar[r, mapsto]
						&\INV_{\GP_T/T} (\pr^{*}I)(t)
		\end{tikzcd}
	\end{equation}
	so by \ref{WF2 fact} $\inv_{T}^{!}(I)$ is \usc$\,$and equal to
	$\tau^{*} \inv_{\v}^{!}(I) = \inv_{\Gamma(\v)}^{!}(I)$ of \ref{EED1 def}.
\end{fact/defn}

The next step is complete local rings and weighted blow ups of their spectra, by way of,

\begin{variant}[of \ref{WC1 constru}]\label{WC2 constru}
	Let $W$ be the spectrum of a complete local regular Noetherian ring (\eg the completion around a not necessarily closed point of $\v$ of \ref{WC1 constru}) then by 
	\cite[28.3]{commringmatsu} 
	we can choose a coefficient field $K$ and coordinates $x_1,...,x_m$ such that
	\begin{tikzcd}[cramped, sep=small]
		W \ar[r, "\sim"] &\spec K \llbracket x_1, ..., x_m \rrbracket.
	\end{tikzcd} Now consider the diagram of rings,
	\begin{equation}\label{W3 eq}
		\begin{tikzcd}
		K\llbracket y_1, ..., y_m \rrbracket 
				\ar[r]
				&K\llbracket x_1, ..., x_m, y_1, ..., y_m \rrbracket 
				\\
				&K\llbracket x_1, ..., x_m \rrbracket 
				\ar[u]
		\end{tikzcd}
	\end{equation}
	so, for example, in the situation that
	\begin{tikzcd}[cramped, sep=small]
		W \ar[r] &\v 
	\end{tikzcd}
	arises from completing $\v$ of \ref{WC1 constru} in a $K$-point, the ring in the top right hand corner is $\Gamma \bigl( \GP_W \bigr)$ in the notation of \eqref{W2 eq}. We are, however, at liberty to apply the functor $\spec$ to \eqref{W3 eq} to get,
	\begin{equation}\label{W4 eq}
		\begin{tikzcd}
						&&&\GP_W 	\ar[dl, hook', "i"']	
									\ar[dddl, bend left, "\pi_W"]
		\\
			W							
				&&P 			 	\ar[ll, "\pr"']	
									\ar[dd, "\pi"]		
		\\
		\\
				&&W					\ar[uu, "\sigma=\Delta", bend left]
		\end{tikzcd}
	\end{equation}
	wherein the distinctions with \eqref{W2 eq} in the case that $W$ comes from $\v$ are: 
	\begin{equation}\label{W4bis eq}
		\begin{split}
		\text{(a)}\; 
		&\text{$\GP_W$ is the completion of $P$ in the diagonal 
			$
			(x_i - y_i | 1 \lqs i \lqs m ).
			$
			} 
		\\
		\text{(b)}\; 
		&\text{Nevertheless both the projections of in \eqref{W2 eq} and \eqref{W4 eq} are projections}
		\\
		&\text{to schemes  with $\pr$ of \eqref{W4 eq} the continuous extension of $\pr$ of \eqref{W2 eq}.}
		\end{split}
	\end{equation}
\end{variant}

Now in the first instance we can profit from these observations to extend the definition of $\inv^{!}$ to ideals $\JJ$ of $W$, \ie exactly as in \eqref{W2bis eq} but for 
\begin{tikzcd}
	T \ar[r] &W,
\end{tikzcd}
\begin{equation}\label{W5 eq}
	\inv_{T}^{!}(\JJ) := \INV_{\GP_T/T}\bigl( i^{*}\pr^{*} \JJ \bigr)
\end{equation}
with for $i,\,\pr$ as in \eqref{W4 eq}, and, of course, this is compatible 
with \ref{EED1 def}, resp. \eqref{W2bis eq} 
by \ref{WF2 fact} if 
$T=W$, resp.
$\JJ$ were pulled back from $\v$. Consequently \ref{WF1 fact} applies and $\inv^{!}$ is \usc$\,$ on $W$ irrespectively of whether $\JJ$ is pulled back from something of finite type or not. The risk, however, is that we lose the possibility of having a good construction of the relative invariant, and whence the \usc$\,$ once we start making weighted blow ups of $W$. To get around this suppose a weighted centre with blocks $X_0,...,X_s$ and weights
$a_0,...,a_s$ 
is given to which we add (a possibly empty) block $Y$ to obtain a system of coordinates on $W$ and identify $\Gamma(W)$ with the completion in the origin of the ring of functions on,
\[
\begin{tikzcd}
	A:= \spec K[X_0,...,X_s,Y] \ar[r, "\sim"]
		&\A_K^{m}
\end{tikzcd}
\] 

Now form the smoothed weighted blow up 
\begin{tikzcd}[cramped, sep=small]
\rho: \buA \ar[r] &A
\end{tikzcd}
in the blocks $X_i$ with weights $a_i$ (so $\wti{\UU}$ of \ref{WF3 fact} would be the completion of $\buA$ in the exceptional divisor) to get a diagram,
\begin{equation}\label{W6 eq}
	\begin{tikzcd}[row sep=1.5cm, column sep=1cm]
	\buA
	&\GP_{\buA}:=\mathrm{Spf} 	
		\bigg(
			\lim\limits_{\xleftarrow[\; n \;]{\null}}
			\CP^n_{\buA/K}
		\bigg)
	\ar[l, "\pr"] \ar[d, "\pi"']
	\\
	&\buA			
	\ar[u, bend right, "\sigma = \text{ zero-section}"']
	\end{tikzcd}
\end{equation}
wherein it goes without saying that, even though the diagonal fails to be an embedding, the jets of \cite[16.7]{egaIV3} are well defined on any Deligne-Mumford champ because of their étale local nature. Finally observe that the smoothed weighted blow up 
\begin{tikzcd}[cramped, sep=small]
\rho: \buW \ar[r] 	&W
\end{tikzcd}
in the said blocks is just the base change,
\begin{equation}\label{W6bis eq}
	\begin{tikzcd}[row sep=1.5cm, column sep=1.5cm]
	\buA
	\ar[d, "\rho"']		
		&\buW
		\ar[d, "\rho"]
		\ar[l]
	\\
	A
		&W
		\ar[l]
	\end{tikzcd}
\end{equation}
so that on base changing the projection $\pi$ of \eqref{W6 eq} we get a diagram,
\begin{equation}\label{W7 eq}
	\begin{tikzcd}[row sep=2.5cm, column sep=2.5cm]
		\buW
		\ar[d, "\rho"']
			&\GP_{\buW}
			\ar[l, "\pr"]
			\ar[ld, "\pr"]
			\ar[d, "\pi_{\buW}"']
		\\
		\buA
			&\buW
			\ar[u, bend right, "\sigma"']
	\end{tikzcd}
\end{equation}
wherein the existence of the horizontal arrow results from the horizontal arrow in \ref{W4 eq}, equivalently \ref{W4bis eq}.(a).

At this juncture, in line with \ref{WD1 obs/defn}, the theory of the relative invariant applies étale locally to the (formally) representable map $\pi_{\buW}$ and we conclude,

\begin{fact/defn}\label{WFD2 f/d}
	Let $\II$ be a sheaf of ideals on the weighted smoothed blow up $\rho: \buW \xrightarrow{\quad} W$ of \eqref{W6bis eq} and define at a geometric (but not necessarily closed) point $w: \mathrm{pt} \xrightarrow{\qquad} \buW$,
	\begin{equation}\label{W8 eq}
	\inv_{\buW}^{!}\big(\II\big)(w) := \inv_{\GP_{\buW}/\buW} \big( \pr^{*} \II \big)(w)
	\end{equation}
	then in the moduli $|\buW|$ of $\buW$ (\ie the projectivisation of the graded algebra which defines the weighted rather than smoothed weighted blow up) $\inv_{\buW}^{!}$ is \usc.
	\end{fact/defn}

Better still the discussion also reveals that we can make numerous improvements to \ref{WF3 fact} to wit:

\begin{corollary}\label{WC0 cor}
	Let $A$ be a complete regular ring of characteristic zero, 
	$W = \spec A$, $I \neq A$ an ideal of $A$, $F^{\bullet}(I)$ the filtration of \ref{WF0 fact}, 
	\begin{tikzcd}[cramped, sep=small]
		\buW \ar[r] &W
	\end{tikzcd}
	the associated smoothed weighted blow up of \cite[I.iv.3]{mp1} and $x$ a closed point, then:
		
		(i)
		For all $\pri \in W$, \[\inv_{W}^{!}(I)(\pri) \lqs \inv_{W}^{!}(I)(x).\]
		
		(ii)
		If $\wti{I}$ is the proper transform of $I$ on $\buW$ then for all geometric point $w$ of $\buW$,
		\[ \inv_{\buW}^{!}(\wti{I})(w) < \inv_{W}^{!}(I)(x).\]
		
		(iii)
		If $X_0,...,X_s$ are the blocks defining the filtration $F^{\bullet}(I)$ then the sub-scheme $X_0=...=X_s=0$ is exactly
		\begin{equation}\label{W9 eq}		
			\what{Z} = 	\left\lbrace \,
						\pri \in W \, \big| \, 
						\inv_{W}^{!}(I)(\pri) = \inv_{\buW}^{!}(I)(x)
						\, \right\rbrace
		\end{equation}
\end{corollary}

\begin{proof}
	Since $A$ is a local ring, \ref{WC0 cor}.(i) is just the \usc of $\inv^{!}$ in \ref{WFD1 f/d} as extended to arbitrary ideals of $A$ in \eqref{W5 eq} \emph{et seq. }. Similarly we already know \ref{WC0 cor}.(ii) after completing the exceptional divisor 
	\begin{tikzcd}[cramped, sep=small]
		\cexe \ar[r, hook] &\buW
	\end{tikzcd}
	by \ref{WF3 fact} and since $\inv_{\buW}^{!}(\wti{I})$ is also \usc in the Zariski topology of the moduli by \ref{WFD2 f/d}, we have it everywhere since $\rho$ is proper. Consequently if in \ref{WC0 cor}.(iii) $\what{Z}$ were not contained in $X_0=...=X_s=0$, then by the \usc of \ref{WFD2 f/d} we would have the absurdity that the invariant would not go down. Conversely the inclusion of
	$X_0=...=X_s=0$ in $\what{Z}$ is essentially automatic from \ref{F/P madic},  case (B), \ref{CB relative}, and the lower semi-continuity of the $c_i$.
\qedhere \end{proof}

All of which can be combined to establish,

\begin{fact}\label{EE1 fact}
	Let $A$ be an excellent regular local ring of characteristic zero, $\what{A}$ its completion in the maximal ideal then for every ideal $I$ of $A$ there exists a filtration $F^{\bullet}(I)$ whose completion is the filtration $F^{\bullet}(\what{I})$ of $\what{A}$ afforded, \ref{WF0 fact}, by $\what{I}= \what{A} \otimes_A I$.
\end{fact}

\begin{proof}
	Put $V = \spec A$, $W= \spec \what{A}$, 
	\begin{tikzcd}[cramped, column sep=12pt]
	R:= W \times_{V} W 
	\ar[r, shift left, close, "t"]
	\ar[r, shift right, close, "s"']
	&W
	\end{tikzcd} the resulting groupoid,
	\begin{tikzcd}[cramped, sep=small]
		\rho: \buW \ar[r]	&W	
	\end{tikzcd} 
	the smoothed weighted blow up of \ref{WC0 cor} associated to $F^{p}(\what{I})$.
	Now the fibre of $R$ over a point, which in turn is cut out by the pull-back to $R$ of any system of coordinates on V, is a point 
thus although $R$ may have many connected components their dimension is at most that of $\v$, which is equally that of $W$, and only one has this maximal value. Furthermore by hypothesis
	\begin{tikzcd}[cramped, sep=small]
		W \ar[r] &\v
	\end{tikzcd} 
	is regular so this is equally true of the source $s$ and sink $t$, so by \ref{dim1 fact},
	\begin{equation}\label{fix3eq1}
		s^{*}\inv_{W}^{!}(\what{I}) 	= 
		\inv_{R}^{!}(s^{*}\what{I}) 	=
		\inv_{R}^{!}(t^{*}\what{I})		=
		t^{*}\inv_{W}^{!}(\what{I}).
	\end{equation}
	Similarly by \ref{EF1 lemma}, we have that $s^{*} F^p(\what{I})$, resp. $t^{*} F^p(\what{I})$, are (after completion) the filtration of \ref{WF0 fact} at every point where 
	$ s^{*}\inv_{W}^{!}(\what{I}) $, resp.
	$ t^{*}\inv_{W}^{!}(\what{I}) $, is maximal. Consequently by \eqref{fix3eq1}, and of course, as implied therein, 
	$s^{*} \what{I} = t^{*} \what{I} = \restr{\what{I}}{R}$,
	$s^{*} F^p(\what{I})$ and $t^{*} F^p(\what{I})$ are the filtration of $\restr{I}{R}$ at every point where 
	$\inv_{R}^{!} 	\bigl( 
						\restr{I}{R}
					\bigr)$
	is maximal. Thus we have a canonical isomorphism between $s^{*}\buW$ and $t^{*}\buW$ which is uniquely determined by its value (\eg the identity is standard birational parlance) where $\rho$ is an isomorphism, whence 
	a descent datum for $F^p(\what{I})$ with respect to the faithfully flat map 
	\begin{tikzcd}[cramped, sep=small]
		W \ar[r] &V
	\end{tikzcd}
	 so we're done by \cite[exposé VIII, 1.1]{sga1}.
\qedhere \end{proof}

An alternative in the geometric case is to appeal directly to the relative invariant in a formal neighbourhood of the diagonal, to wit:

\begin{alternative}\label{WA4 alternative}
	Let $x$ be a not necessarily closed point of a smooth variety $\v/K$ over a field of characteristic zero, $I$ an ideal of $\v$ and $F^{\bullet}(\what{I})$ the canonical filtration, \ref{WF0 fact},
	of the completion of $\OO_{\v,x}$ in its maximal ideal determined by $I$, then $F^{\bullet}(\what{I})$ is algebraic, \ie
	\ref{EE1 fact} holds for $A=\OO_{\v, x}$.
\end{alternative}

\begin{proof}
	By way of notation put $\v_{\zar} = \spec \OO_{\v,x}$, $W= \spec \what{\OO}_{\v,x}$, and 
	\begin{tikzcd}[cramped, sep=small]
		\what{Z} \ar[r, hook] &W
	\end{tikzcd}
	as in \eqref{W9 eq} then from the compatibility of \ref{WFD1 f/d} and \eqref{W5 eq}, the sub-scheme $\what{Z}$ is by \ref{WFD1 f/d} the pre-image under
	\begin{tikzcd}[cramped, sep=small]
		W \ar[r] &\v
	\end{tikzcd}
	of,
	\[
		Z := \left\lbrace \,
			\pri \in \v_{\zar} \, \big| \,
			\inv_{\v}^{!}(I)(\pri) \gqs \inv_{\v}^{!}(I)(x)
		\, \right\rbrace
	\]
	so, as the notation suggests, if $I_Z$ is the ideal of the sub-scheme then,
	\[
		I_{\what{Z}} = I_Z \otimes_{\OO_{\v, x}} \what{\OO}_{\v,x}.
	\]
	It remains to find the blocks themselves rather than just the centre, $Z$, on which they are supported. To do this it is sufficient to do 
	\ref{F/P madic} 
	$I_Z$-adically rather than $\mm(x)$-adically. If, however, we denote by the subscript \emph{ét} strict Henselisation at $x$, then in $\v_{\et}$ we can choose a projection $\pi$ and a section $\sigma$,
	\[
	\begin{tikzcd}[row sep=1.3cm, column sep=1.3cm]
		\v_{\et} 	\ar[r, "\pi"']
		&Z_{\et}	\ar[l, bend right, "\sigma"', hook]	
	\end{tikzcd}
	\]
	such that $\sigma$ is the embedding of 
	\begin{tikzcd}
		Z_{\et} \ar[r, hook, "\sigma"] &\v_{et},
	\end{tikzcd}
	so $I_Z$-adic convergence of the blocks follows from 
	\ref{F4 fact}.(iii). As such we have \ref{EE1 fact} but for a filtration $F^{\bullet}_{\et}(I)$ of $\OO_{\v_{\et}}$.
	Now make $\v$ sufficiently small in the Zariski topology such that,
	\begin{equation}\label{W10bis enumerate}
	\begin{split}
	\text{(a)}\;
	&\text{$\inv_{\v}^{!}(I)(x)$ is the maximum of the invariant over $\v$.} 
	\\
	\text{(b)}\; 
	&\text{There is a filtration $F^{\bullet}_{\et}(I)$ on a geometrically irreducible étale } 
	\\
	&\text{neighbourhood $\v' \xrightarrow{\quad} \v$ of $x$
	satisfying 
	\ref{EE1 fact} after completing } \\
	&\text{in a point $x' \in \v'$ over $x$.}
	\end{split}
	\end{equation}
	In particular, therefore, the support of the graded algebra $\gr_{F_{\et}}^{\bullet}$ is the fibre $Z'$ over $Z$, which in turn is the locus where $\inv_{\v'}(I)$ is maximal. Similarly if we consider the groupoid,
	\begin{equation}\label{W10bis eq}
	\begin{tikzcd}
		R:= \v' \times_{\v} \v'
			\ar[r, shift left, "t"]
			\ar[r, shift right, "s"']
		&\v'
	\end{tikzcd}
	\end{equation}
	$s^{*}F_{\et}$ and $t^{*}F_{\et}$ define the canonical filtration at every point in $s^{*} Z' =t^{*} Z'$, equivalently the locus where $\inv_{R}(I)$ is maximal, so by \ref{WF0 fact} they are equal, and whence \ref{EE1 fact} in the Zariski topology.
\qedhere \end{proof}

In the geometric case, we already have upper semi- continuity of the invariant in \ref{WFD1 f/d}, while in general we appeal to:

\begin{fact}\label{EEF2 fact}
	Let $I$ be an ideal in an excellent regular ring $A$, then on $\spec A$, the function
$x\mapsto\inv_{A}^{!}(I)(x)$	
	of \ref{EED1 def} is \usc on $\spec A$. 
\end{fact}
The strategy follows Villamayor's exposition \cite[6.13]{vill1} of Dade's unpublished Princeton thesis, in the case of the multiplicity,
from which we plagiarise,
\begin{claim}\label{EEC2 claim}
	Let \begin{tikzcd}[cramped, sep=small]
		f: \spec A \ar[r] &\Gamma 
	\end{tikzcd} 
	be a function to a discrete ordered group then $f$ is \usc iff,
	
		(i)
		$\forall\, y \in \spec A, \; x \in \bar{y} \Longrightarrow f(x) \gqs f(y)$.
		
		(ii) 
		$\forall\, y \in \spec A, $ the set $
		\lbrace  
			x \in \bar{y} 
			\big|  
			f(x) \lqs f(y) 
			\rbrace$ 
		contains a non-empty Zariski open subset.
\end{claim}

\begin{proof}
	Plainly given \ref{EEC2 claim}.(i) we can replace inequality by equality  in \ref{EEC2 claim}.(ii), while the conditions are clearly necessary, and we do the converse by induction on dimension of closed sub-spaces, $Y$, which without loss of generality are irreducible of positive dimension. However by \ref{EEC2 claim}.(ii), $Y=Y' \amalg Z$ where $f$ takes the value of its generic point, $f(y)$, on $Y'$. Furthermore by \ref{EEC2 claim}.(i) $\restr{f}{Z} \gqs f(y)$ and
	\begin{tikzcd}[cramped, sep=small]
		Z \ar[r, hook] &Y
	\end{tikzcd}
	is a closed subset of smaller dimension, so $f(x) > f(y)$ is closed and everything left over takes the value $f(y)$.
\qedhere \end{proof}

At which juncture we may proceed to,

\begin{proof}[Proof of fact \ref{EEF2 fact}]
	Since the invariant has self bounding denominators, \ref{MD1 defn}-\ref{MF1 fact}, we plainly need only verify items \ref{EEC2 claim}.(i)-(ii) with $f$ replaced by $inv_A^{!}$. Now (ii) is \ref{EEF1 fact} and is valid for any regular ring while (i) follows from \ref{EC1 cor} and the upper semi-continuity of $\inv^{!}$ for complete local rings, \eqref{W5 eq} \emph{et seq.}
\qedhere \end{proof}

At which point we can move rapidly towards a conclusion by way of,

\begin{fact/defn}\label{WFD3 f/d}
	Let $\v$ be a regular excellent affine scheme of dimension $m$, $I$ an ideal of $\v$;
	$\ui \in \Q_{\gqs 0}^{2m}$ the maximum value of $\inv_{\v}^{!}(I)$ over $\v$; 
	\begin{tikzcd}[cramped, sep=small]
		\mathcal{V} \ar[r] &\v
	\end{tikzcd} 
	the smoothed weighted blow up (whose existence and uniqueness is guaranteed by \ref{EE1 fact} and \ref{WF0 fact} associated to the canonical filtration $F^{\bullet}(I)$; while for $\uq \in \Q_{\gqs 0}^{2m}$ define a modification functor
	\begin{equation}
		M_{I,\uq}(\v) :=
		\begin{cases}
			\mathcal{V},	&\text{if } \ui = \uq, \\
			\v,				&\text{otherwise;}
		\end{cases} 
	\end{equation}
	and extend $M_{I,\uq}$ to arbitrary smooth affine $\v$, \ie a disjoint union of connected components $\coprod_{\alpha}\v_{\alpha}$ by way of,
	\begin{equation}\label{DS1 eq}
		M_{I,\uq}(\v) := \coprod_{\alpha} M_{I,\uq}(\v_{\alpha}).
	\end{equation}
\end{fact/defn}

Then by \ref{WF0 fact} the modification functor $M_{I,\uq}$ is étale local, \ie if 
\begin{tikzcd}[cramped, sep=small]
	\v' \ar[r] &\v
\end{tikzcd}
is étale and $I'$ the pull-back of $I$ to $\v$ then there is a fibre square,
\begin{equation}\label{W11 eq}
	\begin{tikzcd}[row sep=0.5cm, column sep=0.5cm]
		M_{I,\uq}(\v) 
		\ar[dd]	
			&&M_{I',\uq}(\v')
			\ar[dd]	\ar[ll]
	\\
	& \qedsymbol &
	\\
		\v
			&&\v'
			\ar[ll]
	\end{tikzcd}
\end{equation}

In particular if $\XX$ is a regular excellent Deligne-Mumford champ,
$\cg$ a sheaf of ideals on the same and $\uq$ the maximum at geometric points of $\inv_{\XX}^{!}(\cg)$, \eqref{WD11 eq} \emph{et seq.}, then for 
\begin{tikzcd}[cramped, sep=small]
	\v \ar[r] &\XX
\end{tikzcd}
an étale atlas and
\begin{tikzcd}[cramped, column sep=19pt]
	R = \v \times_{\XX} \v	
		\ar[r, shift left=2pt, close, "t" ]
		\ar[r, shift right=2pt, close, "s"']
		& V
\end{tikzcd}
the implied groupoid,
\begin{equation}\label{W11bis eq}
	\begin{tikzcd}[row sep=0.5cm, column sep=0.5cm]
		M_{\cg_{R}, \uq}(R) 
			\ar[rr, shift right, "s"']
			\ar[rr, shift left, "t"]
			\ar[dd]
				&&M_{\cg_{\v},\uq}(\v)
				\ar[dd]
	\\	\\
		R
		\ar[rr, shift right, "s"']
		\ar[rr, shift left, "t"]
			&&\v	
	\end{tikzcd}
\end{equation}
is a map of groupoids in which $M_{\cg_R, \uq}(R)$ (which we may abusively
consider  unique since it's  a modification) is equally the fibre of the rightmost vertical arrow over either $s$ or $t$ by \eqref{W11 eq}, \ie the $M_{\cg_{\v}, \uq}$ patch to a smoothed weighted blow up,
\begin{equation}\label{W11bisbis eq}
	\begin{tikzcd}[column sep=1.3cm]
	M_{\cg}(\XX) \ar[r] &\XX
	\end{tikzcd}
\end{equation}
depending only on $\cg$. We therefore get our first global results, to wit:

\begin{proposition}\label{WP1 prop}
	Let $\cg$ be a (coherent) sheaf of ideals on a regular excellent Deligne-Mumford champ of characteristic zero, and define inductively a sequence of smoothed weighted blow ups in regular weighted centres by,
	\[
		(\XX_0, \cg_0) := 
		(\XX, \cg) \quad \text{and} \quad  
		(\XX_{p+1}, \cg_{p+1}) := 
		(M_{\cg_{p}}\XX_{p}, \wti{\cg}_{p}),\;
			p > 0	
	\]
	where $\wti{\cg_{p}}$ is the proper transform of $\cg_{p}$, then for $p \gg 0$, $\cg_p$ is trivial. In particular if $\cg$ is the sheaf of ideals of an irreducible embedded sub-champ, $N+1$ the smallest $p$ such that $\cg_p$ is trivial and
	\begin{tikzcd}[cramped, sep=small]
		\YY_p 	\ar[r, hook]	&\XX_p
	\end{tikzcd}
	the sub-champ cut out by $\cg_p$ then 
	if at every closed point $\XX$ has the same dimension the chain,
	\begin{equation}\label{chain eq}
		\begin{tikzcd}
		\YY_0
				&\YY_1
				\ar[l]
						& \cdots
						\ar[l]
								&\YY_{N-1}
								\ar[l]
										&\YY_{N}
										\ar[l]
		\end{tikzcd}
	\end{equation}
	is a sequence of smoothed weighted blow ups in regular centres $Z_p$ contained in the singular locus of $\YY_p$, $p < N$, such that $\YY_{N}$ is regular.
	Otherwise, \ie the dimension of $\XX$ is not constant on closed points, the same conclusion \eqref{chain eq} holds provided for each $p$ one changes the invariant to
	\begin{equation}
		\wti{\inv}_{\XX_p}(\cg_{p})(x) =
		\begin{cases}
		\inv_{\XX}^{!}(\cg_{p})(x), &x \not\in \YY_p
		\\
		\inv_{\XX}^{!}(\cg_{p})(x) + \DIF\bigl(\dim \YY - d_p(x)\bigr)
		\end{cases}
	\end{equation}
	where $d_p(x)$ is the dimension of the connected component of $\YY_p \ni x$; and then blow up in strata where $\wti{\inv}$ rather than $\inv^{!}$ is maximal.
\end{proposition}

\begin{proof}
	Since $\inv^{!}$ goes down, \ref{WC0 cor}.(ii), under the modification
	\begin{tikzcd}[cramped, sep=small]
		M_{\cg}\XX	\ar[r]	&\XX
	\end{tikzcd}
	while preserving excellence and the invariant has self bounding denominators, \ref{WF1 fact}.(i), the only thing to check is the in particular which in turn only requires checking that $Z_p$ is contained in $\mathrm{Sing}(\YY_p),\; p < N.$
	 If, however, there were a geometric, so without loss of generality closed, point $y \in Z_p$ where $\YY_p$ was regular then the value of $\wti{\inv}_{\XX}(\cg_p)$ at $y$ would be,
	 \[
	 	\bigl(
	 		\,	1,\, 	\dim(\YY),\, 	\underline{0}\,
	 	\bigr)
	 	\in 	\Q_{\gqs 0}^{2m}
	 \]
	 and since this is equally the minimum value of $\wti{\inv}_{\XX}(\cg_p)$ , $\YY_p$ would be regular contradicting the choice of $N$.
	\qedhere \end{proof}

Arguably the good way to think about \ref{WP1 prop} is in terms of resolving rational maps, which merits:

\begin{remark}\label{EER1 rmk}
	If in \ref{WP1 prop} $\XX$ were a projective variety, $X/K$, over a field $K$ and $\LB$ an ample line bundle then given a sheaf of ideals $\cg$ there is a $n \gg 0$, such that $\LB^n \otimes \cg$ is generated by global sections, and whence $\cg$ is the indeterminacy locus of rational map
	\begin{equation}\label{4eq}
	\begin{tikzcd}
		\varphi_{\cg}: X 	
		\ar[r,dashed] 
				&\PP 	\bigl(
						\HH^{0} ( X, \LB^n \otimes \cg)					
						\bigr)
	\end{tikzcd}
	\end{equation}
	while, since $X$ is, by hypothesis smooth, every rational map
	\begin{tikzcd}[cramped, sep=small]
		\varphi: X \ar[r, dashed] & \PP_{K}^{N}
	\end{tikzcd}
	determines a unique line bundle, $L_{\varphi}$, which is equal to $\varphi^{*} \OO(1)$ in codimension 2, and a space of sections,
	\[
		\varphi^{*}
		\HH^{0} 	\bigl(
					\PP_K^N, \OO(1)
					\bigr)
		\subset
		\HH^{0}		\bigl(
					X, L
					\bigr)
	\]
	which generates the indeterminacy locus, \ie $\cg_{\varphi}L_{\varphi}$, for some sheaf of ideals $\cg_{\varphi}$. Of course this relation between ideals and rational maps may fail even for $X/K$ a scheme of finite type, albeit it suffices to replace \eqref{4eq} by
	\begin{tikzcd}[cramped, sep=small]
		\varphi: \XX \ar[r, dashed] &\mathrm{Bl}_{\cg}\,X
	\end{tikzcd}
	to maintain it in absolute generality.
	In any case the relationship between ideals and rational maps is rather tight, so we can equally think about the modification functor $M_{\cg}\XX$ as a modification functor $M_{\varphi} \XX$ for $\varphi$ a rational map, so that \ref{WP1 prop} becomes the rather satisfactory statement:

	Let $\varphi$ be a rational map on a Deligne-Mumford champ $\XX$, and define inductively a sequence of rational maps by,
	\[
		\bigl(\XX_{0}, \varphi_{0} \bigr) 
		=
		\bigl(\XX, \varphi  \bigr) 
		\qquad
		\bigl(\XX_{p+1}, \varphi_{p+1} \bigr)
		=
		\bigl(M_{\varphi_p}\XX_{p}, \wti{\varphi}_{p} \bigr)
	\]
	where $\wti{\varphi}_p$ is the proper transform of $\varphi$, then for $p \gg 0$, $\varphi_p$ is everywhere defined.
\end{remark}

\section{Excellent Resolution}
\setcounter{equation}{0}
Of course the in particular in \eqref{chain eq} gives a resolution of singularities of anything admitting an embedding in something regular, but this is not a very satisfactory hypothesis so we improve it by way of

\begin{construction}\label{WFD4}
	Let $Y$ be a connected reduced excellent affine scheme of dimension $n$, $y \in Y$ a not necessarily closed point, and $\hO_{Y,y}$ the completion of $\OO_{Y,y}$ in the maximal ideal.
	Now choose a coefficient field $L$ and a presentation,
	\begin{equation}\label{W12 eq}
		\begin{tikzcd}
			0
			\ar[r]
					&I
					\ar[r]
							&A:= L\llbracket S_1,...,S_e \rrbracket
							\ar[r]
									&\hO_{Y,y}
									\ar[r]
											&0 
		\end{tikzcd}
	\end{equation}
	where,
	\begin{equation}\label{EEL1}
		e:=e_{Y}(y) = \dim_{k(y)} \quoziente{\mm(y)}{\mm^2(y)}
	\end{equation}
	is the embedding dimension, and observe that any $2$ such presentations are related by a commutative diagram of exact sequences,
	\begin{equation}\label{W12bis eq}
		\begin{tikzcd}[row sep=1.2cm, column sep=1.2cm]
			0
			\ar[r]
					&I
					\ar[r]
							&A:= L\llbracket S_1,...,S_e \rrbracket
							\ar[r]
									&\hO_{Y,y}
									\ar[r]
											&0
		\\
			0
			\ar[r]
					&I_0
					\ar[r]
					\ar[u]
							&A_0:= L_0\llbracket T_1,...,T_e \rrbracket
							\ar[r]
							\ar[u, "{\sim}" labl]
									&\hO_{Y,y}
									\ar[r]
									\ar[u, equal]
											&0\;.
		\end{tikzcd}
	\end{equation}
	
As such $\inv_{Y}(y):= \inv_{A}(I)$ is well defined, and for $m$ the maximum over all embedding dimensions we correct this to
\begin{equation}\label{D100eq}
	\inv_{Y}^{!}(y):= 	\bigl( 
							\inv_{Y}(y) + \DIF(m - e_{Y}(y))
						\bigr)
						\times \underline{0} \in \Q^{2m}
\end{equation} with an implies block of zeroes whenever $e_{Y}(y) < m$. 

At the same time in the complete local ring $\what{\OO}_{Y,y},$ or better, and equivalently since $Y$ is excellent, in the strict Henselisation $\OO_{Y,y}^{\hh}$, we have,   
\[
	d_{Y}(y) 	:= 	\min\limits_{\mathfrak{q}} 		
						\dim\frac{\what{\OO}_{Y,y}}{\mathfrak{q}}
				=
					\min\limits_{\mathfrak{q}} 	
						\dim\frac{{\OO}_{Y,y}^{\hh}}{\mathfrak{q}}
\]
where the minimum is taken over all the minimal primes in $\what{\OO}_{Y,y}$, or equivalently ${\OO}_{Y,y}^{\hh}$, which in turn affords the invariant,
	\begin{equation}\label{W12bisbis eq}
		\inv_{Y}^{\sharp}(y) 	:= 
			\bigl( 
				\delta_{Y}(y) 	:= e_{Y}(y) - d_{Y}(y) 
			\bigr) 
			\times \inv_{Y}^{!}(y) \; \in \Q_{\gqs 0}^{2m + 1} 	
	\end{equation} 
Similarly the smoothed weighted blow up,
	\begin{equation}\label{W13 eq}
		\begin{tikzcd}
		\what{\YY} \ar[r] & \spec \hO_{Y,y}
		\end{tikzcd}
	\end{equation}
associated to the canonical filtration $F^{\bullet}(I)$ of $\spec A$ is
not only independent of the presentation, but,
\end{construction}

\begin{fact}\label{exEF1 fact}
	If 
	\begin{tikzcd}[cramped, sep=small]
		\OO_{Y,y} \ar[r] &B
	\end{tikzcd}
	is a regular map of local rings and 
	\begin{tikzcd}[cramped, sep=small]
		\buW \ar[r] &\spec \what{B}
	\end{tikzcd}
	the result of performing \eqref{W12 eq}-\eqref{W13 eq} for $\what{B}$ then $\buW$ is the pull-back along
	\begin{tikzcd}[cramped, sep=small]
		\hO_{Y,y} \ar[r]  &\what{B}
	\end{tikzcd}
	of the modification \eqref{W13 eq}.
\end{fact}

\begin{proof}
	Since the map is regular then at worst after a base extension of $L$ we may, as in the proof of \ref{EC1 cor}, suppose that $\what{B} = \hO_{Y,y} 	\llbracket 
							z_1,...,z_{e}
							\rrbracket $
	is a formal power series ring with coefficients in $\hO_{Y,y}$, so this follows by either  \ref{EC1 cor} or \ref{EF1 lemma}.
\qedhere \end{proof}

Nevertheless to make everything fit together in this generality we need to descend $\what{\YY}$ of \ref{W12bisbis eq} to a modification of $\spec \OO_{Y,y}$ and establish the upper semi-continuity of $\inv^{\sharp}$.
The latter is somewhat involved for arbitrary excellent rings so it seems useful to observe that the geometric case is rather trivial, to wit:
\begin{alternative}\label{EEA1}
	Let $Y/K$ be a reduced affine scheme of finite type over a field of characteristic zero with $\OO_{Y,y}^{\hh}$ the strict Henselisation at some point $y \in Y$ then
	at the minor price of base changing, by a separable extension of $L$, we can suppose the presentation \eqref{W12 eq} arises from an étale neighbourhood 
	\begin{tikzcd}[cramped, sep=small]
	Y'	\ar[r]	&Y
	\end{tikzcd}
	around a point 
	\begin{tikzcd}[cramped, sep=small]
		y' \xmapsto{\quad} y 
	\end{tikzcd}
	with $\quotient{K(y')}{K(y)}$ étale, \ie there is an embedding
	\begin{tikzcd}[cramped, sep=small]
	Y' \ar[r, hook]	&\v'
	\end{tikzcd}
	into a smooth $K$-scheme of dimension $e + \mathrm{Tr}\deg_{K}K(y) -n$ such that after completion in $y'$,
	\begin{equation}\label{W14 eq}
	\begin{tikzcd}
	0
	\ar[r]
			&I'
			\ar[r]
					&\Gamma(\v')
					\ar[r]
							&\OO_{Y',y'}
							\ar[r]
										&0
	\end{tikzcd}
	\end{equation}
	becomes \eqref{W12 eq} upon applying $K(y') \otimes_{L} -$. In particular after replacing $Y'$ and $Y$ by appropriately small affine neighbourhoods of themselves we recognise that \eqref{W13 eq} is algebraic indeed after base changing to $K(y')$ it is the formal fibre of the proper transform,
	\begin{equation}\label{W15 eq}
	\begin{tikzcd}
	\YY'	\ar[r]	&Y',
	\end{tikzcd}
	\end{equation}
	of $Y$ in the canonical modification 
	\begin{tikzcd}[cramped, sep=small]
	\mathcal{V}		\ar[r] 		&\v 
	\end{tikzcd}
	of \ref{WFD3 f/d} associated to $I'$.
\end{alternative}	


Now observe that the leading term $\delta_Y$ in $\inv_Y^{\sharp}$ is just,
	\begin{equation}\label{W16 eq}
	\dim_{k(y)} 
		\Omega_{Y/K} \otimes k(y) 	- 
	\dim \bar{y} 					- 
	\min_{Y_0} \dim \OO_{Y_0,y}^{\hh}
	\end{equation}
	where the minimum is taken over all the components $Y_0$ of the Henselian local ring at $y$, so \eqref{W16 eq} is equally,
	\begin{equation}\label{D22eq}
		\dim_{k(y)} 
			\Omega_{Y/K} \otimes k(y) 	- 
		\min_{Y_0} \dim Y_0
	\end{equation}
	where now the minimum is taken over components $Y_0 \ni y$ on étale neighbourhood of $y$. Consequently $\delta_Y$ is the difference of an upper semi-continuous function and a lower semi-continuous one so $\delta_{Y}$ is \usc. To conclude from here that $\inv_{Y}^{\sharp}$ is \usc we require by \ref{lem Elem} to establish that $\inv_{Y}^{!}$ is \usc where $\delta_{Y}$ is constant. To this end say $ \delta_{Y}(x)=\delta_{Y}(z)$, then we may as well say that we're on an étale neighbourhood $Y'$ of a constructible set $Z$, with generic point $z$, where $\Omega_{Y/K}$ has constant rank and around $Z$ we have an embedding 
	\begin{tikzcd}[cramped, sep=small]
		Y' \ar[r, hook] 	&M
	\end{tikzcd}
	into a smooth $K$-variety of dimension $e_{Y}(x)$ for $x$ any geometric point of $Z$. Consequently for $x \in Z$, closed or otherwise,
	\begin{equation}\label{D23eq}
		\inv_{Y}^{!}(x) = 	\inv_{M}^{!}(I_{Y'}) 	+ 
							\DIF(m - \dim M)		
	\end{equation}
	so it's upper semi continuous by \ref{WFD1 f/d}.
	We now proceed to the general case via,

\begin{claim}\label{EEC3 claim}
	Let everything be as in \ref{WFD4} then $\inv_Y^{\sharp}$ satisfies Dade's conditions \ref{EEC2 claim}. More precisely, for $z \in Y$:
	
	(i)
		if $\OO_{Y,x}$ is just a characteristic zero local ring and $f$ is either $\delta_{Y}$ or $\inv_Y^{\sharp}$ then for $\bar{z} \ni x$, $f(x) \gqs f(z)$.

		(ii)	
		If $Y$ is universally catenary and J-2, \ie every reduced closed sub-scheme 
			\begin{tikzcd}[cramped, sep=small]
				Z \ar[r, hook] 	&Y
			\end{tikzcd}
			contains a non-empty Zariski open subset where it is regular then there is a non-empty Zariski open subset of points $x \in \bar{z}$ where $\delta_{Y}(x) \lqs \delta_{Y}(z)$. 

		(iii)
		If $Y$ is excellent then there is a non-empty Zariski open subset of points $x \in \bar{z}$ where $\inv_{Y}^{\sharp}(x) \lqs \inv_{Y}^{\sharp}(z)$.

		In particular if $Y$ is excellent then both $\delta_{Y}$ and $\inv_{Y}^{\sharp}$ are \usc . 
\end{claim}

\begin{proof}
	Consider first the behaviour of $\delta_Y$ in \ref{EEC3 claim}.(i), and observe,
	\cite{germans}, that there is a good theory of the universal finitely generated module, $\Omega_{\what{Y}/k(x)}$, of $k(x)$-derivations, and:
	\begin{equation}\label{D24eq}
		e 	= e_{Y}(x) 
			= \dim_{k(x)} 
				\Omega_{\what{Y}/k(x)} \otimes k(x)
			= \dim \Omega_{X/k(x)} \otimes k(x)
	\end{equation}
	where 
	\begin{tikzcd}[cramped, sep=small]
		\what{Y}	
		\ar[r, hook ]	
				&X = 
				\spec k(x) \llbracket x_1, ..., x_e \rrbracket
	\end{tikzcd}
	is an embedding, afforded by \eqref{W12bis eq}, with ideal $I_Y$.
	Similarly for $\zeta \in \what{Y}$ any point lying over $z$ and 
	$Z = \bar{z}$ with $\what{Z}$ the formal fibre we have an exact sequence,
	\begin{equation}\label{D25eq}
		\begin{tikzcd}[column sep=20pt]
		0
		\ar[r]
			&\quotient{I_Z}{I_Z^2} \otimes k(\zeta)
			\ar[r]
				&\Omega_{\what{Y}/k(x)} \otimes k(\zeta)
				\ar[r]
					&\Omega_{\what{Z}/k(x)} \otimes k(\zeta)
					\ar[r]
						&0
		\end{tikzcd}
	\end{equation}
	from which we obtain,
	\begin{equation}\label{D26eq}
		e_Y(z) + \dim_x \bar{\zeta} = \dim_{k(\zeta)} \Omega_{\what{Y}/k(x)} \otimes k(\zeta).
	\end{equation}
	Now appeal to the Henselian description of $d_Y$ as in \ref{W16 eq}-\ref{D22eq} to obtain,
	\begin{equation}\label{D27eq}
		\delta_{Y}(z) = 
			\dim_{k(\zeta)} 
				\Omega_{\what{Y}/k(x)} \otimes k(\zeta) 	- 
				\min\limits_{Y_0 \supset Z} \dim Y_0
	\end{equation}
	where the minimum is taken over components of $\OO_{Y,x}^{\hh}$ containing $Z$, so it's certainly the case that,
	\begin{equation}\label{D28eq}
		\min\limits_{Y_0 \supset Z} \dim Y_0 \gqs d_{Y}(x)
	\end{equation}
	whence $\delta_{Y}(z) \lqs \delta_{Y}(x)$ by \eqref{D27eq} and the upper semi-continuity of the fibres of $\Omega_{\what{Y}/k(x)}$. Better still we've done \ref{EEC3 claim}.(ii) unless $\delta_{Y}(z) = \delta_{Y}(x)$ which requires not only an identity in \ref{D28eq} but that $\Omega_{\what{Y}/k(x)}$ has constant rank $e=e_Y(x)$ along $\what{Z}$, so \emph{inter alia} $\dim_{x}\bar{\zeta}$ is independent of $\zeta$ from \ref{D26eq}. In any case if $\{ \_ \}$ denotes completion of a local ring in its maximal ideal then we have a presentation,
	\begin{equation}\label{D29eq}
		\begin{tikzcd}
		0
		\ar[r]
			&I_{\what{Y}}
			\ar[r]
				&\OO_{X, \{\zeta\} }
				\ar[r]
					&\OO_{\what{Y}, \{\zeta\}}
					\ar[r]
						&0
		\end{tikzcd}
	\end{equation}
	which by \ref{D26eq} is a base extension 
	\begin{tikzcd}[cramped, sep=small]
		k(z)	\ar[r]	&k(\zeta)
	\end{tikzcd}
	of a presentation of the form \eqref{W12 eq} of 
	$\OO_{Y, \{ \zeta \}}$, so by \ref{WF2 fact},
	\begin{equation}\label{D30eq}
		\inv_{Y}(z) 			= 
		\inv_{\what{Y}} 
			\bigl( \, 
				\what{\zeta} 
			\, \bigr)				=
		\inv_{X}(I_{\what{Y}})(\zeta)		.
	\end{equation}
	Consequently from the definitions \ref{EED1 def} and \eqref{D100eq},
	\begin{equation}
		\inv_{Y}^{!}(z)			-
		\inv_{Y}^{!}(x)			=
		\inv_{X}^{!}
			\bigl(
				I_{\what{Y}}
			\bigr)				(\zeta)-
		\inv_{X}^{!}
			\bigl(
				I_{\what{Y}}
			\bigr) 					(x)+
		\DIF 	
			\bigl(
				\dim \OO_{X, \zeta} - e_Y(z)
			\bigr)
	\end{equation}
	while 
	$\dim \OO_{X, \zeta} = e_{Y}(z)$ 
	from \ref{D26eq} under the hypothesis of 
	$\delta_{Y}(z) = \delta_{Y}(x).$ 
	Consequently 
	$\inv_{Y}^{!}(x) \gqs \inv_{Y}(z)$
	from the \usc of $\inv_{X}^{!}(I_{\what{Y}}),$
	\ie \eqref{W5 eq} \emph{et seq.} , which in turn completes 
	\ref{EEC3 claim}.(i) by \ref{lem Elem}.
	
	As to item (ii) we again begin with $\delta_{Y}$, and without loss of generality we may suppose every point of $\bar{z}$ is regular. Now consider the co-normal sheaf, to $\bar{z}$, \ie
	\[
		\cC := \quotient{I_{\bar{z}}}{I_{\bar{z}}^{2}},
	\]
	then, again, if we restrict to a small enough neighbourhood of $z$ we may suppose that $\cC$ is a bundle of rank $c$. Consequently for any $x \in \bar{z}$ with maximal ideal $\mm(x)$ in $Y$, and $\mm_{z}(x)$ along $\bar{z}$ we have an exact sequence
	\[
		\begin{tikzcd}
		\cC \otimes k(x)
		\rar
			&\quotient{\mm(x)}{\mm(x)^2}
			\rar
				&\quotient{\mm_{\bar{z}}(x)}{\mm_{\bar{z}}(x)^2}
				\rar
					&0						
		\end{tikzcd}
	\]
	which for $x$ arbitrary, resp. the particular choice of $x=z$ gives,
	\begin{equation}\label{DL1eq}
		e_{Y}(x) \lqs e_{\bar{z}}(x) + c, \;\text{ resp. }\;
		e_{Y}(z) = c	
	\end{equation}
	Similarly since $\OO_Y$ is supposed universally catenary,
	there's no difficulty in taking $Y \ni z$ sufficiently small such that $\forall \, x \in \bar{z}$,
	\begin{equation}\label{DL2eq}
		d_{Y}(x)= d_{Y}(z) + \dim_{x} \bar{z},
	\end{equation}
	so that from \eqref{DL1eq} and \ref{DL2eq} we have,
	\begin{equation}\label{DL3eq}
		\delta_{Y}(x) - \delta_{Y}(z) 
			\lqs 
		\bigl( 		
				e_{\bar{z}}(x) -  \dim_{x} \bar{z}
		\bigr),	
	\end{equation}
	and we've already cut things down so that the right hand side of 
	\ref{DL3eq} is zero, so we get \ref{EEC3 claim}.(ii).
	Now to complete the proof of (ii) for $\inv_{Y}^{\sharp}$ requires a fact of independent utility, to wit:
	\begin{fact}\label{EEF3 fact}
		Let
		\begin{tikzcd}[cramped,sep=small]
			\what{\YY} 	\rar 	&\what{Y}
		\end{tikzcd}
		be the modification \eqref{W13 eq}, then if $\OO_{Y,y}$ is excellent there is a smoothed weighted blow up
		\begin{tikzcd}[cramped, sep=small]
			\YY \rar &\spec \OO_{Y,y}
		\end{tikzcd}
		such that $\what{\YY}$ is the formal fibre of $\YY$.
	\end{fact}

\begin{proof}
	Just as in \ref{EE1 fact} we aim to descend $\what{\YY}$ to $V:= \spec \OO_{Y,y}$, so let
	\[
		\begin{tikzcd}[cramped, column sep=19pt]
			R = \what{Y} \times_{V} \what{Y}	
			\ar[r, shift left=2pt, close, "t" ]
			\ar[r, shift right=2pt, close, "s"']
			& \what{Y}
		\end{tikzcd}
	\]
	be the groupoid afforded by 
	\begin{tikzcd}[cramped, sep=small]
		\what{Y} 	\rar 	&V.
	\end{tikzcd}
	Now in a variant of \ref{EE1 fact} choose a system of parameters at $y$, \ie functions $x_i$, $1 \lqs i \lqs \dim V$, such that the sub-scheme, $\bullet$, $x_i=0$ for all $i$, has dimension $0$ at $y$ then the fibre of $R$ over $\bullet$ has dimension $0$ and is cut out by $\dim V$ functions, so, again all $R,\,\what{\YY}$ and $V$ have the same dimension. Furthermore since $\what{Y}$ is a local ring we know by what we've already established in \ref{EEC3 claim}.(i) that  
	wherever $\inv_{\what{Y}}^{\sharp}$ is maximal, $\delta_{Y}$ is maximal, and for $\pri \in \what{Y}$ we can,
	\eqref{D24eq}-\eqref{D30eq},
	without loss of generality suppose that the presentation employed in calculating $\inv_{\what{Y}}^{\sharp}$ is just \eqref{W12 eq} completed at $\pri$. 
	Thus by \ref{dim1 fact}, 
	$\inv_{R}^{\sharp} 		= 
	s^{*} \inv_{Y}^{\sharp} = 
	t^{*} \inv_{Y}^{\sharp},$ 
	and 	$s^{*}\what{\YY}$, 
	resp. 	$t^{*}\what{\YY}$, have by \eqref{exEF1 fact} the same formal fibre at every point where $\inv_{R}^{\sharp}$ is maximal. Consequently profiting from the fact that each contains an everywhere dense subset of $R$ we may reasonably identify them, to obtain a descent datum, and again conclude by \cite[VIII.1.1]{sga1}.
\qedhere \end{proof}

If not perhaps any easier, the geometric case offers an,

\begin{alternative}\label{EEA2 alternative}
	Consider the following statement whose argument is reduced $K$-varieties over a field of characteristic 0.
	
\textbf{S}${(Y)}: \quad$ For $\ui = \ui(Y) \in \Q_{\gqs 0}^{2m+1}$ the 	maximum value of $\inv_{Y}^{\sharp}$ and $y \in Y$ there is a Zariski open neighbourhood $N_{y} \ni y$ and a smoothed weighted blow up 
\begin{tikzcd}[cramped, sep=small]
	\mathcal{N}_y 	\ar[r] 	&N_{y}
\end{tikzcd}
whose formal fibre is \eqref{W13 eq} if 
$\inv_{Y}^{\sharp}(y)= \ui $ 
and the identity otherwise. 
In particular therefore if 
\begin{tikzcd}[cramped, sep=small]
	Z \ar[r, hook] 	&Y
\end{tikzcd}
is the locus $\inv_{Y}^{\sharp}=\ui$, then,
	
	(C.1) 
	for $y \not\in Z$ we can take $N_y= Y \setminus Z$ and S$(Y)$ is trivially true.
	
	(C.2)
	If $\mathcal{N}_y$ exists then it is, by definition, unique.
	
	(C.3)
	By \eqref{W14 eq}-\eqref{W15 eq} there is an étale atlas 
	\begin{tikzcd}[cramped, sep=small]
		U 	\ar[r] 		&Y
	\end{tikzcd}
	such that $S(U)$ is true at every point of $U$, so without loss of generality, we have a smoothed weighted blow up 
	\begin{tikzcd}[cramped, sep=small]
	\CU		\ar[r] 		&U,
	\end{tikzcd}
	which is everywhere the modification of S$(Y)$.

As such we can argue exactly as in \eqref{W10bis eq}, \ie for 
\begin{tikzcd}[cramped, column sep=19pt]
	R:= U \times_{Y} U
	\ar[r, shift left=2pt, close,"t"]
	\ar[r, shift right=2pt, close, "s"']
				&U ,
\end{tikzcd}
$s^{*}\mathcal{U}$ is canonically isomorphic (even equal since its birational) to
$t^{*}\mathcal{U}$ by item (C.2) deduced from the statement, $S(R)$ at $R$, and whence conclude S$(Y)$.
\end{alternative}

Irrespectively can apply \ref{EEF3 fact} in the spirit of \ref{EEC1 claim} to complete the proof of \ref{EEC3 claim}. Specifically throwing away irrelevant closed sets without comment: we have, without loss of generality, an everywhere regular irreducible closed subscheme 
\begin{tikzcd}[cramped,sep=small]
Z = \{ \bar{z} \} \ar[r, hook] &Y,
\end{tikzcd}
and by \ref{EEF3 fact} a smoothed weighted blow up 
\begin{tikzcd}[cramped,sep=small]
\YY \ar[r, hook] &Y	
\end{tikzcd}
whose formal fibre is \eqref{W13 eq}. Now if $x \in Z$, we may from the \usc of $\delta_{Y}$ suppose $\delta_{Y}(x)=\delta_{Y}(z)$, and all of \eqref{D29eq} \emph{et seq.} holds. 
As such if the symbol $\;\what{\bullet}\;$ denotes the spectrum of completion in $x$ (rather than the formal scheme completion) and $X = \spec A$ after a choice of the presentation \eqref{W12 eq} at $x$ 
then we have embeddings,
\begin{equation}\label{C1eq}
	\begin{tikzcd}
	\what{Z} 
	\ar[r, hook]
			&\what{Y}
			\ar[r, hook] 
					&X
	\end{tikzcd}
\end{equation}
together with the fibre
\begin{tikzcd}[cramped, sep=small]
	\what{\YY}:= \YY \times_{Y} Y \ar[r] &\what{Y}	
\end{tikzcd}
of our modification, which if it were trivial, \ie $Y$ is regular at $Z$, there is nothing to do. 
Otherwise from $\delta_{Y}(x)=\delta_{Y}(z)$ and the independence of \eqref{W13 eq} from the presentation as employed in \eqref{D30eq}, 
we may identify 
\begin{tikzcd}[cramped, sep=small]
	\what{\YY} \ar[r] &\what{Z}
\end{tikzcd}
with the modification associated to the filtration $F^p(I_{\what{Y}})$ associated to the ideal $I_{\what{Y}}$ in the completed local ring of $\what{Y}$ at any component of $\what{Z}$. Similarly, without loss of generality, we may equally suppose that the blocks $X_0,...,X_s$ of the filtration defining $\YY$ are defined on $\YY$, and their proper transform $\wti{X}_i, \; 0 \lqs i \lqs s$, cut out a decreasing chain,
\begin{equation}\label{C2eq}
	\YY=\YY_s, \qquad \YY_{t-1}= \YY_{t} \cap (\wti{X}_t =0),
\end{equation}
with completions $\what{\YY}_{t}$ at $x$. However 
by \eqref{D30eq} one recognises from \ref{EEC1 claim} that for any $\what{z} \in \what{Y}$ over $z$,
\begin{equation}
	\inv_{\what{Y}}^{!}(x) 			= 
	\inv_{\what{Y}}^{!}(z) 
			\; \Longleftrightarrow \;
	\dim \what{\YY}_t(x) 			= 
	\dim \what{\YY}_t(z), 
			\; \forall\, 
			0 \lqs t \lqs s,
\end{equation}
which in turn is equivalent to the condition $\dim \YY_t(x) = \dim \YY_t(z)$, which is certainly true on a Zariski open subset of $Z$. \qedhere
\end{proof}

We can put all of this together to conclude,

\begin{summary/defn}\label{ESD1 sum/d}
	Let $Y$ be an excellent affine scheme of characteristic zero, $\ui(Y) \in \Q_{\gqs 0}^{2m+1}$ the maximum value of $\inv_{Y}^{\sharp}$ then,

		(E.1)
			By \ref{EEF3 fact} every point $y \in Y$ has a Zariski open neighbourhood $N_y$ and a smoothed	weighted blow up
			\begin{tikzcd}[cramped, sep=small]
				\mathcal{N}_y \ar[r] &N_y
			\end{tikzcd}
			whose formal fibre is \ref{W13 eq}.
		
		(E.2)
		These patch to a smoothed weighted blow up
			\begin{tikzcd}[cramped, sep=small]
				\YY \ar[r] &Y.
				\end{tikzcd}
				Indeed in the notation \ref{EEA2 alternative}.(C.1)-(C.3) of the conclusion of the proof of \ref{EEC3 claim}, if $x \in N_z$, $x \in \bar{z}$, and $\inv_{Y}^{\sharp}(x) = \inv_{Y}^{\sharp}(z)$, then the formal fibre of $\mathcal{N}_z$ at $x$ is that of $z$.
		
		(E.3)
		Better still if for $\ui(Y) \lqs \uq $ in $\Q_{\gqs 0}^{2m+1}$ and $Y$ connected we define
		\begin{equation}\label{W17 eq}
		M_{\uq}(Y) := 
		\begin{cases}
		\YY, 	& \text{ if } \ui(Y)=\uq 
		\\
		Y, 		& \text{ if } \ui(Y)<\uq 
		\end{cases}
		\end{equation}
	and extend to direct sums of connected components as in \eqref{DS1 eq}, then (since it is enough to check the formal fibres) the smoothed weighted blow up $M_{\uq}(Y)$ commutes with étale maps, \ie if
\begin{tikzcd}[cramped, sep=small]
Y' 	\ar[r] 		&Y
\end{tikzcd}
is étale or even just regular then, by \ref{exEF1 fact} and 
\ref{ESD1 sum/d}.(E.2) we have a fibre square
\begin{equation}\label{W17bis eq}
\begin{tikzcd}[row sep=0.5cm, column sep=0.5cm]
M_{\uq}(Y) 
\ar[dd]	
&&M_{\uq}(Y')
\ar[dd]	\ar[ll]
\\
& \qedsymbol &
\\
Y
&&Y'
\ar[ll]
\end{tikzcd} \;.
\end{equation}
In particular therefore, \cfr \eqref{W11bis eq}-\eqref{W11bisbis eq},
if $\YY$ is an excellent reduced Deligne-Mumford champ \ie all Henselian local rings are excellent, and some (whence any) atlas is J2, then there is a smoothed weighted blow up,
\begin{equation}\label{W18 eq}
\begin{tikzcd}
	M_{\uq}(\YY)	\ar[r]		&\YY
\end{tikzcd}
\end{equation}
supported in the singular locus whose fibre over an étale atlas is the blow up functor \eqref{W17 eq}, and which itself commutes with étale maps, \ie replace $Y \xrightarrow{\; \; \; }Y'$ by an étale map of champ 
\begin{tikzcd}[cramped, sep=small]
\YY' 	\ar[r] 		&\YY
\end{tikzcd}
in \eqref{W17bis eq}. Finally 
for $\inv_{\YY}^{\sharp}$ defined as in \eqref{W12bisbis eq}, let $\ui(\YY)$ be the maximum value of $\inv_{\YY}^{\sharp}$ and $M(\YY):= M_{\ui(\YY)}(\YY)$ then by construction, 
\[
	\ui(\YY)=0		
			\quad \Longleftrightarrow \quad
	\YY\, \text{ is regular.} 	
			\quad \Longleftrightarrow \quad
	M(\YY)=\YY.
\]
\end{summary/defn} 

All of which is easily assembled into a resolution algorithm, to wit:

\begin{proposition}\label{WP2 prop}
	For $\YY$ a reduced excellent Deligne-Mumford champ,
	define a sequence of smoothed weighted blow ups by,
	\begin{equation}
		\YY_0 = \YY, \qquad \YY_{p+1} = M(\YY_p), \quad p \gqs 0
	\end{equation}
	and let $N \gqs 0$ be the smallest integer such that
	\begin{tikzcd}[cramped, sep=small]
		\YY_{N+1} 	\ar[r]		&\YY_{N}
	\end{tikzcd}
	is the identity then the chain of smoothed weighted blow ups,
	\begin{equation}\label{BWY eq}
		\begin{tikzcd}
		\YY = \YY_0
				&\YY_1
				\ar[l]
							&\cdots
							\ar[l]
									&\YY_{N-1}
									\ar[l]
											&\YY_{N}
											\ar[l]				
		\end{tikzcd}
	\end{equation}
	is a resolution of singularities in the 2-category of Deligne-Mumford champs enjoying the functorial resolution properties (E.1)-(E.3) of \ref{ESD1 sum/d}, but for champ rather than just affine schemes.
\end{proposition}

\begin{proof}
	From the definition \eqref{W12bisbis eq} and \ref{WF1 fact}.(i) $\inv_{\YY}^{\sharp}$ has self bounding denominators, 
	\ref{MD1 defn},
	so it suffices to check,
	\[
		\ui(M_{\YY}) < \ui(\YY)
	\]
	
	Plainly, however, the embedding dimension cannot increase under a smoothed weighted blow up and since \eqref{W13 eq} is the formal fibre around any point, this is immediate from the corresponding proposition, \ref{WC0 cor}.(ii), for $\inv^{!}$.
\qedhere \end{proof}

\bibliography{elvis}{}
\bibliographystyle{Gamsalpha}
\end{document}